\documentclass[11pt,a4paper]{article}

\renewcommand{\Im}{\operatorname{Im}}
\renewcommand{\Re}{\operatorname{Re}}

\usepackage{tikz}
\usetikzlibrary{decorations.shapes}
\usetikzlibrary{decorations.markings}
\usetikzlibrary{hobby,decorations.pathreplacing,calc,calligraphy,backgrounds}
\usetikzlibrary{patterns}
\usepackage{anyfontsize}
\usepackage{adjustbox}

\title{\bf Poisson-type problems with transmission conditions at boundaries of infinite metric trees}
\author{Maryna Kachanovska\footnote{POEMS, CNRS, Inria, ENSTA, Institut Polytechnique de Paris, 91120 Palaiseau, France, E-mail: \url{maryna.kachanovska@inria.fr}, Webpage: \url{https://sites.google.com/site/mkachanovska/}, ORCID: \url{https://orcid.org/0000-0001-9991-9874}}
	\and
	Kiyan Naderi\footnote{Technische Universit\"at Graz,  Fakultät für Mathematik, Physik und Geod\"asie, Institut f\"ur Analysis und Zahlentheorie,		
		Kopernikusgasse 24, 8010 Graz, Austria, \url{naderi@math.tugraz.at}}
	\and
	Konstantin Pankrashkin\footnote{Carl von Ossietzky Universit\"at Oldenburg, Fakult\"at V -- Mathematik und Naturwissenschaften, Institut f\"ur Mathematik, Ammerl\"ander Heerstra{\ss}e 114--118, 26129 Oldenburg, Germany,
		E-mail: \url{konstantin.pankrashkin@uol.de},
		Webpage: \url{https://uol.de/pankrashkin}, ORCID: \url{https://orcid.org/0000-0003-1700-7295}}
}
\date{}

\usepackage{amsmath}
\usepackage{amsthm}
\usepackage{amssymb}
\usepackage{amsfonts}
\usepackage{babel}
\usepackage{bbm}
\usepackage{hyperref}

\renewcommand{\Tilde}{\widetilde}

\newcommand{\Id}{\mathrm{Id}}
\newcommand{\loc}{\mathrm{loc}}
\newcommand{\comp}{\mathrm{comp}}

\DeclareMathOperator{\diam}{diam}
\DeclareMathOperator{\ran}{ran}

\DeclareMathOperator{\SL}{\mathsf{SL}}
\DeclareMathOperator{\DL}{\mathsf{DL}}
\DeclareMathOperator{\dtn}{\mathcal{C}}

\newcommand{\cT}{\mathcal{T}}  
\newcommand{\dd}{\mathrm{d}}

\newcommand{\half}{{\frac{1}{2}}}

\theoremstyle{plain}
\newtheorem{thm}{Theorem}[section]
\newtheorem{lem}[thm]{Lemma}
\newtheorem{cor}[thm]{Corollary}
\newtheorem{prop}[thm]{Proposition}

\theoremstyle{definition}
\newtheorem{defn}[thm]{Definition}
\newtheorem{prob}[thm]{Problem}

\newtheorem{rmk}[thm]{Remark}

\theoremstyle{remark}

\newcommand{\N}{\mathbb{N}}     
\newcommand{\R}{\mathbb{R}}     
\newcommand{\C}{\mathbb{C}}     
\newcommand{\T}{\cT}   
\newcommand{\D}{\mathcal{D}}   
\newcommand{\cE}{\mathcal{E}}   
\newcommand{\cG}{\mathcal{G}}   
\newcommand{\cS}{\mathcal{S}}   
\newcommand{\TT}{\mathbb{T}}   

\newcommand{\one}{\mathbbm{1}}
\newcommand{\supp}{\operatorname{supp}}

\numberwithin{equation}{section}

\begin{document}

	\renewcommand{\proofname}{\bf Proof}
	\maketitle

	\begin{abstract}
	The paper introduces a Poisson-type problem on a mixed-dimensional structure combining a Euclidean domain and a lower-dimensional self-similar component touching a compact surface (interface). The lower-dimensional piece is a so-called infinite metric tree (one-dimensional branching structure), and the key ingredient of the study is a rigorous definition of the gluing conditions between the two components. These constructions are based on the recent concept of embedded trace maps
	and some abstract machineries derived from a suitable Green-type formula. The problem is then reduced to 
	the study of Fredholm properties of a linear combination of Dirichlet-to-Neumann maps for the tree and the Euclidean domain, which yields desired existence and uniqueness results. One also shows that finite sections of tree can be used for an efficient approximation of the solutions.
	\end{abstract}

	\tableofcontents

	\section{Introduction}
	
	The aim of the present work is to develop an abstract framework to study the solvability and approximations for a Poisson-type problem on a mixed-dimensional structure combining a Euclidean component and a lower-dimensional self-similar component touching along a surface. The Euclidean component is a so-called exterior domain $\Omega\subset\R^m$ (with $m\ge 2$), i.e. $\Omega$ is the complement of a compact set. We additionally assume that $\Omega$
	has smooth boundary, hence,
	\[
	\Gamma:=\partial\Omega,
	\]
	is a compact hypersurface in $\R^m$ (without boundary), and that both $\Omega$ and $\Gamma$ are connected. The lower-dimensional component is an infinite metric tree constructed as follows. 
	Fix some branching number $p\ge 2$ and consider some bounded interval (root edge), with one of its endpoints being declared as a root vetrex $o$. In the first step we attach to the non-root endpoint of this interval $p$ further bounded intervals ($1$st generation edges). If all $n$-th generation edges are constructed, one attaches $p$ new bounded intervals to the free end of each of them to obtain the $(n+1)$-th generation edges, and this process continues infinitely and creates a so-called rooted $p$-adic metric tree, denoted as $\cT$, which is a kind of a branched one-dimensional structure (see Fig.~\ref{fig-intro1} for an illustration and Subsection~\ref{sec21} below for a detailed description).
	
	\begin{figure}[h]
		
		\centering
		
		\tikzset{every picture/.style={line width=0.75pt}} 
		
		\scalebox{1.0}{\begin{tikzpicture}[x=0.75pt,y=0.75pt,yscale=-1,xscale=1]
			
			\draw [line width=1.5]    (111,21.5) -- (111,61) ;
			\draw [line width=1.5]    (111,61) -- (70,90) ;
			\draw [line width=1.5]    (111,61) -- (111.5,88.5) ;
			\draw [line width=1.5]    (111,61) -- (153,89) ;
			\draw [line width=1.5]    (70,90) -- (43,110) ;
			\draw [line width=1.5]    (60,110.5) -- (70,90) ;
			\draw [line width=1.5]    (73.5,110) -- (70,90) ;
			\draw [line width=1.5]    (111.5,88.5) -- (104,109) ;
			\draw [line width=1.5]    (112,109) -- (111.5,88.5) ;
			\draw [line width=1.5]    (123,109) -- (111.5,88.5) ;
			\draw [line width=1.5]    (180.5,111) -- (153,89) ;
			\draw [line width=1.5]    (164,111.5) -- (153,89) ;
			\draw [line width=1.5]    (149,110.5) -- (153,89) ;
			\draw  [dash pattern={on 0.84pt off 2.51pt}]  (9.5,60) -- (211.5,60.5) ;
			\draw  [dash pattern={on 0.84pt off 2.51pt}]  (9.5,91) -- (210,90.5) ;
			\draw  [dash pattern={on 0.84pt off 2.51pt}]  (10,111.5) -- (210.5,111) ;
			
			\draw (106.5,2.9) node [anchor=north west][inner sep=0.75pt]    {$o$ (root vertex)};
			\draw (116,30) node [anchor=north west][inner sep=0.75pt]   [align=left] {root edge};
			\draw (159,65) node [anchor=north west][inner sep=0.75pt]   [align=left] {1st generation edges};
			\draw (189.5,90.5) node [anchor=north west][inner sep=0.75pt]   [align=left] {2nd generation edges};
			\draw (75.5,113) node [anchor=north west][inner sep=0.75pt]   [align=left] {(and so on)};

		\end{tikzpicture}}
		
		\caption{An illustration of the structure of the tree $\T$ with $p=3$.}\label{fig-intro1}
		
	\end{figure}
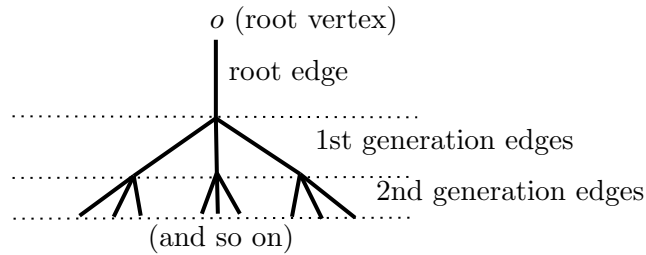

	Using the usual differentiation along each edge
	one then defines Sobolev-type spaces on $\T$ and the Laplace operator on $\cT$ (which is just the second derivative on each edge with transmission conditions at the branching point, see below for a detailed explanation).
	The key ingredient of our analysis is the possibility to interpret the above $\Gamma$ as the boundary of $\cT$, which allows one to define a trace operator $\gamma^\cT_0$ associating
	to each Sobolev-regular function $u$ on the tree a function $\gamma^\T_0 u$ in $L^2(\Gamma)$ to be considered at its boundary trace. The abstract idea of such a map originates from the paper~\cite{msv} considering discrete graphs with Euclidean boundaries, which was extended 
	and generalized to the case of metric trees and manifolds by the authors and their collaborators in \cite{Embedded Trace,jks}. Intuitively, the definition of the trace map in the context of the present work prescribes the way  how the metric  $\cT$ is ``glued'' to $\Gamma$ in order to create a hybrid structure consisting of $\Omega$ and $\cT$ (see Fig.~\ref{fig-intro2}). 	Once the trace has been defined, the normal derivatives $\gamma_1^{\mathcal{T}}u$ of suitably regular functions $u$ on $\cT$ can be defined by duality, via a suitable Green's formula.  We additionally note that these constructions are very close to the 
	abstract approach proposed in \cite{post} and are in the spirit of the general approaches 
	to transmission and boundary value problems with ``bad'' boundaries, see e.g.~\cite{ach,lancia}.

	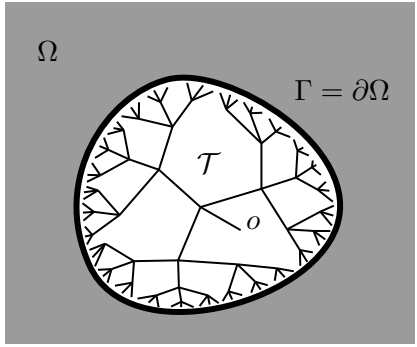
\begin{figure}
		\centering
		
		\scalebox{1.0}{			
			\tikzset{every picture/.style={line width=0.75pt}} 
						\begin{tikzpicture}[x=0.75pt,y=0.75pt,yscale=-1,xscale=1]
				
				\draw  [draw opacity=0][fill={rgb, 255:red, 155; green, 155; blue, 155 }  ,fill opacity=1 ] (22.6,53) -- (227,53) -- (227,225) -- (22.6,225) -- cycle ;
				\draw  [fill={rgb, 255:red, 255; green, 255; blue, 255 }  ,fill opacity=1 ][line width=2.25]  (103.75,92) .. controls (134.25,84.25) and (202,127) .. (188,166.5) .. controls (174,206) and (84.25,228.9) .. (64.75,185.5) .. controls (45.25,142.1) and (73.25,99.75) .. (103.75,92) -- cycle ;
				\draw    (120.2,156.2) -- (140.2,167.8) ;
				\draw    (150.6,147) -- (120.2,156.2) ;
				\draw    (120.2,156.2) -- (99.4,137.4) ;
				\draw    (109,183) -- (120.2,156.2) ;
				\draw    (163.4,167.4) -- (150.6,147) ;
				\draw    (150.6,147) -- (167.8,141.4) ;
				\draw    (150.6,147) -- (150.6,124.2) ;
				\draw    (99.4,137.4) -- (106.2,115.8) ;
				\draw    (84.6,132.2) -- (99.4,137.4) ;
				\draw    (79.8,160.2) -- (99.4,137.4) ;
				\draw    (87,183.4) -- (109,183) ;
				\draw    (109,183) -- (138.6,185) ;
				\draw    (109,183) -- (109.4,196.6) ;
				\draw    (87,183.4) -- (86.2,194.2) ;
				\draw    (72.6,174.6) -- (87,183.4) ;
				\draw    (87,183.4) -- (75.8,187.4) ;
				\draw    (109.4,196.6) -- (98.2,200.6) ;
				\draw    (123.8,200.6) -- (109.4,196.6) ;
				\draw    (110.6,203.8) -- (109.4,196.6) ;
				\draw    (134.2,198.6) -- (138.6,185) ;
				\draw    (138.6,185) -- (147.5,193.75) ;
				\draw    (138.6,185) -- (158.25,187) ;
				\draw    (163.4,167.4) -- (169.8,177.8) ;
				\draw    (163.4,167.4) -- (176.6,168.2) ;
				\draw    (163.4,167.4) -- (180,158.5) ;
				\draw    (167.8,141.4) -- (178.6,147.4) ;
				\draw    (175.8,136.2) -- (167.8,141.4) ;
				\draw    (168.6,127.8) -- (167.8,141.4) ;
				\draw    (157.8,118.2) -- (150.6,124.2) ;
				\draw    (150.6,124.2) -- (145,109) ;
				\draw    (150.6,124.2) -- (133,107) ;
				\draw    (115.8,102.6) -- (106.2,115.8) ;
				\draw    (100.2,103.4) -- (106.2,115.8) ;
				\draw    (89.4,111.8) -- (106.2,115.8) ;
				\draw    (81.8,117.4) -- (84.6,132.2) ;
				\draw    (84.6,132.2) -- (74.2,127.4) ;
				\draw    (84.6,132.2) -- (70.2,140.2) ;
				\draw    (79.8,160.2) -- (69,147.8) ;
				\draw    (67,158.6) -- (79.8,160.2) ;
				\draw    (79.8,160.2) -- (66.2,170.2) ;
				\draw    (80.75,110.5) -- (89.4,111.8) ;
				\draw    (85.5,106) -- (89.4,111.8) ;
				\draw    (89.5,103.5) -- (89.4,111.8) ;
				\draw    (103.75,96) -- (100.2,103.4) ;
				\draw    (100.2,103.4) -- (97,97.5) ;
				\draw    (93.75,100.5) -- (100.2,103.4) ;
				\draw    (115.8,102.6) -- (113.5,93.75) ;
				\draw    (121.5,95) -- (115.8,102.6) ;
				\draw    (126.5,98) -- (115.8,102.6) ;
				\draw    (133,107) -- (130,98.75) ;
				\draw    (135,99.75) -- (133,107) ;
				\draw    (133,107) -- (138.75,102.25) ;
				\draw    (141.75,103.5) -- (145,109) ;
				\draw    (145,109) -- (146,105.75) ;
				\draw    (151,108.75) -- (145,109) ;
				\draw    (157.8,118.2) -- (154.5,111.75) ;
				\draw    (160.5,116) -- (157.8,118.2) ;
				\draw    (165,121) -- (157.8,118.2) ;
				\draw    (168.6,127.8) -- (166.75,122.25) ;
				\draw    (171.5,126.25) -- (168.6,127.8) ;
				\draw    (168.6,127.8) -- (174.25,130.25) ;
				\draw    (175.8,136.2) -- (180.25,140) ;
				\draw    (175.8,136.2) -- (178,135) ;
				\draw    (175.8,136.2) -- (176.25,132.75) ;
				\draw    (182.75,143.5) -- (178.6,147.4) ;
				\draw    (178.6,147.4) -- (184.5,148) ;
				\draw    (178.6,147.4) -- (186.25,153.5) ;
				\draw    (180,158.5) -- (186.75,155.5) ;
				\draw    (186.5,160.5) -- (180,158.5) ;
				\draw    (185.75,165.5) -- (180,158.5) ;
				\draw    (179.5,175.5) -- (176.6,168.2) ;
				\draw    (185,168.25) -- (176.6,168.2) ;
				\draw    (182.5,172.5) -- (176.6,168.2) ;
				\draw    (176.25,179) -- (169.8,177.8) ;
				\draw    (169.8,177.8) -- (173.5,183.75) ;
				\draw    (169.8,177.8) -- (170,185.75) ;
				\draw    (158.25,187) -- (157.25,193.25) ;
				\draw    (162.25,191.5) -- (158.25,187) ;
				\draw    (166.75,188.5) -- (158.25,187) ;
				\draw    (81.8,117.4) -- (78.5,113.25) ;
				\draw    (76.75,116.5) -- (81.8,117.4) ;
				\draw    (73.75,119.5) -- (81.8,117.4) ;
				\draw    (67.5,131.5) -- (74.2,127.4) ;
				\draw    (74.2,127.4) -- (72,121.25) ;
				\draw    (69.75,125.5) -- (74.2,127.4) ;
				\draw    (66,134) -- (70.2,140.2) ;
				\draw    (70.2,140.2) -- (63.75,140) ;
				\draw    (70.2,140.2) -- (63.5,143.5) ;
				\draw    (69,147.8) -- (63.25,146) ;
				\draw    (62,148.75) -- (69,147.8) ;
				\draw    (69,147.8) -- (61.75,152.25) ;
				\draw    (67,158.6) -- (61.25,154.25) ;
				\draw    (61.5,159) -- (67,158.6) ;
				\draw    (61.75,163) -- (67,158.6) ;
				\draw    (62.5,166) -- (66.2,170.2) ;
				\draw    (66.2,170.2) -- (62.5,171) ;
				\draw    (66.2,170.2) -- (63.5,175) ;
				\draw    (72.6,174.6) -- (64.75,177.75) ;
				\draw    (66.25,181.5) -- (72.6,174.6) ;
				\draw    (67.5,186) -- (72.6,174.6) ;
				\draw    (69,187.75) -- (75.8,187.4) ;
				\draw    (71.5,191) -- (75.8,187.4) ;
				\draw    (75.8,187.4) -- (73.75,194.25) ;
				\draw    (86.2,194.2) -- (77.5,195.75) ;
				\draw    (86.2,194.2) -- (81,199.5) ;
				\draw    (86.2,194.2) -- (85.75,201.5) ;
				\draw    (98.2,200.6) -- (88.75,202.25) ;
				\draw    (98.2,200.6) -- (94.25,204.25) ;
				\draw    (99.25,206.5) -- (98.2,200.6) ;
				\draw    (153.75,195.5) -- (147.5,193.75) ;
				\draw    (150.25,197.75) -- (147.5,193.75) ;
				\draw    (147.5,193.75) -- (145.5,199.25) ;
				\draw    (134.2,198.6) -- (141,200.5) ;
				\draw    (136,202.75) -- (134.2,198.6) ;
				\draw    (131.25,203.5) -- (134.2,198.6) ;
				\draw    (105.5,206.25) -- (110.6,203.8) ;
				\draw    (110.6,203.8) -- (110.5,206.5) ;
				\draw    (116.75,206.5) -- (110.6,203.8) ;
				\draw    (129.25,204.5) -- (123.8,200.6) ;
				\draw    (123.8,200.6) -- (119.75,206) ;
				\draw    (123.8,200.6) -- (124.75,205.25) ;
				
				\draw (141.4,159.6) node [anchor=north west][inner sep=0.75pt]    {$o$};
				\draw (165.6,90.8) node [anchor=north west][inner sep=0.75pt]    {$\Gamma =\partial \Omega $};
				\draw (37.2,71) node [anchor=north west][inner sep=0.75pt]    {$\Omega $};
				\draw (116.8,127.8) node [anchor=north west][inner sep=0.75pt]    {$\T$};

			\end{tikzpicture}}
		
		\caption{Gluing between the tree $\T$ and the exterior domain $\Omega$.}\label{fig-intro2}
	\end{figure}
	
The introduction of the above objects provides necessary ingredients
to define the following Poisson-type problem with transmission conditions: Given source terms (functions) $f_\T:\T\to\C$ and $f_\Omega:\Omega\to\C$, transmission coefficients $\alpha_0: \Gamma\to \C $, $\alpha_1\in\C\setminus\{0\}$, and a constant $c\in\C$, find functions $u_{\mathcal{T}}: \, \mathcal{T}\rightarrow \mathbb{C}$, ${u}_{\Omega}: \, {\Omega}\rightarrow \mathbb{C}$ satisfying
\begin{equation}
	\label{transm-intro}
	\left\{\begin{aligned}
		\Delta_\cT u_\cT&=f_\cT \text{ on } \cT,\\
		\Delta u_\Omega&=f_\Omega \text{ on } \Omega,\\
		\gamma^\Omega_0 u_\Omega&=\gamma^\cT_0 u_\cT \text{ on } \Gamma,\\
		\gamma^\Omega_1 u_\Omega-\alpha_1\gamma^\cT_1 u_\cT&=\alpha_0 \gamma^\Omega_0 u_\Omega \text{ on }\Gamma,\\
		u_\Omega(x)&=O(|x|^{2-m}) \text{ for } |x|\to\infty,\\
		u_\cT(o)&=c,
	\end{aligned}
	\right.	
\end{equation}
where as usual by $\gamma^\Omega_0 u_\Omega$ and $\gamma^\Omega_1 u_\Omega$ we denote the boundary trace and the (inner) normal derivative of $u_\Omega$ on $\Gamma$.
The main objective of our analysis is to define a suitable functional framework (in particular, suitable function spaces for the sources $f_\T$, $f_\Omega$ and the solutions $u_\T$, $u_\Omega$) guaranteeing the (unique) solvability of the above problem. This is done via restating it as an equivalent boundary integral equation on $\Gamma$, which involves the exterior and interior Dirichlet-to-Neumann maps; we refer to~\cite{costabel_stephan} for a similar approach in the purely Euclidean setting
exploiting the boundary integral equation formalism.
Our analysis relies on studying the mapping and positivity properties of the Dirichlet-to-Neumann operators, and allows to conclude about the well-posedness of the problem \eqref{transm-intro} under minor restrictions on the transmission coefficients. 

In addition to purely theoretical questions, we address some problems related to the numerical analysis. In particular, we study the effect of replacing the tree $\mathcal{T}$ by its counterpart with finitely many generations and study the effect of truncations on the solution of \eqref{transm-intro}.  Let us remark that this is different from \cite{jk1,jk2,jks}, where the Dirichlet-to-Neumann operator was associated to the root of the tree, and similarity properties of the tree could be used to recover its various explicit representations.   While the primary motivation of this work is purely theoretical, its results can be used, for example, in the context of modeling fractal tree antennas \cite{petko_werner} approximated by one-dimensional structures, see \cite{joly_semin} for the related analysis.

This paper is organized as follows. In Section \ref{sec:2} we describe results related to the boundary-value problems on the metric tree. In particular, we recall associated function spaces, define the Laplace operator and trace spaces (Sections \ref{sec21}--\ref{sec23}). Section \ref{sec24} is devoted to the definition of the normal derivative and the Dirichlet-to-Neumann operator on the tree. Finally, Sections \ref{sec-fin} and \ref{sec-dtn} deal with approximation of the Dirichlet-to-Neumann map by finite-dimensional operators. 
Section \ref{sec3-bvp} is dedicated to the boundary-value problems on exterior domains. 
We note that the properties of the Dirichlet-to-Neumann map on exterior domains seem less known that they should be, so we decided to review a part of the theory in Section \ref{sec3-bvp} in order to establish all necessary properties. 
The concluding Section \ref{sec:4} establishes the well-posedness of \eqref{transm-intro} through the analysis of an equivalent problem posed on the interface $\Gamma$ and of its approximated counterpart.

	\section{Boundary value problems on the metric tree}
	\label{sec:2}
	This section is dedicated to the definition and well-posedness of the boundary value problems on the metric tree. In Section \ref{sec21}, we provide a detailed definition of the metric tree $\mathcal{T}$ and related geometrical assumptions; we state the function spaces and their properties. Section \ref{sec22} is dedicated to the definition of the associated Laplace operator. In Section \ref{sec23} we review the construction of the trace map. This allows to define the conormal trace (normal derivative) and associated Dirichlet-to-Neumann map in Section \ref{sec24}. Finally, in Sections \ref{sec-fin} and \ref{sec-dtn} we show how this map can be efficiently approximated by finite-dimensional maps.
	
	\subsection{Function spaces}\label{sec21}
	
	Let $p\in\N$ with $p\ge 2$ and a root $o$ be given. We glue to $o$ an edge $e_{0,0}$ represented by an interval of length $\ell_{0,0}$, the second vertex of $e_{0,0}$ will be called $X_{0,0}$. If all $e_{n,k}$ and $X_{n,k}$ with $n\in\N_0:=\N\cup\{0\}$ and $k\in\{0,\dots,p^n-1\}$ are already constructed, then to each $X_{n,k}$ we attach
	$p$ new edges $e_{n+1,pk+j}$, with $j\in\{0,\dots,p-1\}$, having lengths $\ell_{n+1,pk+j}$,
	and the pendant vertices of $e_{n+1,pk+j}$, to be denoted by $X_{n+1,pk+j}$, will be viewed as children of $X_{n,k}$. This process continues infinitely, which creates a infinite rooted metric tree $\cT$ (see Figure~\ref{fig1} for an illustration).
	
	\begin{figure}
		\centering

		\scalebox{0.75}{\begin{tikzpicture}[x=0.75pt,y=0.75pt,yscale=-1,xscale=1]
				
				\draw [line width=2.25]    (131.4,19.2) -- (131.4,78.8) ;
				\draw [line width=2.25]    (131.4,78.8) -- (81.4,118) ;
				\draw [line width=2.25]    (181.4,118) -- (189,155.6) ;
				\draw [line width=2.25]    (81.4,118) -- (36.2,152.8) ;
				\draw [line width=2.25]    (75.8,154.4) -- (81.4,118) ;
				\draw [line width=2.25]    (181.4,118) -- (131.4,78.8) ;
				\draw [line width=2.25]    (181.4,118) -- (225,154.8) ;
				\draw  [fill={rgb, 255:red, 0; green, 0; blue, 0 }  ,fill opacity=1 ] (128.8,78.8) .. controls (128.8,77.36) and (129.96,76.2) .. (131.4,76.2) .. controls (132.84,76.2) and (134,77.36) .. (134,78.8) .. controls (134,80.24) and (132.84,81.4) .. (131.4,81.4) .. controls (129.96,81.4) and (128.8,80.24) .. (128.8,78.8) -- cycle ;
				\draw  [fill={rgb, 255:red, 0; green, 0; blue, 0 }  ,fill opacity=1 ] (78.8,118) .. controls (78.8,116.56) and (79.96,115.4) .. (81.4,115.4) .. controls (82.84,115.4) and (84,116.56) .. (84,118) .. controls (84,119.44) and (82.84,120.6) .. (81.4,120.6) .. controls (79.96,120.6) and (78.8,119.44) .. (78.8,118) -- cycle ;
				\draw  [fill={rgb, 255:red, 0; green, 0; blue, 0 }  ,fill opacity=1 ] (73.2,154.4) .. controls (73.2,152.96) and (74.36,151.8) .. (75.8,151.8) .. controls (77.24,151.8) and (78.4,152.96) .. (78.4,154.4) .. controls (78.4,155.84) and (77.24,157) .. (75.8,157) .. controls (74.36,157) and (73.2,155.84) .. (73.2,154.4) -- cycle ;
				\draw  [fill={rgb, 255:red, 0; green, 0; blue, 0 }  ,fill opacity=1 ] (33.6,152.8) .. controls (33.6,151.36) and (34.76,150.2) .. (36.2,150.2) .. controls (37.64,150.2) and (38.8,151.36) .. (38.8,152.8) .. controls (38.8,154.24) and (37.64,155.4) .. (36.2,155.4) .. controls (34.76,155.4) and (33.6,154.24) .. (33.6,152.8) -- cycle ;
				\draw  [fill={rgb, 255:red, 0; green, 0; blue, 0 }  ,fill opacity=1 ] (178.8,118) .. controls (178.8,116.56) and (179.96,115.4) .. (181.4,115.4) .. controls (182.84,115.4) and (184,116.56) .. (184,118) .. controls (184,119.44) and (182.84,120.6) .. (181.4,120.6) .. controls (179.96,120.6) and (178.8,119.44) .. (178.8,118) -- cycle ;
				\draw  [fill={rgb, 255:red, 0; green, 0; blue, 0 }  ,fill opacity=1 ] (128.8,16.6) .. controls (128.8,15.16) and (129.96,14) .. (131.4,14) .. controls (132.84,14) and (134,15.16) .. (134,16.6) .. controls (134,18.04) and (132.84,19.2) .. (131.4,19.2) .. controls (129.96,19.2) and (128.8,18.04) .. (128.8,16.6) -- cycle ;
				\draw  [fill={rgb, 255:red, 0; green, 0; blue, 0 }  ,fill opacity=1 ] (186.4,155.6) .. controls (186.4,154.16) and (187.56,153) .. (189,153) .. controls (190.44,153) and (191.6,154.16) .. (191.6,155.6) .. controls (191.6,157.04) and (190.44,158.2) .. (189,158.2) .. controls (187.56,158.2) and (186.4,157.04) .. (186.4,155.6) -- cycle ;
				\draw  [fill={rgb, 255:red, 0; green, 0; blue, 0 }  ,fill opacity=1 ] (222.4,154.8) .. controls (222.4,153.36) and (223.56,152.2) .. (225,152.2) .. controls (226.44,152.2) and (227.6,153.36) .. (227.6,154.8) .. controls (227.6,156.24) and (226.44,157.4) .. (225,157.4) .. controls (223.56,157.4) and (222.4,156.24) .. (222.4,154.8) -- cycle ;
				\draw  [dash pattern={on 0.84pt off 2.51pt}]  (91,120) -- (171.4,120) ;
				\draw  [dash pattern={on 0.84pt off 2.51pt}]  (43.4,153.6) -- (69.4,153.6) ;
				\draw  [dash pattern={on 0.84pt off 2.51pt}]  (195,155.2) -- (219.8,154.8) ;
				\draw [line width=2.25]  [dash pattern={on 2.53pt off 3.02pt}]  (36.2,152.8) -- (15.8,190.8) ;
				\draw [line width=2.25]  [dash pattern={on 2.53pt off 3.02pt}]  (36.2,152.8) -- (49.4,190.8) ;
				\draw [line width=2.25]  [dash pattern={on 2.53pt off 3.02pt}]  (75.8,154.4) -- (65.8,191.2) ;
				\draw [line width=2.25]  [dash pattern={on 2.53pt off 3.02pt}]  (75.8,154.4) -- (91.8,191.6) ;
				\draw [line width=2.25]  [dash pattern={on 2.53pt off 3.02pt}]  (188.6,154.8) -- (168.2,192.8) ;
				\draw [line width=2.25]  [dash pattern={on 2.53pt off 3.02pt}]  (188.6,154.8) -- (198.2,189.6) ;
				\draw [line width=2.25]  [dash pattern={on 2.53pt off 3.02pt}]  (225,154.8) -- (215,191.6) ;
				\draw [line width=2.25]  [dash pattern={on 2.53pt off 3.02pt}]  (225,154.8) -- (232.2,171.54) -- (241,192) ;
				\draw  [dash pattern={on 0.84pt off 2.51pt}]  (26.6,184.8) -- (41.4,185.2) ;
				\draw  [dash pattern={on 0.84pt off 2.51pt}]  (72.6,186) -- (84.6,186) ;
				\draw  [dash pattern={on 0.84pt off 2.51pt}]  (180.2,183.6) -- (185,183.6) -- (192.2,183.6) ;
				\draw  [dash pattern={on 0.84pt off 2.51pt}]  (221.8,188) -- (226.6,188) -- (233.8,188) ;
				
				\draw (116.9,3.8) node [anchor=north west][inner sep=0.75pt]    {$o$};
				\draw (136.4,28.2) node [anchor=north west][inner sep=0.75pt]    {$e_{0,0}$};
				\draw (98,58.6) node [anchor=north west][inner sep=0.75pt]    {$X_{0,0}$};
				\draw (52,99) node [anchor=north west][inner sep=0.75pt]    {$X_{1,0}$};
				\draw (185,99) node [anchor=north west][inner sep=0.75pt]    {$X_{1,p-1}$};
				\draw (82,81) node [anchor=north west][inner sep=0.75pt]    {$e_{1,0}$};
				\draw (154,82) node [anchor=north west][inner sep=0.75pt]    {$e_{1,p-1}$};
				\draw (1.6,145) node [anchor=north west][inner sep=0.75pt]    {$X_{2,0}$};
				\draw (233.2,145) node [anchor=north west][inner sep=0.75pt]    {$X_{2,p^{2} -1}$};
				\draw (80,145) node [anchor=north west][inner sep=0.75pt]    {$X_{2,p-1}$};
				\draw (128.4,145) node [anchor=north west][inner sep=0.75pt]    {$X_{2,p( p-1)}$};
				\draw (30,125) node [anchor=north west][inner sep=0.75pt]    {$e_{2,0}$};
				\draw (83.4,125) node [anchor=north west][inner sep=0.75pt]    {$e_{2,p-1}$};
				\draw (128.6,125) node [anchor=north west][inner sep=0.75pt]    {$e_{2,p( p-1)}$};
				\draw (209.4,125) node [anchor=north west][inner sep=0.75pt]    {$e_{2,p^{2} -1}$};

			\end{tikzpicture}}
		\caption{The tree $\cT$}\label{fig1}
	\end{figure}

	The subtree of $\cT$ starting at $X_{n,k}$, i.e. the subtree spanned by the offsping of $X_{n,k}$ (the children, the children of the children etc.), will be denoted by $\cT_{n,k}$. If $j\in\{0,\dots,p-1\}$, denote by $\cT^j_{n,k}$ the subtree $e_{n+1,pk+j}\cup \cT_{n+1,pk+j}$, which has the same combinatorial structure as $\cT$ with $X_{n,k}$ considered as a root (see Fig.~\ref{fig2}). Note that $\T_{n,k}$ represents the union of $\T^j_{n,k}$ with $j\in\{0,\dots,p-1\}$.

	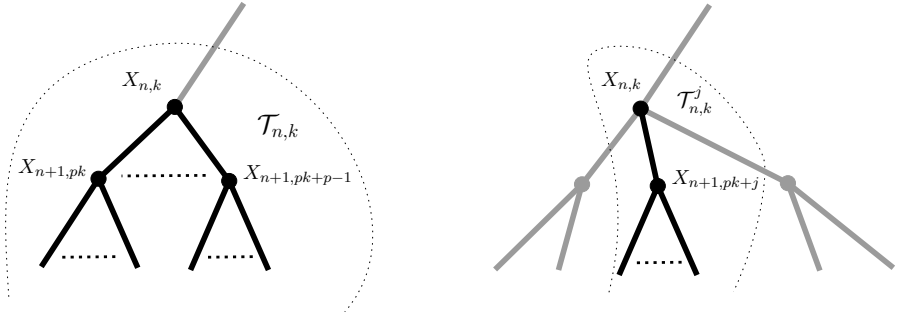
\begin{figure}
		
		\centering
		
		\scalebox{0.7}{\begin{tikzpicture}[x=0.75pt,y=0.75pt,yscale=-1,xscale=1]
				
				\draw [color={rgb, 255:red, 155; green, 155; blue, 155 }  ,draw opacity=1 ][line width=3]    (138.6,82.6) -- (187.8,8.2) ;
				\draw  [fill={rgb, 255:red, 0; green, 0; blue, 0 }  ,fill opacity=1 ] (133.2,82.6) .. controls (133.2,79.62) and (135.62,77.2) .. (138.6,77.2) .. controls (141.58,77.2) and (144,79.62) .. (144,82.6) .. controls (144,85.58) and (141.58,88) .. (138.6,88) .. controls (135.62,88) and (133.2,85.58) .. (133.2,82.6) -- cycle ;
				\draw [line width=3]    (138.6,82.6) -- (84.2,134.2) ;
				\draw  [fill={rgb, 255:red, 0; green, 0; blue, 0 }  ,fill opacity=1 ] (78.8,134.2) .. controls (78.8,131.22) and (81.22,128.8) .. (84.2,128.8) .. controls (87.18,128.8) and (89.6,131.22) .. (89.6,134.2) .. controls (89.6,137.18) and (87.18,139.6) .. (84.2,139.6) .. controls (81.22,139.6) and (78.8,137.18) .. (78.8,134.2) -- cycle ;
				\draw [line width=3]    (177.4,135.8) -- (138.6,82.6) ;
				\draw  [fill={rgb, 255:red, 0; green, 0; blue, 0 }  ,fill opacity=1 ] (172,135.8) .. controls (172,132.82) and (174.42,130.4) .. (177.4,130.4) .. controls (180.38,130.4) and (182.8,132.82) .. (182.8,135.8) .. controls (182.8,138.78) and (180.38,141.2) .. (177.4,141.2) .. controls (174.42,141.2) and (172,138.78) .. (172,135.8) -- cycle ;
				\draw [line width=1.5]  [dash pattern={on 1.69pt off 2.76pt}]  (160.6,131.8) -- (100.6,132.2) ;
				\draw [line width=3]    (84.2,134.2) -- (43.4,197.4) ;
				\draw [line width=3]    (84.2,134.2) -- (112.2,198.2) ;
				\draw [line width=3]    (177.4,135.8) -- (205.4,199.8) ;
				\draw [line width=3]    (177.4,135.8) -- (149.8,199.8) ;
				\draw [line width=1.5]  [dash pattern={on 1.69pt off 2.76pt}]  (96.2,190.2) -- (58.2,190.6) ;
				\draw [line width=1.5]  [dash pattern={on 1.69pt off 2.76pt}]  (194.2,190.2) -- (162.2,190.6) ;
				\draw [color={rgb, 255:red, 0; green, 0; blue, 0 }  ,draw opacity=1 ] [dash pattern={on 0.84pt off 2.51pt}]  (20.2,219) .. controls (15.4,141) and (7.8,55) .. (121.8,38.2) .. controls (235.8,21.4) and (319.8,161.8) .. (259.8,229.8) ;
				\draw [color={rgb, 255:red, 155; green, 155; blue, 155 }  ,draw opacity=1 ][line width=3]    (470.8,83.8) -- (520,9.4) ;
				\draw [color={rgb, 255:red, 155; green, 155; blue, 155 }  ,draw opacity=1 ][fill={rgb, 255:red, 128; green, 128; blue, 128 }  ,fill opacity=1 ][line width=3]    (470.8,83.8) -- (429,138.2) ;
				\draw [line width=3]    (483,139.4) -- (470.8,83.8) ;
				\draw  [fill={rgb, 255:red, 0; green, 0; blue, 0 }  ,fill opacity=1 ] (477.6,139.4) .. controls (477.6,136.42) and (480.02,134) .. (483,134) .. controls (485.98,134) and (488.4,136.42) .. (488.4,139.4) .. controls (488.4,142.38) and (485.98,144.8) .. (483,144.8) .. controls (480.02,144.8) and (477.6,142.38) .. (477.6,139.4) -- cycle ;
				\draw [line width=3]    (483,139.4) -- (511,203.4) ;
				\draw [line width=3]    (483,139.4) -- (455.4,203.4) ;
				\draw [line width=1.5]  [dash pattern={on 1.69pt off 2.76pt}]  (500,194.2) -- (468,194.6) ;
				\draw [color={rgb, 255:red, 0; green, 0; blue, 0 }  ,draw opacity=1 ] [dash pattern={on 0.84pt off 2.51pt}]  (448.2,215.8) .. controls (477,95) and (396.2,40.2) .. (455.2,39) .. controls (514.2,37.8) and (557,107) .. (559,122.2) .. controls (561,137.4) and (554.6,177.8) .. (537,215.4) ;
				\draw [color={rgb, 255:red, 155; green, 155; blue, 155 }  ,draw opacity=1 ][fill={rgb, 255:red, 155; green, 155; blue, 155 }  ,fill opacity=1 ][line width=3]    (473.4,85) -- (575.8,137.4) ;
				\draw  [color={rgb, 255:red, 155; green, 155; blue, 155 }  ,draw opacity=1 ][fill={rgb, 255:red, 155; green, 155; blue, 155 }  ,fill opacity=1 ] (423.6,138.2) .. controls (423.6,135.22) and (426.02,132.8) .. (429,132.8) .. controls (431.98,132.8) and (434.4,135.22) .. (434.4,138.2) .. controls (434.4,141.18) and (431.98,143.6) .. (429,143.6) .. controls (426.02,143.6) and (423.6,141.18) .. (423.6,138.2) -- cycle ;
				\draw  [fill={rgb, 255:red, 0; green, 0; blue, 0 }  ,fill opacity=1 ] (465.4,83.8) .. controls (465.4,80.82) and (467.82,78.4) .. (470.8,78.4) .. controls (473.78,78.4) and (476.2,80.82) .. (476.2,83.8) .. controls (476.2,86.78) and (473.78,89.2) .. (470.8,89.2) .. controls (467.82,89.2) and (465.4,86.78) .. (465.4,83.8) -- cycle ;
				\draw  [color={rgb, 255:red, 155; green, 155; blue, 155 }  ,draw opacity=1 ][fill={rgb, 255:red, 155; green, 155; blue, 155 }  ,fill opacity=1 ] (570.4,137.4) .. controls (570.4,134.42) and (572.82,132) .. (575.8,132) .. controls (578.78,132) and (581.2,134.42) .. (581.2,137.4) .. controls (581.2,140.38) and (578.78,142.8) .. (575.8,142.8) .. controls (572.82,142.8) and (570.4,140.38) .. (570.4,137.4) -- cycle ;
				\draw [color={rgb, 255:red, 155; green, 155; blue, 155 }  ,draw opacity=1 ][fill={rgb, 255:red, 128; green, 128; blue, 128 }  ,fill opacity=1 ][line width=3]    (575.8,137.4) -- (598.6,200.2) ;
				\draw [color={rgb, 255:red, 155; green, 155; blue, 155 }  ,draw opacity=1 ][fill={rgb, 255:red, 128; green, 128; blue, 128 }  ,fill opacity=1 ][line width=3]    (429,138.2) -- (412.2,199.8) ;
				\draw [color={rgb, 255:red, 155; green, 155; blue, 155 }  ,draw opacity=1 ][fill={rgb, 255:red, 128; green, 128; blue, 128 }  ,fill opacity=1 ][line width=3]    (575.8,137.4) -- (651,198.2) ;
				\draw [color={rgb, 255:red, 155; green, 155; blue, 155 }  ,draw opacity=1 ][fill={rgb, 255:red, 128; green, 128; blue, 128 }  ,fill opacity=1 ][line width=3]    (429,138.2) -- (367,199.4) ;
				
				\draw (99.6,56.6) node [anchor=north west][inner sep=0.75pt]    {$X_{n,k}$};
				\draw (25.6,118.8) node [anchor=north west][inner sep=0.75pt]    {$X_{n+1,pk}$};
				\draw (186,122.8) node [anchor=north west][inner sep=0.75pt]    {$X_{n+1,pk+p-1}$};
				\draw (196,86.6) node [anchor=north west][inner sep=0.75pt]    {\Large $\T_{n,k}$};
				\draw (441.4,56.2) node [anchor=north west][inner sep=0.75pt]    {$X_{n,k}$};
				\draw (491.8,126.8) node [anchor=north west][inner sep=0.75pt]    {$X_{n+1,pk+j}$};
				\draw (495.4,66.6) node [anchor=north west][inner sep=0.75pt]    {\large $\T_{n,k}^{j}$};

			\end{tikzpicture}
		}
		
		\caption{The subtrees $\T_{n,k}$ and $\T^j_{n,k}$.}\label{fig2}
	\end{figure}
	
%
	
	For subsequent constructions it will be useful to introduce coordinates on $\cT$. Denote by $L_{n,k}$ the distance between the root $o$ and  $X_{n,k}$, i.e. the length of the unique path between $o$ and $X_{n,k}$ obtained by summing the lengths of all edges in the path. Then by $(n,k,t)$ with $t\in[L_{n,k}-\ell_{n,k},L_{n,k}]$ we denote the point of $e_{n,k}$ which is at the distance  $L_{n,k}-t$ from $X_{n,k}$. In this notation,
	\[
	X_{n,k}=(n,k,L_{n,k})=(n+1,pk+j,L_{n,k}) \text{ for any } j\in\{0,\dots,p-1\}.
	\]
	Let $w:\cT\to(0,\infty)$ be a locally bounded measurable function, which will be used
	as an integration weight: for $f:\cT\to\C$ one defines
	\[
	\int_\cT f\dd\mu:=\sum_{n=0}^\infty\sum_{k=0}^{p^n-1} \int_{L_{n,k}-\ell_{n,k}}^{L_{n,k}}f(n,k,t)\,w(n,k,t)\dd t,
	\]
	then
	\[
	L^2(\cT):=\big\{f:\cT\to\C:\ \|f\|^2_{L^2(\cT)}:=\int_\cT |f|^2\dd\mu<\infty\big\}.
	\quad
	\]
	Due to the above definition the set of vertices has zero measure. Therefore, each measurable function $f:\cT\to\C$
	can be identified with a family of functions $(f_{n,k})$,
	\begin{gather*}
		f_{n,k}:=f(n,k,\cdot):\ (L_{n,k}-\ell_{n,k},L_{n,k})\to \C,\quad
		n\in\N_0,
		\quad
		k\in\{0,\dots,p^n-1\}.
	\end{gather*}
	Then $f=(f_{n,k})$ belongs to $L^2(\cT)$ if and only if
	\[
	\|f\|^2_{L^2(\cT)}:=\sum_{n=0}^\infty\sum_{k=0}^{p^n-1} \int_{L_{n,k}-\ell_{n,k}}^{L_{n,k}}\big|f_{n,k}(t)\big|^2w_{n,k}(t)\dd t<\infty.
	\]
	A function $f:\T\to\C$ is called continuous if $t\mapsto f_{n,k}(t)$ is continuous for all $n, k$ and, additionally, $f$ is continuous in the vertices. 
	In other words, $f_{n,k}(L_{n,k}^-)=f_{n+1,pk+j}(L_{n,k}^+)$ for all $n,k,j$. 
	If $f=(f_{n,k})$ is such that all $f_{n,k}$ have locally integrable distributional derivatives $f'_{n,k}$, we denote $f':=(f'_{n,k})$.

	The first Sobolev space $H^1(\cT)$ on $\cT$ is then introduced as
	\begin{gather*}
		H^1(\cT):=\{f\in L^2(\cT): f \text{ is continuous with } f'\in  L^2(\cT)\},\\
		\|f\|^2_{H^1(\cT)}:=\|f\|^2_{L^2(\cT)}+\|f'\|^2_{L^2(\cT)}.
	\end{gather*}
	Moreover, we denote
	\begin{align*}
		H^1_c(\cT)&:=\{f\in H^1(\cT): \text{ there exists $N\in\N$ such that } f_{n,k}\equiv 0\\
		&\qquad\qquad\qquad \text{ for all $(n,k)$ with $n>N$}\},\\
		H^1_0(\cT)&:=\text{the closure of $H^1_c(\cT)$ in $H^1(\cT)$,}\\
		\Tilde H^1(\cT)&:=\big\{ f\in H^1(\cT):\ f(o)=0\big\},\\
		\Tilde H^1_0(\cT)&:=\big\{ f\in H^1_0(\cT):\ f(o)=0\big\},\\
		\Tilde H^1_c(\cT)&:=\big\{ f\in H^1_c(\cT):\ f(o)=0\big\},
	\end{align*}
and remark that $\Tilde H^1_c(\cT)$ is a dense subset of $\Tilde H^1_0(\cT)$.
	
While the above definitions make sense for any choice of lengths $\ell_{n,k}$ and weights $w$, we introduce
additional conditions in order to have a more controllable global structure of $\cT$ and nice properties
of the associated Sobolev spaces. First, we assume that the weight function $w$ is constant on each edge, which 
induces the edge weigths
\[
\omega_{n,k}:=w|_{e_{n,k}}\in(0,\infty).
\]
The above $p$-adic metric tree $\cT$ equipped with edge lengths $\ell_{n,k}$ and edge weights $\omega_{n,k}$ will
be sometimes denoted as
\[
\cT^p\big((\ell_{n,k}),(\omega_{n,k})\big).
\]
In what follows we will always assume that there are some $\omega>0$ and $\ell\in(0,1)$ with
\begin{equation}
	\label{eq-lpa}
\ell < \omega p < \frac{1}{\ell}
\end{equation} 
and $C\ge 1$ such that for every $(n,k)$ there holds
		\begin{equation}
			\label{eq-ccc}
		\frac{1}{C} \ell^n \le  \ell_{n,k} \le C \ell^n \quad \text{ and } \quad \frac{1}{C} \omega^n \le  \omega_{n,k} \le C \omega^n.
		\end{equation}
The above assumption on the edge lengths guarantees, in particular, that the height of the tree is finite (which means that the distance to the root is bounded). On the other hand, the assumption on the weights combined with the upper bound in \eqref{eq-lpa} is a necessary and sufficient condition for constant functions to belong to $L^2(\cT)$, cf. \cite[proof of Theorem 3.10]{jks}. The role of the remaining assumptions will be discussed further in the paper. 

Moreover, if the stronger condition
\[
\ell_{n,k}=L_0 \ell^n,\quad
\omega_{n,k}=\omega_0 \omega^n
\text{ for all $(n,k)$}
\]		
is fulfilled (with some fixed $L_0>0$ and $\omega_0>0$), then the tree will be called \emph{geometric}	and denoted as
\begin{equation}
	\label{tgeom}
\TT^p\big(L_0,\ell,\omega_0,\omega\big).
\end{equation}
The following two results will show that the spaces $H^1(\cT)$, $H^1_0(\cT)$, $L^2(\cT)$ behave similarly to their counterparts on finite intervals of $\mathbb{R}$.  First of all, due to the assumption \eqref{eq-ccc} on $\ell$ and $\omega$, the following important result holds true: 
\begin{lem}\label{lemcomp}
The embedding $H^1(\T)\hookrightarrow L^2(\T)$ is compact, and
there is a constant $C_0>0$ such that
\begin{equation}
	\label{poincare}
	\|f\|_{L^2(\cT)} \leq C_0 \|f'\|_{L^2(\cT)} \qquad \text{ for all } f \in \widetilde{H}^1(\cT),
\end{equation}
and 
\begin{equation}
	\label{skal1}
	\langle f,g\rangle_{\Tilde H^1(\T)}:=\langle f',g'\rangle_{L^2(\T)}\equiv\int_\T f'\overline{g'}\dd \mu
\end{equation}
is a scalar product on $\Tilde H^1(\T)$ which is equivalent to the induced scalar product inherited from $H^1(\T)$.
\end{lem}
\begin{proof}
For the case $\cT=\TT^p(1,\ell,1,\omega)=:\mathbb{T}$ the compactness of the embedding is shown in \cite[Sec.~3.5]{jks}
and the assertions on the ineqality  \eqref{poincare} and the scalar product are shown in \cite[Sec.~2.3]{Embedded Trace}. 
Now we consider the dilation operator
\begin{gather*}
\varphi:\ \TT\to\cT,\qquad
\varphi(n,k,t)=\Big(n,k,L_{n,k}-\ell_{n,k}+\dfrac{t-t_{n-1}}{\ell^n}\ell_{n,k}\Big),\\
t_{-1}:=0,\quad t_n:=\sum_{k=0}^n \ell^k \text{ for }n\in\N_0,
\end{gather*}
and note that due to the assumptions \eqref{eq-ccc} the map $f\mapsto f\circ\varphi$ defines isomorphisms $L^2(\cT)\to L^2(\TT)$, $H^1(\cT)\to H^1(\TT)$ and $\Tilde H^1(\T)\to \Tilde H^1(\TT)$, see \cite[Sec.~4.3]{Embedded Trace}, which extends the both results to $\cT$.
\end{proof}
The next result will be elaborated later, cf. Theorem \ref{thm-trace}, however, we state it here to facilitate the understanding of the function spaces.
	\begin{lem}
		\label{lem:func_spaces}
$H^1_0(\mathcal{T})\subsetneq H^1(\mathcal{T})$ and $\widetilde{H}^1_0(\mathcal{T})\subsetneq \widetilde{H}^1(\mathcal{T})$.	
	\end{lem}

\subsection{Laplacian}\label{sec22}
We will use the dual space
\[
H^{-1}(\T):=\big(\widetilde{H}^1_0(\T)\big)'.
\]

\begin{defn}\label{defn-laplacian}
For $f\in H^1(\T)$ define the Laplacian $\Delta_\T f \in H^{-1}(\cT)$ of $f$ by
\begin{equation}
	\label{eq-deltat}
	\big(\Delta_\T f, g \big)_{H^{-1}(\T),\widetilde{H}^{1}_0(\T)} := - \int_\T f' g' \dd\mu\equiv - \langle f, \overline{g} \rangle_{\widetilde{H}^1(\T)} \text{ for all } g\in \widetilde{H}^1_0(\T).
\end{equation}
We also define 
\[
H^1_\Delta(\T) :=\Big\{ f \in H^1(\T):\ \Delta_\T f \in L^2(\T) \Big\},
\]
which will be equipped with the scalar product
\[
\langle f,g\rangle_{H^1_\Delta(\T)}:=\langle f,g\rangle_{H^1(\T)}+\langle \Delta_\cT f,\Delta_\cT g\rangle_{L^2(\T)}.
\]
\end{defn}

	Remark that Definition \ref{defn-laplacian} defines a weighted Laplace operator, i.e.\, it is a counterpart of an operator $\mu^{-1}\partial_x(\mu\partial_x.)$ on the real line, rather than $\partial_x^2$.
	This is explained in the following coordinate reformulation:
\begin{prop}\label{prop23}
For any $f=(f_{n,k})\in H^1(\T)$ and $h\in L^2(\cT)$ the following two con\-ditions are equivalent:
\begin{enumerate}
	\item[(a)] $\Delta_\cT f=h$,
	\item[(b)] $f''\in L^2(\T)$, and $f$ satisfies the Kirchhoff transmissition conditions
	\begin{equation} \label{kirchhoff}
		f'_{n,k}(L_{n,k}^-)\omega_{n,k} = \sum_{j=0}^{p-1} f'_{n+1,pk+j}(L_{n,k}^+)\omega_{n+1,pk+j}
	\end{equation}
	at each node $X_{n,k}$, and $h=f''$.
\end{enumerate}	
\end{prop}

\begin{proof}
Let $f\in H^1_\Delta(\cT)$. Let $g\in \Tilde H^1_c(\cT)$ with $\supp g\subset e_{n,k}$, then 
\begin{align*}
\big( \Delta_\T f, g \big)_{H^{-1}(\T),\widetilde{H}^{1}_0(\T)}=\int_{\cT}  (\Delta_\cT f)_{n,k}g_{n,k}\dd\mu&\equiv
\int_{L_{n,k}-\ell_{n,k}}^{L_{n,k}} (\Delta_\cT f)_{n,k}(t)g_{n,k}(t)\omega_{n,k}\,\dd t \\
&=  -\int_{L_{n,k}-\ell_{n,k}}^{L_{n,k}} f'(t)g'(t) \omega_{n,k}\,\dd t.
\end{align*}
Due to $g_{n,k}\in H^1_0(e_{n,k})$ this implies $(\Delta_\T f )|_{{n,k}} = f''_{n,k}$ and, hence, $\Delta_\cT f=f''$.
Now let $g\in \Tilde H^1_c(\cT)$ be only supported on the edges incident to $X_{n,k}$, then the integration by parts
yields
\begin{align*}
	\big( \Delta_\T f, &g \big)_{H^{-1}(\T),\widetilde{H}^{1}_0(\T)}\equiv
	\int_{L_{n,k}-\ell_{n,k}}^{L_{n,k}}f''_{n,k}(t)g_{n,k}(t)\omega_{n,k}\dd t\\
	&\quad+ \sum_{j=0}^{p-1} \int_{L_{n+1,pk+j}-\ell_{n+1,pk+j}}^{L_{n+1,pk+j}} f''_{n+1,pk+j}(t) g_{n+1,pk+j}(t)\omega_{n+1,pk+j}\dd t\\
	&=-\int_{L_{n,k}-\ell_{n,k}}^{L_{n,k}}f'_{n,k}(t)g'_{n,k}(t)\omega_{n,k}\dd t\\
	&\quad- \sum_{j=0}^{p-1} \int_{L_{n+1,pk+j}-\ell_{n+1,pk+j}}^{L_{n+1,pk+j}} f'_{n+1,pk+j}(t) g'_{n+1,pk+j}(t)\omega_{n+1,pk+j}\dd t\\
	&\quad +\Big(f'_{n,k}(L_{n,k}^-)\omega_{n,k} - \sum_{j=0}^{p-1} f'_{n+1,pk+j}(L_{n,k}^+)\omega_{n+1,pk+j}\Big)g(X_{n,k})\\
	&=-\int_\cT f' g'\dd\mu +\Big(f'_{n,k}(L_{n,k}^-)\omega_{n,k} - \sum_{j=0}^{p-1} f'_{n+1,pk+j}(L_{n,k}^+)\omega_{n+1,pk+j}\Big)g(X_{n,k}),
\end{align*}
and due to the arbitrariness of $g(X_{n,k})$ we arrive at the condition \eqref{kirchhoff}. This shows that (a) implies (b), and the reciprocate implication follows directly using the integration by parts.
\end{proof}

Using the scalar product $\langle\cdot,\cdot\rangle_{\widetilde{H}^1(\T)}$ from \eqref{skal1} 
we arrive at the orthogonal direct sum decomposition
\begin{equation}
	\label{h1decomp}
\widetilde{H}^1(\T) = \widetilde{H}^1_0(\T) \oplus \big( \Tilde H^1(\cT)\cap \ker\Delta_\cT\big).
\end{equation}
By Lemma \ref{lem:func_spaces}, this decomposition is non-trivial. 
Furthermore, one has the following result:
\begin{lem}\label{lem24}
	For any $h\in L^2(\cT)$ there is a unique $u\in \Tilde H^1_0(\cT)$ with $\Delta_\cT u=h$,
	and
	\[
	L^2(\cT)\ni h\mapsto u\in \Tilde H^1_0(\cT)
	\]
	is a bounded linear map.
\end{lem}

\begin{proof}
By definition, a function $u\in \Tilde H^1_0(\cT)$ satisfies $\Delta_\cT u=h$
if and only if
\begin{align}
\label{eq:hg}
\langle h,g\rangle_{L^2(\cT)}\equiv
\big(h, \overline{g} \big)_{H^{-1}(\T),\widetilde{H}^{1}_0(\T)} = - \langle u, g\rangle_{\widetilde{H}^1(\T)} \text{ for all } g\in \widetilde{H}^1_0(\T).
\end{align}
Due to Lemma \ref{lemcomp} the left-hand side defines a continuous anti-linear functional on $\Tilde H^1(\cT)$ with respect to $g$,
and the existence and the uniqueness of $u$ follow by the Riesz representation theorem.
Choosing in \eqref{eq:hg} $g=u$ yields the upper bound
\begin{align*}
	\|u\|_{\Tilde H^1(\cT)}^2=-\langle h, u\rangle_{L^2(\cT)}\leq \|h\|_{L^2(\cT)}\|u\|_{L^2(\cT)}\leq C_0\|h\|_{L^2(\cT)}\|u\|_{\Tilde H^1(\cT)},
\end{align*} 
where the constant $C_0$ is from  \eqref{poincare}. This implies the boundedness of the map $h\mapsto u$. 
\end{proof}

\subsection{Multi-scale boundary decomposition and trace map}\label{sec23}
This section is devoted to the introduction of the notion of a trace on functions from $\mathcal{T}$. As discussed in the introduction, we will define a trace operator with values in $L^2(\Gamma)$, and provide tools necessary to analyze its properties. 
It will be convenient to denote
	\[
	d:=m-1
	\]
	then $\Gamma$ is a compact $d$-dimensional Riemannian manifold (see the introduction).
	We will need a special decomposition of $\Gamma$. To define it, we start with its counterpart in the Euclidean case; for a Lebesgue-measurable set $U\subset \mathbb{R}^d$ we denote by $|U|$ its $d-$dimensional Lebesgue measure. 
	
	\begin{defn}[Multiscale decomposition, Euclidean case]
		Let $U\subset\R^d$ be a bounded open set. A \emph{regular strongly balanced $p$-multiscale decomposition} of $U$ is a collection
		\[
		(U_{n,k})_{n\in\N_0,\, k\in\{0,\dots,p^n-1\}}
		\]
		of non-empty subsets $U_{n,k}$ of $U$ such that
		\begin{enumerate}
			\item $U_{0,0}=U$:
			\item For any $n\in\N_0$ the sets $U_{n,0}, \dots, U_{n,p^{n}-1}$ are disjoint-
			\item For any $n\in \N_0$ and $k\in\{0,\dots,p^n-1\}$ one has
			\begin{gather*}
				U_{n+1,pk+j}\subset U_{n,k} \text{ for any $j\in\{0,\dots,p-1\}$,}\quad
				\Big| U_{n,k}\setminus \bigcup\limits_{j=0}^{p-1} U_{n+1,pk+j}\Big|=0.
			\end{gather*}
			\item $|U_{n,k}|=\dfrac{|U|}{p^n}$ for all $n\in \N_0$ and $k\in\{0,\dots,p^n-1\}$.
			\item  There is $c_1>0$ such that for all $n\in\N_0$ and $k\in\{0,\dots,p^n-1\}$ one has
			\[
			\diam U_{n,k} \leq c_1 p^{-\frac{n}{d}}.
			\]
			\item There is $c_2>0$ such that for all $h\in \R^d$, $n\in\N_0$, $k\in\{0,\dots,p^n-1\}$ one has
			\[
			\big|U_{n,k}\setminus(U_{n,k}+h)\big| \leq c_2 |h|p^{-\frac{n(d-1)}{d}}.
			\]
		\end{enumerate}
	\end{defn}

The above conditions can be viewed as an hierarchical decomposition procedure: One sets $U_{0,0}:=U$, and if for some $n$ all $U_{n,k}$ are already constructed, then
one decomposes each $U_{n,k}$ (up to zero measure sets) into $p$ disjoint pieces $U_{n+1,pk+j}$, $j\in\{0,\dots,p-1\}$. The last three conditions control, in a sense, the size and the shape of $U_{n,k}$.

By using a cover of $\Gamma$ by local charts, we can define a multiscale decomposition on $\Gamma$. Here, given a subset $\Gamma'\subset \Gamma$, we denote by $|\Gamma'|$ its hypersurface measure. 
\begin{defn}[Multiscale decomposition, manifold case]\label{defn2}
A collection
\[
(\Gamma_{n,k})_{n\in\N_0,\,k=0,\dots,p^n-1}
\]
of subsets $\Gamma_{n,k}\subset\Gamma$
is called a \emph{regular strongly balanced regular $p$-multiscale decomposition} of $\Gamma$,
if the following conditions hold:
	\begin{enumerate}
		\item $\Gamma_{0,0}=\Gamma$.
		\item For any $n\in\N_0$ the sets $\Gamma_{n,0}, \dots, \Gamma_{n,p^{n}-1}$ are mutually disjoint.
		\item For any $n\in \N_0$ and $k\in\{0,\dots,p^n-1\}$ one has
		\begin{gather*}
			\Gamma_{n+1,pk+j}\subset \Gamma_{n,k} \text{ for any $j\in\{0,\dots,p-1\}$,}\quad
			\Big| \Gamma_{n,k}\setminus \bigcup\limits_{j=0}^{p-1}\Gamma_{n+1,pk+j}\Big|=0.
		\end{gather*}
		\item There is $N_0\in\N_0$ such that for all $K_0\in\{0,\dots,p^{N_0}-1\}$ the following conditions are satisfied:
		\begin{enumerate}
		\item $|\Gamma_{N_0,K_0}|=p^{-N_0} |\Gamma|$.
		\item The closure $\overline{\Gamma_{N_0,K}}$ is covered by a local chart
		$\Phi_{N_0,K_0}$ on $\Gamma$ such that the sets
		\[
		\Tilde\Gamma_{N_0,K_0}:=\Phi_{N_0,K_0}^{-1}(\Gamma_{N_0,K_0})
		\]
		are bounded open sets with Lipschitz boundaries in $\R^d$.
		\item The sets
		\[
		\Tilde \Gamma^{N_0,K_0}_{n,k}:=\Tilde \Gamma_{N_0+n,p^n K_0+k}:=\Phi_{N_0,K_0}( \Gamma_{N_0+n,p^n K_0+k}), \ n\in\N_0, \ k\in\{0,\dots,p^n-1\},
		\]
		form a regular strongly balanced $p$-multiscale decomposition of $\Tilde\Gamma_{N_0,K_0}$.
		\end{enumerate}
	\end{enumerate}
\end{defn}

For the rest of the paper we pick a collection $(\Gamma_{n,k})$ as in Definition~\ref{defn2}. Note the existence of such collections is proved in \cite[Example 5.4]{Embedded Trace}.

Let us define the parameter
\begin{equation}
	\label{eq-sigma}
	\sigma \equiv \sigma(\cT):= \frac{1}{2}\left(1- \frac{\log(\ell)-\log(\omega)}{\log(p)}\right),
\end{equation}
which is strictly positive due to \eqref{eq-lpa}. For what follows we additionally assume
\begin{equation}
	\label{eq-sigma2}
	\sigma d < \frac{1}{2}.
\end{equation}	

The following theorem was the main result of the paper~\cite{Embedded Trace}:

\begin{thm}[Dirichlet trace map on $\cT$]\label{thm-trace} The linear map
	\begin{align*}
		\gamma_0^\T: H^1(\T) &\longrightarrow H^{\sigma d}(\Gamma), \\
		f &\mapsto \lim_{N \rightarrow \infty} \sum_{K=0}^{p^N-1}f(X_{N,K}) \one_{\Gamma_{N,K}},
	\end{align*}
	with the limit being taken in $H^{\sigma d}(\Gamma)$, is a well-defined bounded surjective linear operator,
	and
	\[
	\ker \gamma_0^\T = H^1_0(\T).
	\]
\end{thm}
As a corollary one easily obtains
\begin{equation}
	\label{eq-hhh}
	\Tilde H^1(\cT)\cap \ker \gamma_0^\T =\Tilde H^1_0(\cT).
\end{equation}

For what follows it will be useful to have an explicit right inverse of $\gamma^\cT_0$ and to revise some constructions related to fractional Sobolev spaces. To do so, let us introduce some preliminary notation. First of all, given $N\in\N_0$ denote 
\begin{equation}
\begin{aligned}
	\label{pvn}
V_N(\Gamma)&:=\mathop{\mathrm{span}}\big\{\one_{\Gamma_{N,K}}:\ K\in\{0,\dots,p^N-1\}\big\},\\
P_N&:=\,\text{the orthogonal projector }L^2(\Gamma)\to V_N(\Gamma).
\end{aligned}
\end{equation}
As $V_N(\Gamma)\subset V_{N+1}(\Gamma)$ for any $N$, it follows that
\begin{align}
\label{eq:commutativity}
P_N P_n=P_{\min\{N,n\}} \text{ for any }N,n\in\N_0,\text{ which implies that }P_NP_n=P_nP_N.
\end{align}
Recall that for any $r\in(0,\frac{1}{2})$ one has the equality
\[
H^r(\Gamma) = \Big\{ f\in L^2(\Gamma):\  (p^{\frac{nr}{d}}\|f-P_nf\|_{L^2(\Gamma)}) \in \ell^2  \Big\}
\]
with an equivalent norm given by
\begin{equation}
	\label{norm-ar}
\|f\|^2_{A^r(\Gamma)}:=\|P_0 f\|^2_{L^2(\Gamma)}+\sum_{n=0}^\infty p^{\frac{2nr}{d}}\|f-P_n f\|^2_{L^2(\Gamma)},
\end{equation}
see \cite[Sec.~3]{Embedded Trace}. 
In particular, for any $N\to\infty$ one has
$P_N \big( H^r(\Gamma)\big)\subset H^r(\Gamma)$, and
we also recall that
\begin{equation}
	\label{l2conv}
\|f-P_N f\|_{L^2(\Gamma)}\xrightarrow{N\to\infty}0 \text{ for any } f\in L^2(\Gamma).
\end{equation}
Let us now argue that the above holds true with $L^2(\Gamma)$ replaced by $H^r(\Gamma)$. Indeed, provided $f\in H^r(\Gamma)$, by \eqref{norm-ar} it holds that 
\[
\|f-P_N f\|^2_{A^r(\Gamma)}=
\|P_0 (f-P_N f)\|^2_{L^2(\Gamma)}+\sum_{n=0}^\infty p^{\frac{2nr}{d}}\big\|(I-P_n)(f-P_Nf)\big\|^2_{L^2(\Gamma)}.
\]
Remark that each summand on the right-hand side converges to $0$ for $N\to\infty$, as the operators $P_0$ and $I-P_n$ are bounded and \eqref{l2conv} holds. Moreover, for each $n$ one has, with \eqref{eq:commutativity}, 
\[
\big\|(I-P_n)(I-P_N)f \big\|^2_{L^2(\Gamma)}=\big\|(I-P_N)(I-P_n)f \big\|^2_{L^2(\Gamma)}\le \|f-P_n f\|^2_{L^2(\Gamma)}.
\]
Hence, the dominated convergence theorem shows
\begin{equation}
	\label{pnf}
	P_N f\xrightarrow{N\to\infty}f \text{ in $H^r(\Gamma)$ for any }f\in H^r(\Gamma).
\end{equation}
Now we have all the necessary prerequisites to state the following result, which constructs an explicit $\Tilde H^1(\cT)$ lifting of an element of  $H^{\sigma d}(\Gamma)$. 
\begin{lem}\label{lem-vg}
Let $g\in H^{\sigma d}(\Gamma)$ and denote
	\[
	g_{n,k}:=\dfrac{1}{|\Gamma_{n,k}|}\int_{\Gamma_{n,k}}g\,\dd x.
	\]
Let $v:\cT\to\C$ be linear	on each edge $e_{n,k}$ with
\[
v(o)=0,\quad v(X_{n,k})=g_{n,k} \text{ for all }(n,k),
\]
then $v\in \Tilde H^1(\cT)$ with $\gamma^\cT_0 v=g$. Moreover, the linear mapping $H^{\sigma d}(\Gamma)\ni g\mapsto v\in \widetilde{H}^1(\mathcal{T})$ is bounded.  
\end{lem}	
\begin{proof}
To prove the above lemma, we need to verify that $v$ as defined in the statement of the lemma belongs to the space $\widetilde{H}^1(\mathcal{T})$, and that its trace is given by $g$.

Step 1: Estimating  $\|v'\|_{L^2(\T)}$. We have
\begin{equation}
	\label{vvv}
	\begin{aligned}
		v_{0,0}(t)&=g_{0,0}\dfrac{t}{\ell_{0,0}}, \quad v'_{0,0}(\cdot)=\dfrac{g_{0,0}}{\ell_{0,0}},\\
		v_{n+1,pk+j}(t)&=g_{n,k}+(g_{n+1,pk+j}-g_{n,k})\dfrac{t-L_{n,k}}{\ell_{n+1,pk+j}},\\
		v'_{n+1,pk+j}(\cdot)&=\dfrac{g_{n+1,pk+j}-g_{n,k}}{\ell_{n+1,pk+j}},
	\end{aligned}
\end{equation}
hence,
\begin{align*}
	\|v'\|^2_{L^2(\cT)}&=\dfrac{\omega_{0,0}}{\ell_{0,0}}|g_{0,0}|^2+\sum_{n=0}^\infty \sum_{k=0}^{p^n-1}\sum_{j=0}^{p-1}\dfrac{\omega_{n+1,pk+j}}{\ell_{n+1,pk+j}}|g_{n+1,pk+j}-g_{n,k}|^2.
\end{align*}
Using the assumption \eqref{eq-ccc} we have
\[
\dfrac{\omega_{n,k}}{\ell_{n,k}}\le C^2 \Big(\dfrac{\omega}{\ell}\Big)^n \text{ for all }(n,k),
\]
which results in
\begin{align}
	\label{eq:boundV}
\|v'\|^2_{L^2(\cT)}\le C^2\bigg[|g_{0,0}|^2+\sum_{n=0}^\infty \sum_{k=0}^{p^n-1}\sum_{j=0}^{p-1}\Big(\dfrac{\omega}{\ell}\Big)^{n+1}|g_{n+1,pk+j}-g_{n,k}|^2\bigg].
\end{align}
To estimate the left-hand side via $g$, we use an equivalent norm $\|\cdot\|_*$ in $H^{\sigma d}(\Gamma)$:
\begin{align}
	\label{eq:equiv}
\|h\|^2_*:=\|P_0 h\|^2_{L^2(\Gamma)}+\sum_{n=0}^\infty p^{2 (n+1)\sigma}\|P_n g-P_{n+1} g\|^2_{L^2(\Gamma)},
\end{align}
see \cite[Sec.~3.2 and 3.4]{Embedded Trace}. Also, due to the choice of $\sigma$ in \eqref{eq-sigma} we have
\[
p^{2\sigma}=\dfrac{\omega p}{\ell}.
\]
For the function $g$ we have, using properties of the multiscale decomposition of $\Gamma$: 
\begin{align}
	\nonumber
	P_n g&=\sum_{k=0}^{p^n-1} g_{n,k}\one_{\Gamma_{n,k}},\\
	\label{eq:P0g}
	\|P_n g\|^2_{L^2(\Gamma)}&=\sum_{k=0}^{p^n-1} |g_{n,k}|^2 |\Gamma_{n,k}|,\\
	\nonumber
	P_n g-P_{n+1}g&=\sum_{k=0}^{p^n-1} g_{n,k}\one_{\Gamma_{n,k}} - \sum_{k=0}^{p^n-1}
	\sum_{j=0}^{p-1} g_{n+1,pk+j}\one_{\Gamma_{n+1,pk+1}},\\
	\nonumber
	&=\sum_{k=0}^{p^n-1}\sum_{j=0}^{p-1} (g_{n,k}-g_{n+1,pk+j})\one_{\Gamma_{n+1,pk+j}},\\
	\label{eq:Png}
	\|P_n g-P_{n+1}g\|^2_{L^2(\Gamma)}&=\sum_{k=0}^{p^n-1}\sum_{j=0}^{p-1}
	|g_{n,k}-g_{n+1,pk+j}|^2|\Gamma_{n+1,pk+1}|.
\end{align}
Due to the assumptions on $(\Gamma_{n,k})$ we can  find a constant $c_0>0$ such that $|\Gamma_{n,k}|\ge c_0p^{-n}$ for all $(n,k)$. 
From \eqref{eq:P0g} and \eqref{eq:Png} it follows that
\begin{align}
	\label{eq:firstg}
 \sum\limits_{k=0}^{p^n-1}|g_{n,k}|^2&\leq c_0^{-1}p^{n}\|P_ng\|_{L^2(\Gamma)}^2,\\
	\label{eq:secondg}
\sum_{k=0}^{p^n-1}\sum_{j=0}^{p-1}
	|g_{n,k}-g_{n+1,pk+j}|^2&\leq c_0^{-1}p^{n+1}\|P_n g-P_{n+1}g\|^2_{L^2(\Gamma)}.
\end{align}
Plugging in the above two bounds into \eqref{eq:boundV} yields 
\begin{align*}
	\|v'\|^2_{L^2(\mathcal{T})}\leq C^2c_0^{-1}\left(\|P_0g\|^2+\sum\limits_{n=0}^{\infty}\left(\frac{\omega p}{\ell}\right)^{n+1}\|P_{n}g-P_{n+1}g\|^2_{L^2(\Gamma)}\right)\leq C^2 c_0^{-1}\|g\|_*^2,
\end{align*}
by definition \eqref{eq:equiv}. 

Step 2: Estimating  $\|v\|_{L^2(\T)}$. From the explicit expressions \eqref{vvv} for $v$ we see that 
\begin{align*}
	\|v_{0,0}\|^2_{L^2(0,\ell_{0,0})}&\le |g_{0,0}|^2\ell_{0,0},\\
	\|v_{n+1,pk+j}\|^2_{L^2(L_{n,k},L_{n+1,pk+j})}&\le 2\big(|g_{n,k}|^2+|g_{n+1,pk+j}-g_{n,k}|^2\big)\ell_{n+1,pk+j}.
\end{align*}
Hence 
\begin{multline*}
	\|v\|^2_{L^2(\T)}\le |g_{0,0}|^2\ell_{0,0}\omega_{0,0}\\
	+2\sum_{n=0}^\infty\sum_{k=0}^{p^n-1}\sum_{j=0}^{p-1}\big(|g_{n,k}|^2+|g_{n+1,pk+j}-g_{n,k}|^2\big)\ell_{n+1,pk+j}\omega_{n+1,pk+j}.
\end{multline*}
With the help of \eqref{eq-lpa} and \eqref{eq-ccc} we estimate $\ell_{n,k}\omega_{n,k}\le C^2 (\ell\omega)^{n}$ for all $(n,k)$, and the above further rewrites 
\begin{multline*}
	\|v\|^2_{L^2(\T)} 
	\le C^2\big( |g_{0,0}|^2+2p\sum_{n=0}^\infty\sum_{k=0}^{p^n-1}(\ell\omega)^{n+1}|g_{n,k}|^2\\
	+2\sum_{n=0}^\infty(\ell\omega)^{n+1}\sum_{k=0}^{p^n-1}\sum_{j=0}^{p-1}|g_{n+1,pk+j}-g_{n,k}|^2\big).
\end{multline*}
We combine \eqref{eq:firstg} and \eqref{eq:secondg} to obtain 
\begin{multline*}
\|v\|^2_{L^2(\T)} 
	\le C^2c_0^{-1}\big( \|P_0g\|^2+2\sum_{n=0}^\infty (\ell\omega p)^{n+1}\|P_ng\|^2_{L^2(\Gamma)}\\+2\sum_{n=0}^\infty(\ell\omega p)^{n+1}\|P_ng-P_{n+1}g\|^2_{L^2(\Gamma)}\big).
\end{multline*}
The assumption \eqref{eq-lpa} $\ell \omega p <1$ and the inequality $\|P_n g\|_{L^2(\Gamma)}\leq \|g\|_{L^2(\Gamma)}$ allow to bound the first sum in the right-hand side by $\|g\|_{L^2(\Gamma)}$. As for the second sum, since $\ell\omega p<p^{2\sigma}$, we bound it by the norm \eqref{eq:equiv}. This yields the desired bound:
\begin{align*}
	\|v\|^2_{L^2(\T)}\leq C^2c_0^{-1}(2\|g\|_{L^2(\Gamma)}^2+2\|g\|_{*}^2).
\end{align*}

The results of the steps 1 and 2 show that $\|v\|_{H^1(\T)}\leq \widetilde{C}\|g\|_{H^{\sigma d}(\Gamma)}$ with some  $\widetilde{C}>0$ independent of $g$. 

Step 3: Computing the trace of $v$. 
By construction for each $n\in\N$ it holds that
\[
\sum_{k=0}^{p^n-1} v(X_{n,k})\one_{\Gamma_{n,k}}=\sum_{k=0}^{p^n-1} g_{n,k}\one_{\Gamma_{n,k}}=P_n g.
\]
For $n\to\infty$ the left-hand side converges in $H^{\sigma d}(\Gamma)$ to $\gamma^\cT_0 v$ (see Theorem \ref{thm-trace}), while the right-hand side converges to $g$ as shown in \eqref{pnf}, which gives the sought property $\gamma^\cT_0 v=g$.
\end{proof}

\subsection{Normal derivative and Dirichlet-to-Neumann map}\label{sec24}

We will start with a simple boundary value problem for the Laplacian $\Delta_\cT$.
\begin{lem} \label{DPLEwell}
For any $g\in H^{\sigma d}(\Gamma)$ there is a unique solution $u=u_g$
of the Dirichlet problem
\begin{equation}
	 \label{eq-dirichlet}
\begin{aligned}
	\Delta_\T u &= 0, &
	\gamma^\T_0 u &= g, & u\in \Tilde H^1(\cT),
\end{aligned}
\end{equation}
and the map
\begin{equation}
	\label{eq-pt}
P_\cT: H^{\sigma d}(\Gamma)\ni g\mapsto u_g\in \Tilde H^1(\cT)
\end{equation}
is a bounded linear operator with $\Delta_\cT P_\cT=0$ and $\gamma^\cT_0 P_\cT=\Id$.
\end{lem}

\begin{proof}
In view of the decomposition \eqref{h1decomp} and the identity \eqref{eq-hhh} the map
\[
\gamma_0^\cT:\ \Tilde H^1(\cT)\cap \ker\Delta_\cT\to H^{\sigma d}(\Gamma)
\]
is a bounded bijective linear operator, hence, an isomorpshism by the closed graph theorem,
and its inverse is exactly the map $P_\cT$.
\end{proof}

\begin{defn}
For $u\in H^1_{\Delta}(\T)$, its normal derivative $\gamma^\cT_1 u\in H^{-\sigma d}(\Gamma)$ on $\Gamma$
is defined by the duality product:
	\begin{equation}
		\label{deriv}
	\big(\gamma^{\T}_1 u, \gamma^\cT_0 v\big)_{H^{-\sigma d}(\Gamma), H^{\sigma d}(\Gamma)}:=\int_\cT (\Delta_\T u) v\,\dd\mu+\int_\cT u'v'\,\dd\mu \text{ for all } v \in \Tilde H^1(\cT).
	\end{equation} 
\end{defn}

Remark that with the above choice of the duality product, for $g\in H^{-\sigma d}(\Gamma)$, $H^{\sigma d}(\Gamma)\ni\varphi\mapsto (g, \varphi)_{H^{-\sigma d}(\Gamma), H^{\sigma d}(\Gamma)}$ is a linear, rather than antilinear, form.

\begin{lem}\label{lem18}
	The map
	\[
	H^1_\Delta(\cT)\ni u\mapsto \gamma^\cT_1 u \in H^{-\sigma d}(\Gamma)
	\]
	is a well-defined bounded linear operator.
\end{lem}

\begin{proof}
	For $\gamma_0^\cT v=0$ in \eqref{deriv} one has $v\in\Tilde H^1_0(\cT)$, and the right-hand side is zero due to the definition of $\Delta_\cT$, so $\gamma^\cT_1 u$ is well-defined, and its linearity is clear.	It remains to show the continuity properties of $\gamma_1^{\T}$. 
	 With the help of the map $P_\cT$ from \eqref{eq-pt}, for any $g\in H^{\sigma d}(\cT)$ one has
	\begin{align*}
	\Big| \big(\gamma^{\T}_1 u, g\big)_{H^{-\sigma d}(\Gamma), H^{\sigma d}(\Gamma)}\Big|&\equiv 
	\Big|\big(\gamma^{\T}_1 u, \gamma^\cT_0 P_\cT g\big)_{H^{-\sigma d}(\Gamma), H^{\sigma d}(\Gamma)}\Big|\\
	&=\Big| \int_\cT (\Delta_\T u) P_\cT g\,\dd\mu+\int_\cT u' (P_\cT g)'\,\dd\mu\Big|\\
	&\le \|\Delta_\cT u\|_{L^2(\cT)}\|P_\cT g\|_{L^2(\cT)}+\|u'\|_{L^2(\cT)}\|(P_\cT g)'\|_{L^2(\cT)}\\
	&\le 2 \|u\|_{H^1_\Delta(\cT)} \|P_\cT g\|_{H^1(\cT)}\\
	&\le 2 \|u\|_{H^1_\Delta(\cT)} \|P_\cT\|_{H^{\sigma d}(\Gamma)\to H^1(\cT)} \|g\|_{H^{\sigma d}(\Gamma)},
	\end{align*}
	and by taking the supremum over all $g$ with $\|g\|_{H^{\sigma d}(\Gamma)}\le 1$ we arrive at the conclusion.
\end{proof}

\begin{defn}
The Dirichlet-to-Neumann operator $\D$ for $\cT$ is defined by
	\[
	\D:=\gamma^\cT_1 P_\cT:\ H^{\sigma d}(\Gamma)\to H^{-\sigma d}(\Gamma).
	\]
By Lemmas \ref{DPLEwell} and \ref{lem18} it is a bounded linear operator.
\end{defn}

For the following we make the trivial observation that the maps $\gamma^\cT_0$, $\gamma^\cT_1$, $\Delta_\cT$ and $P_\cT$ are real, e.g. commute with the pointwise complex conjugation. The following property of the Dirichlet-to-Neumann map will be important:

\begin{thm} \label{DTNtree}
The operator $\D$ is positive and coercive, i.e. for some $c>0$ one has,
	\[
	\big( \D g, \overline{g} \big)_{H^{-\sigma d}(\Gamma), H^{\sigma d}(\Gamma)} \ge c\|g\|^2_{H^{\sigma d}(\Gamma)} \text{ for any $g\in H^{\sigma d}(\Gamma)$.}
	\] 
\end{thm}

\begin{proof}
Let $g\in H^{\sigma d}(\Gamma)$. The substitution
$u:=P_\cT g$ and $v:=P_\cT \overline{g}$ in \eqref{deriv} yields
\begin{align*}
\big( \D g, \overline{g} \big)_{H^{-\sigma d}(\Gamma), H^{\sigma d}(\Gamma)}&=\int_\cT \big|(P_\cT g)'\big|^2\dd\mu=\|P_\cT g\|^2_{\Tilde H^1(\cT)},
\end{align*}
and with $\|g\|_{H^{\sigma d}(\Gamma)}=\|\gamma^\cT_0 P_\cT g\|_{H^{\sigma d}(\Gamma)}\le \|\gamma^\cT_0\|_{\Tilde H^1(\cT)\to H^{\sigma d}(\Gamma)}\|P_\cT g\|_{\Tilde H^1(\cT)}$
we arrive at
\[
\big( \D g, \overline{g} \big)_{H^{-\sigma d}(\Gamma), H^{\sigma d}(\Gamma)}\ge\dfrac{\|g\|^2_{H^{\sigma d}(\Gamma)}}{\|\gamma^\cT_0\|_{\Tilde H^1(\cT)\to H^{\sigma d}(\Gamma)}^2}.\qedhere
\]
\end{proof}

\subsection{Boundary values on finite sections}\label{sec-fin}
The question that we study in two next sections is the following. Assume that the infinite tree $\mathcal{T}$ is replaced by a tree $\widetilde{\T}_N$ which only contains a finite number of generations, and instead of solving the boundary-value problem for the Laplace equation with the Cauchy data $\gamma_0^{\mathcal{T}}u=g$, we solve the boundary-value problem on $\widetilde{\T}_N$, by imposing on its truncated boundary $\gamma_0^{\widetilde{\T_N}}u_N=P_Ng$. Is it possible to choose $\widetilde{\mathcal{T}}_N$ so that the associated Dirichlet-to-Neumann map is suitably close to $\D P_N$? In these two sections we give a positive answer to this question (Theorem \ref{thm221}), provided some auxiliary assumptions on $\mathcal{T}$. We start by formalizing the problem.

For $N\in\N$ denote by $\cT_N\equiv\cT^p_N\big((\ell_{n,k}),(\omega_{n,k})\big)$ the finite portion of the tree $\cT$ obtained by keeping the edges $e_{n,k}$ with $n\le N$ only, with the same edge lengths $\ell_{n,k}$ and edge weights $\omega_{n,k}$, i.e.
\[
\cT_N:=\big\{(n,k,t)\in\cT:\ n\le N\big\},
\]
and for $f:\cT_N\to\C$ one defines
\[
\int_{\cT_N} f\dd\mu:=\sum_{n=0}^N\sum_{k=0}^{p^n-1} \int_{L_{n,k}-\ell_{n,k}}^{L_{n,k}}f_{n,k}(t) \omega_{n,k}\dd t,
\qquad
f_{n,k}:=f(n,k,\cdot).
\]
This induces the spaces
\begin{gather*}
	L^2(\cT_N):=\big\{f:\cT_{N}\to\C:\ \|f\|^2_{L^2(\cT_N)}:=\int_{\cT_N} |f|^2\dd\mu<\infty\big\},\\
	\langle f,g\rangle_{L^2(\cT_N)}:=\int_{\cT_N} f\overline{g}\,\dd\mu,
\end{gather*}
and
\begin{gather*}
	H^1(\cT_N):=\big\{f\in L^2(\cT_N): f \text{ is continuous with } f'\in  L^2(\cT_N)\big\},\\
	\langle f,g\rangle_{H^1(\cT_N)}:=\langle f,g\rangle_{L^2(\cT_N)}+\langle f',g'\rangle_{L^2(\cT_N)}
\end{gather*}
as well as
\[
\Tilde H^1(\cT_N):=\big\{f\in H^1(\cT_N):\ f(o)=0\big\}.
\]
As the embedding $\Tilde H^1(\cT)\hookrightarrow L^2(\cT)$ is compact, there is $f_N\in \Tilde H^1(\cT)$ with $f_N\not\equiv 0$ such that for all $f\in \Tilde H^1(\cT_N)$ with $f\not\equiv 0$ it holds
\[
a_N:=\dfrac{\|f'_N\|^2_{L^2(\cT_N)}}{\|f_N\|^2_{L^2(\cT_N)}}\le \dfrac{\|f'\|^2_{L^2(\cT_N)}}{\|f\|^2_{L^2(\cT_N)}},
\]
and $a_N>0$ as $f_N$ cannot be constant. Hence, we have the Poincar\'e inequality
\[
\|f'\|^2_{L^2(\cT_N)}\ge a_N \|f\|^2_{L^2(\cT_N)} \text{ for all } f\in \Tilde H^1(\cT_N),
\]
which shows that
\[
\langle f,g\rangle_{\Tilde H^1(\cT_N)}:=\langle f',g'\rangle_{L^2(\cT_N)}
\]
is a scalar product on $\Tilde H^1(\cT_N)$ which is equivalent to the scalar product inherited from $H^1(\cT_N)$.

Recall that the subspaces $V_N(\Gamma)\subset L^2(\Gamma)$ were defined in \eqref{pvn}.
The set of the pending nodes $X_{N,K}$ with $K\in\{0,\dots,p^N-1\big\}$ represents the natural boundary of $\cT_N$, and by using the same multiscale decomposition $(\Gamma_{n,k})$ of $\Gamma$ one defines the bounded surjective Dirichlet trace operator
\[
	\gamma^{\T_N}_0:\  \widetilde{H}^1(\T_N) \ni f\mapsto \sum_{K=0}^{p^N-1} f(X_{N,K}) \one_{\Gamma_{N,K}}\in V_N(\Gamma)
\]
and the spaces
\[
H^1_0(\cT_N):=\big\{f\in H^1(\cT_N):\ \gamma^{\cT_N}_0f=0\big\},
\qquad
\Tilde H^1_0(\cT_N):=\big\{f\in \Tilde H^1(\cT_N):\ \gamma^{\cT_N}_0f=0\big\}.
\]
Any function in $H^1_0(\cT_N)$, respectively $\Tilde H^1_0(\cT_N)$, can be extended by zero to a function in $H^1_0(\cT)$, respectively $\Tilde H^1_0(\cT)$, and this extension preserves the respective norms.

\begin{defn}
	For $f\in H^1(\T_N)$ define its Laplacian
	\[
	\Delta_{\T_N} f \in H^{-1}(\cT_N):= \big(\Tilde H^1_0(\cT_N)\big)'
	\]
	by
	\begin{equation}
		\label{eq-deltatn}
		\big(\Delta_{\cT_N} f, g \big)_{H^{-1}(\cT_N),\widetilde{H}^{1}_0(\T_N)} := - \int_{\T_N} f' g' \dd\mu\equiv - \langle f, \overline{g} \rangle_{\widetilde{H}^1(\T_N)} \text{ for all } g\in \widetilde{H}^1_0(\T_N).
	\end{equation}

	We also define 
	\[
	H^1_\Delta(\T_N) :=\Big\{ f \in H^1(\T_N):\ \Delta_{\T_N} f \in L^2(\T_N) \Big\},
	\]
	which will be equipped with the scalar product
	\[
	\langle f,g\rangle_{H^1_\Delta(\T_N)}:=\langle f,g\rangle_{H^1(\T_N)}+\langle \Delta_{\cT_N}f,\Delta_{\cT_N} g\rangle_{L^2(\T_N)}.
	\]
\end{defn}
Remark that for any $f\in H^1_\Delta(\cT)$ one also has $f\in H^1_\Delta(\cT_N)$ with $\Delta_{\cT_N} f=\Delta_\cT f$ on $\cT_N$. Similarly to Proposition~\ref{prop23} one shows:

\begin{lem}\label{prop23-tn}
	For any $f=(f_{n,k})\in H^1(\T_N)$ and $h\in L^2(\T_N)$ the following two conditions are equivalent:
	\begin{enumerate}
		\item[(a)] $\Delta_{\T_N}f=h$,
		\item[(b)] $f''\in L^2(\T_N)$, and $f$ satisfies the Kirchhoff transmission conditions
		\begin{equation} \label{kirchhoff-tn}
			f'_{n,k}(L_{n,k}^-)\omega_{n,k} = \sum_{j=0}^{p-1} f'_{n+1,pk+j}(L_{n,k}^+)\omega_{n+1,pk+j}
		\end{equation}
		at each node $X_{n,k}\in T_N$ with $n\le N-1$, and $f''=h$.
	\end{enumerate}	
\end{lem}
The assertion (b) in Lemma~\ref{prop23-tn} does not contain any condition at the pending nodes $X_{N,K}$, which is used in the following definition:

\begin{defn}
	For $u \in H^1_{\Delta}(\T_N)$ define its normal
	derivative 
	\[
	\gamma^{\T_N}_1 u \in V_N(\Gamma)\subset L^2(\Gamma)
	\]
	by
	\[
	\int_{\Gamma} (\gamma^{\T_N}_1 u) \gamma^{\T_N}_0 v \,\dd s:= \int_{\T_N}
	(\Delta_{\cT_N} u) \overline{v} \,\dd\mu  + \int_{\T_N}   u'  v' \,\dd\mu ,
	\]
	for all $v\in \Tilde H^1(\T_N)$. 
	Using Lemma \ref{prop23-tn} one directly shows
	\begin{align} \label{fnderiv} 
		\gamma_1^{\T_N} u = \sum_{K=0}^{p^N-1} \frac{\omega_{N,K}}{|\Gamma_{N,K}|} u_{N,K}'(L_{N,K}^-) \one_{\Gamma_{N,K}}. 
	\end{align}
\end{defn}
The following result can be proven literally  as Lemma \ref{lem18}. 
	 \begin{lem}\label{lem18_Tn}
		The map
		\[
		H^1_\Delta(\cT_{N})\ni u\mapsto \gamma^{\cT_N}_1 u \in H^{-\sigma d}(\Gamma)
		\]
		is a well-defined bounded linear operator.
	\end{lem}

Remark that $\gamma^{\cT_N}_0 u\xrightarrow{N\to\infty}\gamma^{\cT} u$ in $H^{\sigma d}(\Gamma)$ for any $u\in H^1(\cT)$, see Theorem~\ref{thm-trace}. Let us establish a related approximation result for $\gamma^\cT_1 u$.

\begin{lem}\label{lem-weak}
	For any $u\in H^1_\Delta(\cT)$ one has $\gamma^{\cT_N}_1 u\xrightarrow{N\to\infty}\gamma^\cT_1 u$
	weakly in $H^{-\sigma d}(\Gamma)$.
\end{lem}
\begin{proof}
	Let $u\in H^1_\Delta(\cT)$ and $g\in H^{\sigma d}(\Gamma)$. In virtue of Lemma \ref{lem-vg} there is a function $v\in \Tilde H^1(\cT)$ such that
	\[
	\gamma^\cT_0 v=g,\quad v(X_{n,k})=\dfrac{1}{|\Gamma_{n,k}|}\int_{\Gamma_{n,k}} g\,\dd s \text{ for all }(n,k).
	\]
	Then
	\begin{align*}
		(\gamma^{\cT}_1 &u,g)_{H^{-\sigma d}(\Gamma),H^{\sigma d}(\Gamma)}=
		\int_{\cT}(\Delta_\T u)v\,\dd\mu+\int_{\cT}u'v'\,\dd\mu\\
		&=\lim_{N\to\infty}\Big( \int_{\cT_N}(\Delta_{\T_N} u)v\,\dd\mu+\int_{\cT_N}u'v'\,\dd\mu\Big)
		=\lim_{N\to\infty}\int_\Gamma (\gamma^{\cT_N}_1 u)(\gamma^{\cT_N}_0 v)\,\dd s\\
		&=\lim_{N\to\infty}\int_\Gamma \bigg( \sum_{K=0}^{p^N-1} \frac{\omega_{N,K}}{|\Gamma_{N,K}|} u_{N,K}'(L_{N,K}^-) \one_{\Gamma_{N,K}}\bigg)\bigg( \sum_{K=0}^{p^N-1} \dfrac{1}{|\Gamma_{N,K}|}\int_{\Gamma_{N,K}} g\,\dd s \,\one_{\Gamma_{N,K}}\bigg)\,\dd s\\
		&= \lim_{N\to\infty}\sum_{K=0}^{p^N-1} \frac{\omega_{N,K}}{|\Gamma_{N,K}|} u_{N,K}'(L_{N,K}^-) \int_{\Gamma_{N,K}} g\,\dd s \\
		&= \lim_{N\to\infty}\int_\Gamma\bigg(\sum_{K=0}^{p^N-1} \frac{\omega_{N,K}}{|\Gamma_{N,K}|} u_{N,K}'(L_{N,K}^-) \one_{\Gamma_{N,K}} \bigg)g\,\dd s\\
		&=\lim_{N\to\infty}(\gamma^{\cT_N}_1 u,g)_{H^{-\sigma d}(\Gamma),H^{\sigma d}(\Gamma)}. \qedhere
	\end{align*}	
\end{proof}

Using the scalar product $\langle\cdot,\cdot\rangle_{\widetilde{H}^1(\T_N)}$ we obtain the orthogonal direct sum decomposition
\begin{equation}
	\label{h1decomptn}
	\widetilde{H}^1(\T_N) = \widetilde{H}^1_0(\T_N) \oplus \big( \Tilde H^1(\cT_N)\cap \ker\Delta_{\cT_N}\big).
\end{equation}
As $\widetilde{H}^1_0(\T_N)$ is exactly the kernel of the map
\[
\gamma^{\cT_N}_0:\ \Tilde H^1(\cT_N)\to V_N(\Gamma),
\]
it follows that for any $g\in V_N(\Gamma)$ there is a unique solution $u=u_g$
of the Dirichlet problem
\begin{equation}
	\label{eq-dirichlet-tn}
	\begin{aligned}
		\Delta_{\T_N} u &= 0, &
		\gamma^{\T_N}_0 u&= g, &  u\in \Tilde H^1(\cT_N),
	\end{aligned}
\end{equation}
and the Dirichlet-to-Neumann map $\D^{\cT_N}$ for $\cT_N$ is defined by
\[
\D^{\cT_N}:\ V_N(\Gamma)\ni g\mapsto \gamma^{\cT_N}_1 u_g \in V_N(\Gamma).
\]

\subsection{Approximations of the Dirichlet-to-Neumann map}\label{sec-dtn}

The purpose of this subsection is to construct a good approximation for the Dirichlet-to-Neumann map $\D$ using finite-dimensional operators. These constructions are mainly adaptations of the ideas used in~\cite{msv} for discrete dyadic trees to the case of metric trees.

In Subsection \ref{sec23} we have seen that $P_N g\xrightarrow{N\to\infty}g$ in $H^{\sigma d}(\Gamma)$
for any $g\in H^{\sigma d}(\Gamma)$. Let us show that the convergence rate can be estimated under additional regularity assumptions. Recall that the norms $\|\cdot\|_{A^r(\Gamma)}$ are defined in \eqref{norm-ar}.

\begin{lem}\label{lem-pngg}
	Let $\sigma'$ be such that
	\[
	0<\sigma d<\sigma' d<\frac{1}{2},
	\]
	then for any $N\in\N_0$ and any $g\in H^{\sigma' d}(\Gamma)$ one has
	\begin{equation}
		\label{pngg}
	\|P_Ng-g\|_{A^{\sigma d}(\Gamma)} \le \dfrac{p^{2\sigma}}{p^{2\sigma}-1}p^{-N(\sigma'- \sigma )}\|g\|_{A^{\sigma'd}(\Gamma)}.
	\end{equation}
\end{lem}

\begin{proof}
With the definition of the norm $\|\cdot\|_{A^{\sigma' d}(\Gamma)}$, the right-hand side of 
 \eqref{pngg} rewrites as
\begin{align*}
	p^{-2N(\sigma'- \sigma )} \|g\|^2_{A^{\sigma'd}(\Gamma)} &= p^{-2N(\sigma'-\sigma)}\|P_0g\|^2_{L^2(\Gamma)}+ \sum_{n=0}^\infty p^{2n\sigma' - 2N\sigma'+2N\sigma } \|g-P_ng\|^2_{L^2(\Gamma)} \\
	&\ge p^{2N\sigma} \|g-P_Ng\|^2_{L^2(\Gamma)}  +p^{2N\sigma} \sum_{n=1}^\infty p^{2n\sigma'}\|g-P_{n+N}g\|^2_{L^2(\Gamma)}=:I. 
\end{align*}

To study the left-hand side of \eqref{pngg} we note first that
	\begin{align*}
	P_0(P_N g-g)&=P_0P_N g- P_0g\\
	&=P_{\min\{N,0\}}-P_0 g=P_0 g-P_0 g=0,\\[\smallskipamount]
	(P_Ng-g) - P_n(P_Ng -g) &= P_Ng - g - P_nP_Ng + P_n g\\
	&=P_Ng - g - P_{\min\{n,N\}} + P_n g =P_{\max\{N,n\}}g-g,
	\end{align*}
which yields	
	\begin{align*}
		\|P_Ng-g\|^2_{A^{\sigma d}(\Gamma)} &= \sum_{n=0}^{N} p^{2n\sigma } \|P_Ng-g\|^2 + \sum_{n=N+1}^\infty p^{2n\sigma } \|P_ng-g\|^2_{L^2(\Gamma)} \\
		&= \frac{p^{2\sigma (N+1)}-1}{p^{2\sigma}-1} \|P_Ng-g\|^2_{L^2(\Gamma)} + p^{2N\sigma}\sum_{n=1}^{\infty} p^{2n\sigma} \|P_{n+N}g-g\|^2_{L^2(\Gamma)}.
	\end{align*}
By estimating
\[
\dfrac{p^{2\sigma (N+1)}-1}{p^{2\sigma}-1}=
\dfrac{p^{2\sigma}-\dfrac{1}{p^{2\sigma N}}}{p^{2\sigma}-1} p^{2N \sigma }
\le \dfrac{p^{2\sigma}}{p^{2\sigma}-1}p^{2N \sigma },
\qquad
p^{2n\sigma}\le \dfrac{p^{2\sigma}}{p^{2\sigma}-1}p^{2n\sigma'}
\]
we arrive at
\begin{align*}
	\|P_Ng-g\|^2_{A^{\sigma d}(\Gamma)} &\le \dfrac{p^{2\sigma}}{p^{2\sigma}-1} p^{2N\sigma} \|P_Ng-g\|^2_{L^2(\Gamma)} + \dfrac{p^{2\sigma}}{p^{2\sigma}-1}p^{2N\sigma}\sum_{n=1}^{\infty} p^{2n\sigma'} \|P_{n+N}g-g\|^2_{L^2(\Gamma)}\\
	&\le \dfrac{p^{2\sigma}}{p^{2\sigma}-1}I\le \dfrac{p^{2\sigma}}{p^{2\sigma}-1}p^{-2N(\sigma'- \sigma )} \|g\|^2_{A^{\sigma'd}(\Gamma)}.\qedhere
\end{align*}
\end{proof}

As a corollary we obtain:
\begin{cor}\label{cor219}
Let $\sigma'$ be such that
\[
\sigma d <\sigma'd<\frac{1}{2},
\]
then there is a constant $c>0$ such that for any $N\in\N_0$ it holds
	\begin{equation*}
		\| \D\circ P_N  - \D  \|_{H^{\sigma' d}(\Gamma) \to H^{-\sigma d}(\Gamma)} \le c p^{-N(\sigma'-\sigma)}.
	\end{equation*}
\end{cor}
\begin{proof}
Let us use the norm $\|\cdot\|_{H^r}:=\|\cdot\|_{A^r}$, then
Lemma \ref{lem-pngg} yields for any $g\in H^{\sigma'd}(\Gamma)$:
	\begin{align*}
		\|\D\circ P_N g-\D g \|_{H^{-\sigma d}}&=\|\D (P_N g-g)\|_{H^{-\sigma d}}
		 \leq \|\D\|_{H^{\sigma d} \to H^{-\sigma d}} \|P_N g - g \|_{H^{\sigma d}}\\ &\leq \dfrac{p^{2\sigma}}{p^{2\sigma}-1}\|\D\|_{H^{\sigma d} \to H^{-\sigma d}} p^{-N(\sigma'-\sigma)}\|g\|_{H^{\sigma'd}}. \qedhere
	\end{align*}
\end{proof}

In order to work efficiently with the ``truncated map'' $\D\circ P_N$ we have to make an additional assumption on the structure of the tree $\cT$. Recall that the subtrees $\T^j_{N,k}$ and trees $\mathbb{T}^p$ were defined in Subsection \ref{sec21}.

\begin{defn} \label{regularity}
The metric tree $\T$ is called \emph{geometric from the generation $N_1$} (with some $N_1\in\N_0$) if for any
	$k\in\{0,\dots,p^{N_1}-1\}$ and any $j\in\{0,\dots,p-1\}$ it holds
	\[
	\T_{N_1,k}^j = \TT^p( \ell_{N_1+1,pk+j}, \ell,\omega_{N_1+1,pk+j}, \omega).  
	\]
In this case for any $N\ge N_1$ we denote by $\Tilde \cT_{N+1}$ the finite metric tree which is obtained from the truncated tree $\cT_{N+1}$ in the following way:
\begin{itemize}
	\item the combinatorial structure remains unchanged,
	\item the lengths of the pending (leaf) edges $e_{N+1,pk+j}$ becomes 
	\[
	{\Tilde \ell}_{N+1,pk+j}:=\dfrac{\ell_{N+1,pk+j}}{1-\dfrac{\ell}{\omega p}},
	\]
	and these new edges will be denoted by $\Tilde e_{N+1,pk+j}$,
	\item the lengths of all other edges and the weights of all edges remain unchanged.
\end{itemize}
The tree $\Tilde\cT_{N+1}$ will be called the $(N+1)$-condensation of $\cT$.	
\end{defn} 

Recall that Definition~\ref{defn2} contains a parameter $N_0\in \N_0$ related to the chosen decom\-position $(\Gamma_{n,k})$  of $\Gamma$, as well as associated local charts $\Phi_{N_0,K_0}$ covering $\overline{\Gamma_{N_0,K_0}}$ and the induced open sets $\Tilde \Gamma_{N_0,K_0}:=\Phi^{-1}_{N_0,K_0}(\Gamma_{N_0,K_0})\subset \R^d$ for $K_0\in\{0,\dots,p^{N_0}-1\}$.
We consider the respective Jacobians
\[
J_{N_0,K_0}:\ \Tilde\Gamma_{N_0,K_0}\to\R, \quad
J_{N_0,K_0}(y):=\sqrt{\det \big(\big(D\Phi_{N_0,K_0}(y)\big)^TD\Phi_{N_0,K_0}(y)\big)},
\]
which are smooth functions, bounded and separated from zero due to the above assumptions,
and their push-forwards
\[
H_{N_0,K_0}:=J_{N_0,K_0}\circ \Phi_{N_0,K_0}^{-1}:\  \Gamma_{N_0,K_0}\to\R.
\]
\begin{thm}\label{thm221}
Let the metric tree $\cT$ be geometric from some generation $N_1$, and for any $N\ge N_1$
let
\[
\Tilde \D_{N+1}:\ V_{N+1}(\Gamma)\to V_{N+1}(\Gamma)
\]
be the Dirichlet-to-Neumann map	for its $(N+1)$-condensation $\Tilde\cT_{N+1}$ constructed as in Definition~\ref{regularity}. Then
\begin{gather}
\D  P_{N+1}=H_{N_0}\Tilde \D_{N+1} P_{N+1} \text{ for any $N\ge\max\{N_0,N_1\}$}\nonumber\\
\text{with }H_{N_0}:=\dfrac{1}{p^{N_0-N-1}}\sum_{K_0=0}^{p^{N_0}-1} \dfrac{\one_{\Gamma_{N_0,K_0}}}{|\Tilde\Gamma_{N_0,K_0}|H_{N_0,K_0}}.\label{hn0}
\end{gather}
\end{thm}	

\begin{rmk}
Before delving into the proof of the above result, let us discuss why we work with condensed trees rather than truncated trees. This is particularly easy to understand for the case when $\T$ is a geometric tree for $p=1$, $\omega=1$, $L_0=\omega_0=1$. In that case the tree $\T$ can be seen as the interval $(0, L)$ with $L=(1-\ell)^{-1}$. The traces and normal derivatives reduce to scalars, and the Dirichlet-to-Neumann map is the simple multiplication operator: $\D=L^{-1}$. The truncated tree $\T_N$ is then identified with the interval $\big(0, (1-\ell^N)L\big)$, and the associated Dirichlet-to-Neumann map is then $\D_N=(1-\ell^N)^{-1}L^{-1}$. However, the condensed tree $\widetilde{\mathcal{T}}_{N+1}$ again becomes the interval $(0, L)$, and this ensures whose Dirichlet-to-Neumann map $\Tilde D_N$
coincides with $\D$. It is a generalization of this observation that allows to express $\D P_{N+1}$ via $\widetilde{\mathcal{D}}_{N+1}$ exactly as stated in the above result. 
\end{rmk}

\begin{proof}[\bf Proof of Theorem \ref{thm221}]
We first derive an explicit expression of $\Tilde{\mathcal{D}}_{N+1}$, and next compute explicitly $\mathcal{D}P_{N+1}$.  

Step 1: Computing $\Tilde{\mathcal{D}}_{N+1}$. 
Let $f\in V_{N+1}(\Gamma)$, then it is given by
\[
f=\sum_{K=0}^{p^{N+1}-1} f_{N+1,K} \one_{\Gamma_{N+1,K}},
\quad
f_{N+1,K}=\dfrac{1}{|\Gamma_{N+1,K}|}\int_{\Gamma_{N+1,K}}f_{N+1}\,\dd x.
\]	
Denote by $u\in \Tilde H^1(\Tilde \cT_{N+1})$ the uniquely defined solution of
\[
\Delta_{\Tilde \cT_{N+1}}u=0,\qquad 
\gamma^{\Tilde \cT_{N+1}}_0u=f.
\]
For $K\in\{0,\dots,p^N-1\}$ and $j\in\{0,\dots,p-1\}$ denote
\begin{align*}
b_{N,K}&:=u(X_{N,K}),&
c_{N+1,pK+j}&:=u'_{N+1,pK+j}(L_{N,K}^{+}).
\end{align*}
The harmonicity of $u$ on the modified edges $\Tilde e_{N+1,pK+j}$ implies its linearity,
\begin{align*}
u_{N+1,pK+j}(t)&=u_{N+1,pK+j}(L_{N,k})+c_{N+1,pK+j}(t-L_{N,K})\\
&\equiv b_{N,K}+c_{N+1,pK+j}(t-L_{N,K}) \text{ for all } t\in(L_{N,K},L_{N,K}+\Tilde \ell_{N+1,pK+j}),
\end{align*}
hence,
\begin{align}
	\nonumber
\gamma^{\Tilde\cT_{N+1}}_0u&=\sum_{K=0}^{p^N-1}\sum_{j=0}^{p-1} u_{N+1,pK+j}(L_{N,K}+\Tilde \ell_{N+1,pK+j})\one_{\Gamma_{N+1,pK+j}}\\
\label{eq:trace_u}
&=\sum_{K=0}^{p^N-1} \sum_{j=0}^{p-1} \Big( b_{N,K}+c_{N+1,pK+j} \Tilde \ell_{N+1,pK+j}\Big)\one_{\Gamma_{N+1,pK+j}}.
\end{align}
The boundary condition for $u$ shows that
\begin{equation}
	\label{bcf}
b_{N,K}+\sum_{j=0}^{p-1} c_{N+1,pK+j} \Tilde\ell_{N+1,pK+j}=f_{N+1,pK+j},\ K\in\{0,\dots,p^N-1\},\ j\in\{0,\dots,p-1\},
\end{equation}
and in addition we have, with the use of the explicit expression \eqref{fnderiv} for $\gamma_1^{\mathcal{T}_N}$,
\begin{equation}
	\label{tildedn}
\Tilde\D_{N+1}f\equiv
\gamma^{\Tilde \cT_{N+1}}_1 u=\sum_{K=0}^{p^N-1}\sum_{j=0}^{p-1} \dfrac{\omega_{N+1,pK+j} c_{N+1,pK+j}}{|\Gamma_{N+1,pK+j}|}\one_{\Gamma_{N+1,pK+j}}. 
\end{equation}

Step 2: Computing $\mathcal{D}P_{N+1}$. To compute the desired expression, we will construct an appropriate ansatz to the solution of the Dirichlet boundary problem on $\T$.
	
Step 2.1: Ansatz for the solution.
Let $v:\cT\to\C$ be defined as follows: we set $v:=u$ on $\cT_N$. Next, on $e_{N+1,pK+j}$, we define 
\[
v_{N+1,pK+j}(t):=b_{N,K}+c_{N+1,pK+j}(t-L_{N,K}) \text{ for all } j\in\{0,\dots,p-1\},
\]
which ensures that the transmission condition \eqref{kirchhoff} is fulfilled at each $X_{N,K}$,
and then extend it to a continuous function on the whole of $\cT$ such that $v$ is radial along each subtree $\cT^j_{N,K}$ (which means that for all $x\in\cT^j_{N,K}$ the value of $v(x)$ only depends on the distance between $X_{N,K}$ and $x$), linear on each edge $e_{n,k}\in \cT^j_{N,K}$ and the transmission conditions \eqref{kirchhoff} are satisfied at all nodes $X_{n,k}\in \cT^j_{N,K}$. Note that such an extension is unique and is given by the following expressions:
\begin{align*}
v=u^0+u^f,\quad \text{ where } u^0_{n,k}&=b_{N,K}r^{n-N-1}\Big(1+(r-1)\frac{t-L_{n,k}+\ell_{n,k}}{\ell_{n,k}}\Big), \\
\text{and }
u^f_{n,k}&=\alpha_{N,K,j}\left(\sum\limits_{s=0}^{n-N-2}r^s+r^{n-N-1}\frac{t-L_{n,k}+\ell_{n,k}}{\ell_{n,k}}\right), \\ \text{ with }\quad r&=\frac{\ell}{p\omega},\quad \alpha_{N,K,j}= c_{N+1,pK+j}\ell_{N+1,pK+j}-b_{N,K}(r-1).
\end{align*}
By \cite[Theorem 4.4]{jks}, it follows that $u^0$ is a harmonic function with $\gamma^{\T}_0u^0=0$ which equals to $b_{N,K}$ in $X_{N,K}$. At the same time, $u^f$ is a harmonic function with $\gamma^{\T}_0u^f=f$ that vanishes in $X_{N,K}$ (and hence the notation). Let us prove that indeed this decomposition defines a function in ${H}^1(\cT_{N,K}^j)$.  

Step 2.2: Let us show that $v\in \widetilde{H}^1(\cT)$.
As $r=\frac{\ell}{p\omega}<1$ by assumption \eqref{eq-lpa}, we conclude that 
\begin{align}
	\label{eq:boundedness}
	\begin{split}
	&\|u^0_{n,k}\|_{L^{\infty}(\mathcal{T}^j_{N,K})}\leq |b_{N,K}|(1+(r-1)), \\ 	&\|u^f_{n,k}\|_{L^{\infty}(\mathcal{T}^j_{N,K})}\leq |\alpha_{N,K,j}|\left((1-r)^{-1}+1\right).
	\end{split}
\end{align}
Using the assumptions \eqref{eq-lpa} and \eqref{eq-ccc} we see that constant functions are in $L^2(\mathcal{T})$: 
\begin{align*}
	\int_{\cT}\dd{\mu}=\sum\limits_{n=0}^{\infty}\sum\limits_{k=0}^{p^n-1}\ell_{n,k}\omega_{n,k}\leq C^2\sum\limits_{n=0}^{\infty}\sum\limits_{k=0}^{p^n-1}\ell^{n}\omega_{n}\leq  \sum\limits_{n=0}^{\infty}C^2(\ell\omega p)^n<\infty.
\end{align*}
Therefore, the above together with \eqref{eq:boundedness} yields that $v\in L^2(\cT)$. Next, let us remark that for $e_{n,k}\in \cT^{j}_{N,k}$,  
\begin{align}
	\nonumber
	v'_{n,k}&=(b_{N,K}(r-1)+\alpha_{N,K,j})r^{n-N-1}\ell_{n,k}^{-1}\\
	\nonumber
	&=(b_{N,K}(r-1)+\alpha_{N,K,j})\ell_{N+1,pK+j}^{-1}r^{n-N-1}\ell^{N-n+1}\\
	\label{eq:vprime_nk}
	&=c_{N+1,pk+j}r^{n-N-1}\ell^{N-n+1}=c_{N+1,pk+j}(p\omega)^{-n+N+1}.
\end{align} 
A direct computation yields
\begin{align*}
	\int_{\cT^j_{N,K}}|v'|^2\dd\mu&=|c_{N+1,pk+j}|^2\sum\limits_{n=N+1}^{\infty}\sum\limits_{k=0}^{p^n-1}\ell_{n,k}\omega_{n,k}(p\omega)^{2(N+1-n)}\\
	&=|c_{N+1,pk+j}|^2\ell_{N+1,pK+j}\omega_{N+1,pK+j}\sum\limits_{n=N+1}^{\infty}p^n\ell^{n-N-1}\omega^{n-N-1}(p\omega)^{2(N+1-n)}\\
	&\leq C_{N,K,j}\sum\limits_{n=N+1}^{\infty}\ell^n(p\omega)^{-n}=C_{N,K,j}\sum\limits_{n=N+1}^{\infty}r^n<\infty \text{ (due to $|r|<1$).} 
\end{align*}
Recall that $\Delta_{\mathcal{T}}v=0$ by construction, and it follows that $v\in \widetilde{H}^1_\Delta(\cT)$. 

Step 2.3: Let us show that $\gamma_0^{\T}v=f$. First of all, we remark that for all $n\geq N+1$
with $e_{n,k}\in \mathcal{T}_{N,K}^j$ we have
\begin{align}
	\label{eq:vval}
	\begin{split}
	v(X_{n,k})&=v_{n,k}(L_{n,k})=b_{N,K}r^{n-N-1}+\alpha_{N,K,j}\sum\limits_{s=0}^{n-N-2}r^s\\
	&=r^n\Big(b_{N,K}r^{-N-1}-r^{-N-1}(1-r)^{-1}\alpha_{N,K,j}\Big)+\alpha_{N,K,j}(1-r)^{-1},
	\end{split}
\end{align}
and the right-hand part does not depend on $k$ due to the radiality of $v$ on $\cT^j_{N,K}$.
Therefore, for $n\geq N+1$,
\begin{align*}
	\gamma_0^{\mathcal{T}_n}v&=\sum\limits_{k=0}^{p^n-1}v(X_{n,k}){\one}_{\Gamma_{n,k}}=\sum\limits_{k=0}^{p^n-1}\varphi_{n,N,K,j}{\one}_{\Gamma_{n,k}}
	=\sum\limits_{K=0}^{p^N-1}\sum\limits_{j=0}^{p-1}\varphi_{n,N,K,j}{\one}_{\Gamma_{N+1,pK+j}}.
\end{align*}
Next, we take $\lim\limits_{n\rightarrow +\infty}$ of the above; with \eqref{eq:vval} and $r<1$, as well as using the definition of $\alpha_{N,K,j}$ we conclude that 
\begin{align*}
	\gamma_0^{\mathcal{T}}v&=\sum\limits_{K=0}^{p^N-1}\sum\limits_{j=0}^{p-1}\left(b_{N,K}+c_{N+1,pK+j}\frac{\ell_{N+1,pK+j}}{1-r}\right){\one}_{\Gamma_{N+1,pK+j}}\\
	&=\sum\limits_{K=0}^{p^N-1}\sum\limits_{j=0}^{p-1}\left(b_{N,K}+c_{N+1,pK+j}\Tilde\ell_{N+1,pK+j}\right){\one}_{\Gamma_{N+1,pK+j}}=\gamma_0^{\Tilde{\cT}_{N+1}}u=f. 
\end{align*}
as per \eqref{eq:trace_u} and \eqref{bcf}.

Step 2.4: Relating $\gamma_1^{\T}v$ to $\gamma_1^{\widetilde{\T}_{N+1}}v$. 
We will compute $\gamma^\cT_1 v$ using Lemma \ref{lem-weak} and an explicit expression of the co-normal derivative \eqref{fnderiv}. Let $g\in H^{\sigma d}(\Gamma)$, then
\begin{align}
	(\gamma^{\cT}_1 v,g)_{H^{-\sigma d}(\Gamma),H^{\sigma d}(\Gamma)}&=\lim_{n\to\infty}
	\int_{\Gamma} \sum_{k=0}^{p^n-1} \dfrac{\omega_{n,k}}{|\Gamma_{n,k}|}v'_{n,k}(L_{n,k}^-)\one_{\Gamma_{n,k}} g\,\dd s\nonumber\\
	&=\lim_{n\to\infty}\sum_{K=0}^{p^N-1}\sum_{j=0}^{p-1} \int_\Gamma F_n(N,K,j) g\,\dd s,\label{sumfn}\\
	F_n(K,j)&:=\sum_{k:\, \cT_{n,k}\subset\cT^j_{N,K}} \dfrac{\omega_{n,k}}{|\Gamma_{n,k}|}v'_{n,k}(L_{n,k}^-)\one_{\Gamma_{n,k}} \text{ for }n\ge N+1.\nonumber
\end{align}
By \eqref{eq:vprime_nk} and the assumption on $\cT^j_{N,K}$ we have
\begin{align*}
	F_n(K,j)&=\sum_{k:\, \cT_{n,k}\subset\cT^j_{N,K}} \dfrac{\omega_{N+1,pK+j}\omega^{n-N-1}}{|\Gamma_{n,k}|}\dfrac{c_{N+1,pK+j}}{(p\omega)^{n-N-1}}\one_{\Gamma_{n,k}}\\
	&=\omega_{N+1,pK+j}c_{N+1,pK+j}\sum_{k:\, \cT_{n,k}\subset\cT^j_{N,K}} \dfrac{1}{p^{n-N-1}|\Gamma_{n,k}|}\one_{\Gamma_{n,k}},
\end{align*}
and the substitution into \eqref{sumfn} gives
\begin{equation}
	\label{ggtt11}
	\begin{aligned}
		(\gamma^{\cT}_1 v,g)_{H^{-\sigma d}(\Gamma),H^{\sigma d}(\Gamma)}&=
		\sum_{K=0}^{p^N-1}\sum_{j=0}^{p-1} \omega_{N+1,pK+j}c_{N+1,pK+j}\lim_{n\to\infty}G_n(K,j),\\
		G_n(K,j):&=\sum_{k:\, \cT_{n,k}\subset\cT^j_{N,K}} \int_\Gamma\dfrac{1}{p^{n-N-1}|\Gamma_{n,k}|}\one_{\Gamma_{n,k}} g\,\dd s.
	\end{aligned}
\end{equation}

For the computation of the limits we are going to use the local charts $\Phi_{N_0,K_0}$.
Let $K_0\in\{0,\dots,p^{N_0}-1\}$ be such that $\cT^j_{N,K}\subset \cT_{N_0,K_0}$, then for 
$\cT_{n,k}\subset\cT^j_{N,K}$ one has
\begin{equation}
	\begin{aligned}
		\int_\Gamma\dfrac{1}{p^{n-N-1}|\Gamma_{n,k}|}\one_{\Gamma_{n,k}} g\,\dd s&=\int_{\Gamma_{n,k}} \dfrac{1}{p^{n-N-1}|\Gamma_{n,k}|} g\,\dd s\\
		&=\int_{\Tilde\Gamma_{n,k}}\dfrac{1}{p^{n-N-1}|\Gamma_{n,k}|}g\big(\Phi_{N_0,K_0}(y)\big) J_{N_0,K_0}(y)\,\dd y.
	\end{aligned}
	\label{jnk1}
\end{equation}
Choose arbitrary points $a_{n,k}\in \Tilde \Gamma_{n,k}$ and denote $j_{n,k}:=J_{N_0,K_0}(a_{n,k})$. As $J_{N_0,K_0}$ is smooth, one has $|J_{N_0,K_0}(y)-j_{n,k}|\le b|y-a_{n,k}|$ with some fixed $b>0$ uniformly in $(y,n,k)$. Remark that by assumption the diameter of $\Tilde \Gamma_{n,k}$ is $O(p^{-\frac{n}{d}})$ for large $n$, which shows that $J_{N_0,K_0}(y)=j_{n,k}+O(p^{-\frac{n}{d}})$ uniformly for $y\in \Tilde \Gamma_{n,k}$ as $n$ becomes large. Similarly, 
\begin{align*}
	|\Gamma_{n,k}|&=\int_{\Tilde \Gamma_{n,k}} J_{N_0,K_0}(y)\,\dd y
	=\int_{\Tilde \Gamma_{n,k}} j_{n,k}\,\dd y+\int_{\Tilde \Gamma_{n,k}} \big(J_{N_0,K_0}(y)-j_{n,k}\big)\,\dd y\\
	&=\Big(j_{n,k}+O(p^{-\frac{n}{d}})\Big)|\Tilde \Gamma_{n,k}|
	=\dfrac{j_{n,k}|\Tilde \Gamma_{n,k}|}{1+O(p^{-\frac{n}{d}})}.
\end{align*}
The substitution of these asymptotic estimates into \eqref{jnk1} yields
\begin{align*}
	\int_\Gamma\dfrac{1}{p^{n-N-1}|\Gamma_{n,k}|}\one_{\Gamma_{n,k}} g\,\dd s&=
	\int_{\Tilde\Gamma_{n,k}}\dfrac{1}{p^{n-N-1}|\Tilde\Gamma_{n,k}|}g\big(\Phi_{N_0,K_0}(y)\big) \,
	\big(1+O(p^{-\frac{n}{d}}\big)\big)
	\dd y\\
	&=\dfrac{1}{p^{N_0-N-1}|\Tilde \Gamma_{N_0,K_0}|}\int_{\Tilde\Gamma_{n,k}}g\big(\Phi_{N_0,K_0}(y)\big) \,
	\big(1+O(p^{-\frac{n}{d}}\big)\big)
	\dd y
\end{align*}
where the $O$-term is uniform in $y\in \Gamma_{N_0,K_0}$ and we have used  $|\Tilde\Gamma_{n,k}|=p^{N_0-n}|\Tilde \Gamma_{N_0,K_0}|$. One has then
\begin{align*}
	\lim_{n\to\infty} G_n(K,j)&=\lim_{n\to\infty}\sum_{k:\, \cT_{n,k}\subset\cT^j_{N,K}}
	\dfrac{1}{p^{N_0-N-1}|\Tilde \Gamma_{N_0,K_0}|}\int_{\Tilde\Gamma_{n,k}}g\big(\Phi_{N_0,K_0}(y)\big) \,
	\big(1+O(p^{-\frac{n}{d}}\big)\big)
	\dd y\\
	&=\lim_{n\to\infty}\dfrac{1}{p^{N_0-N-1}|\Tilde \Gamma_{N_0,K_0}|}\int_{\Tilde \Gamma_{N+1,pK+j}}g\big(\Phi_{N_0,K_0}(y)\big) \,
	\big(1+O(p^{-\frac{n}{d}}\big)\big)
	\dd y\\
	&=\dfrac{1}{p^{N_0-N-1}|\Tilde \Gamma_{N_0,K_0}|}\int_{\Tilde \Gamma_{N+1,pK+j}}g\big(\Phi_{N_0,K_0}(y)\big) \, \dd y\\
	&=\dfrac{1}{p^{N_0-N-1}|\Tilde \Gamma_{N_0,K_0}|} \int_{\Gamma_{N+1,pK+j}} \dfrac{g}{H_{N_0,K_0}}\,\dd s.
\end{align*}
Recall that this result was obtained under the assumption $\cT_{N,K}\subset\cT_{N_0,K_0}$.

We now regroup the summands in \eqref{ggtt11}:
\begin{align*}
	(\gamma^{\cT}_1 v,&g)_{H^{-\sigma d}(\Gamma),H^{\sigma d}(\Gamma)}=
	\sum_{K_0=0}^{p^{N_0}-1} \sum_{(K,j):\,\cT^j_{N,K}\subset\cT_{N_0,K_0}}
	\omega_{N+1,pk+j}c_{N+1,pK+j}\lim_{n\to\infty}G_n(K,j)\\
	&=\sum_{K_0=0}^{p^{N_0}-1} \sum_{(K,j):\,\cT^j_{N,K}\subset\cT_{N_0,K_0}}
	\dfrac{\omega_{N+1,pk+j}c_{N+1,pK+j}}{p^{N_0-N-1}|\Tilde \Gamma_{N_0,K_0}|} \int_{\Gamma_{N+1,pK+j}} \dfrac{g}{H_{N_0,K_0}}\,\dd s\\
	&=\int_\Gamma \bigg(\dfrac{1}{p^{N_0-N-1}}\sum_{K_0=0}^{p^{N_0}-1} \dfrac{\one_{\Gamma_{N_0,K_0}}}{|\Tilde\Gamma_{N_0,K_0}|H_{N_0,K_0}}\bigg)	\\
	&\qquad\times
	\bigg(\sum_{K=0}^{p^N-1}\sum_{j=0}^{p-1} \dfrac{\omega_{N+1,pK+j} c_{N+1,pK+j}}{|\Gamma_{N+1,pK+j}|}\one_{\Gamma_{N+1,pK+j}}\bigg)\, g\,\dd s\\
	&=\int_{\Gamma} H_{N_0} \gamma^{\Tilde \cT_{N+1}}_1 u\, g\,\dd s,
\end{align*}
where \eqref{hn0} and \eqref{tildedn} were used in the last step,
and we arrive at
\[
\D f=\gamma^\cT_1 v=H_{N_0} \gamma^{\Tilde \cT_{N+1}}_1 u=H_{N_0}\Tilde \D_{N+1} f. \qedhere
\]
\end{proof}

\section{Boundary value problems on the exterior domain}\label{sec3-bvp}
Recall that in the context of boundary value problems for an open set $U\subset\R^m$ one usually denotes by $H^1_\loc(U)$ the set of the functions $f$ on $U$ such that $\varphi f\in H^1(\Omega)$ for any cut-off function $\varphi\in C^\infty_c(\R^m)$. In particular, $H^1_\loc(U)=H^1(U)$ for all bounded $U$.

For the rest of the section it will be convenient to denote
\[
\Omega_+:=\Omega,\qquad \Omega_-:=\R^m\setminus\overline{\Omega},
\]
and let $\nu$ be the unit normal on $\Gamma$ pointing to $\Omega_+$.
We then have the respective Dirichlet trace maps
\[
\gamma^+_0\equiv \gamma^\Omega_0: \ H^1_\loc(\Omega_+)\to H^\half(\Gamma),
\qquad
\gamma^-_0: \ H^1(\Omega_-)\to H^\half(\Gamma),
\]
so that for the functions $u$ smooth up to a boundary one has $\gamma^\pm_0 u:=u|_{\Gamma}$,
and the Neumann trace maps
\begin{align*}
	\gamma^+_1\equiv\gamma^\Omega_1:& \ \big\{u\in H^1_{\loc}(\Omega_+):\ \Delta u \in L^2_\loc(\Omega_+)\big\}\to H^{-\half}(\Gamma),\\
	\qquad
	\gamma^-_1:& \ \big\{u\in H^1(\Omega_-):\ \Delta u\in L^2(\Omega_-)\big\} \to H^{-\half}(\Gamma),
\end{align*}
such that for the functions $u$ smooth up to the boundary one has
\[
\gamma^\pm_1 u:=\partial_\nu u|_\Gamma,
\]
where $\nu$ is the unit normal on $\Gamma$ pointing to $\Omega\equiv \Omega_+$. Remark that with this sign convention, one has $u\in H^1(\R^m)$ if and only if $u\in H^1(\R^m\setminus\Gamma)$ with $(\gamma_0^+-\gamma_0^-)u=0$. In addition, $u\in H^1(\R^m)$ with $\Delta u\in L^2(\R^m)$
if and only if $u\in H^1(\R^m\setminus\Gamma)$ with $\Delta u\in L^2(\R^m\setminus\Gamma)$ (in the sense of distributions on $\R^m\setminus\Gamma$) with $(\gamma_0^+-\gamma_0^-)u=0$ and $(\gamma_1^+ -\gamma_1^-)u=0$.

The aim of the present section is to give detailed proofs of the following results:
\begin{thm}[Exterior Dirichlet problem]\label{thm21}
	For any $f\in L^2_\comp(\Omega)$ and $g\in H^\half(\Gamma)$ there is a unique solution $u\in H^1_\loc(\Omega)$
	of
	\begin{equation}
		\label{dir-ext}
		\left\{\begin{aligned}
			- \Delta u &=f  \text{ in }\Omega,\\
			\gamma^\Omega_0 u &=g  \text{ on }\Gamma,\\
			u(x)&= 	O(|x|^{2-m}) \text{ for } |x|\to \infty,
		\end{aligned}\right.
	\end{equation}
and the solution $u$ depends continuously on the right-hand sides $f$ and $g$ in the following sense: for any $\varphi\in C^\infty_c(\R^m)$ there is a constant $c>0$ such that for any $(f,g)$ one has
\[
\|\varphi u\|_{H^1(\Omega)}\le c\big(\|f\|_{L^2(\Omega)}+\|g\|_{H^\half(\Gamma)}\big).
\]

\end{thm}

\begin{thm}[Dirichlet-to-Neumann map for the exterior domain]\label{thm22}
	For $g\in H^\half(\Gamma)$ let $u_g$ denote the solution of \eqref{dir-ext} with $f=0$.	 Define the Dirichlet-to-Neumann operator 
	\[
	\dtn: \ H^{\half}(\Gamma)\ni g\mapsto \gamma^\Omega_1 u_g\in H^{-\half}(\Gamma),
	\]
	then $\dtn$ is bounded, Fredholm of index zero, and non-positive, i.e. 
	\[
	-\big(\dtn g, \overline{g}\big)_{H^{-\half}(\Gamma), H^\half(\Gamma)} \ge 0 \text{ for any } g\in H^{\half}(\Gamma).
	\]  	
	Moreover,	
	\begin{itemize}
		\item[(i)] if $m\ge 3$, then $\dtn$ is coercive, i.e. one can find a constant $c>0$ such that
		\begin{equation}
			\label{dtngg}
			-\big(\dtn g, \overline{g}\big)_{H^{-\half}(\Gamma), H^\half(\Gamma)} \ge c\|g\|_{H^\half(\Gamma)}^2
		\end{equation}
		holds for all $g\in H^\half(\Gamma)$, in particular, $\dtn$ is bijective,
		\item[(iii)] if $m=2$, then $\ker \dtn=\C \one_\Gamma$, and $\dtn$ is coercive on
		\[
		H^\half_0(\Gamma):=\big\{ g\in H^\half(\Gamma):\, (\one_\Gamma,g)_{H^{-\half}(\Gamma),H^\half(\Gamma)}=0\big\},
		\]
		i.e. there is a constant $c>0$ such that \eqref{dtngg} is fulfilled for all $ g\in H^{\half}_0(\Gamma)$.
	\end{itemize}
\end{thm}	

While both results are indeed folkloric, we did not manage to find a suitable reference containing all necessary details for the required $H^1$-setting. McLean's book \cite{McLean} provides all important tools but lacks a precise formulation of final results for the problems in exterior domains, so we decided to complete the respective part of the argument.

In the next subsection we recall some constructions related to boundary integral operators.
These machineries are then used to prove Theorems \ref{thm21} and \ref{thm22}, first for $m\ge 3$ in Subsection \ref{sec-m3} and then for $m=2$ in Subsection~\ref{sec-m2}.

\subsection{Surface potentials}

We begin by citing some of the important statements from~\cite{McLean}. Let $G$ be the standard fundamental solution for the Laplace equation in $\R^m$,

\begin{equation}\label{fundameltalsolution}
	G(x) =\begin{cases} \dfrac{1}{2 \pi} \log \dfrac{r}{|x|},& m=2,\\[\bigskipamount]
		\dfrac{1}{(m-2)\Upsilon_m} \dfrac{1}{|x|^{m-2}}, & m\ge 3,
		\end{cases}
\end{equation}
where $\Upsilon_m$ is the hypersurface area of the unit sphere in $\R^m$ and $r>0$ is arbitrary (to be chosen later), and consider the convolution operator
\[
\cG: \ \cE'(\R^m)\ni f\mapsto G\star f \in \cS'(\R^m),
\]
which satisfies $-\Delta \circ \cG = -\cG \circ \Delta=\Id$. For $g\in H^{-\half}(\Gamma)$ define  $\gamma^*_0 g\in \cS'(\R^m)$ by
\[
(\gamma^*_0 g, \phi) = (g, \gamma^+_0 \phi)_{H^{-\half}(\Gamma), H^{\half}(\Gamma)}  \text{ for all } \phi\in\cS(\R^m).
\]
Analogously, for $g\in H^\half(\Gamma)$ we define the distribution $\gamma^*_1 g\in \cS'(\R^m)$ by
\[
(\gamma^*_1 g, \phi) = (g, \gamma^+_1 \phi)_{H^{\half}(\Gamma), H^{-\half}(\Gamma)} \text{ for all } \phi\in\cS(\R^m).
\]
By construction both $\gamma^*_0 g$ and $\gamma^*_1 g$ are supported by $\Gamma$ for all admissible $g$, and in fact,
\[
\gamma_0^*:\ H^{-\half}(\Gamma)\to H^{-1}_\comp(\R^m),
\qquad
\gamma_1^*:\ H^{\half}(\Gamma)\to H^{-2}_\comp(\R^m).
\]
The single-layer potential $\SL$ and the double-layer potential $\DL$ associated with $\Gamma$ are defined by
\begin{align*}
\SL:= \cG\circ \gamma^*_0,
\qquad
\DL:= \cG \circ \gamma^*_1.
\end{align*} 
By construction, the functions $\SL g$ and $\DL g$ are harmonic in $\Omega_\pm$ for all admissible $g$.
For every cutoff function $\chi \in C^\infty_c(\R^m)$, the operators
\[
H^{-\half}(\Gamma)\ni g\mapsto \chi \SL g\in H^1(\Omega_\pm),
\qquad
H^{\half}(\Gamma)\ni g\mapsto \chi \DL g\in H^1(\Omega_\pm)
\]
are bounded, and the jump relations
\begin{equation}
	\label{jump}
\begin{aligned}
	(\gamma^+_0-\gamma^-_0) \SL g&=0,& (\gamma^+_1-\gamma^-_1)\SL g &= -g & \text{for all }& g\in H^{-\half}(\Gamma), \\
	(\gamma^+_0-\gamma^-_0)\DL g&=g,& (\gamma^+_1-\gamma^-_1) \DL g &= 0 &\text{ for all } &g\in H^{\half}(\Gamma),
\end{aligned}
\end{equation}
are fulfilled. We will frequently use the identity 
\begin{equation}
	 \label{sl-dl}
\begin{aligned}(\gamma^\pm_1\SL \phi,\psi)_{H^{-\half}(\Gamma),H^\half(\Gamma)}&=(\phi,\gamma^\mp_0\DL\psi)_{H^{-\half}(\Gamma),H^\half(\Gamma)}\\
\text{ for any }& \phi\in H^{-\half}(\Gamma),\ \psi \in H^\half(\Gamma),
\end{aligned}
\end{equation}
see \cite[Thm. 6.17 and Sec.~7]{McLean}.

The above considerations give rise to four important boundary operators:
\begin{align*}
	S:=&&\gamma^+_0\SL: \ & H^{-\half}(\Gamma) \to H^{\half}(\Gamma),\\
	R:=&&-\gamma^+_1 \DL:\  &H^{\half}(\Gamma) \to H^{-\half}(\Gamma),\\
	T:=&&(\gamma^+_0 + \gamma^-_0)\DL:\ & H^{\half}(\Gamma) \to H^{\half}(\Gamma),
\intertext{and remark that the adjoint of $T$ is given by}	
	T^*:=&&(\gamma^+_1+\gamma^-_1)\SL:\ &  H^{-\half}(\Gamma) \to H^{-\half}(\Gamma).
\end{align*}
By \cite[Theorem 6.11]{McLean} these operators are bounded in respective spaces, and the jump relation \eqref{jump} yield the identities
\begin{equation}
	\label{tracesl}
\gamma^\pm_1 \SL g = \frac{1}{2} (\mp g + T^* g),
\qquad
\gamma^{\pm}_0 \DL g = \frac{1}{2}(\pm g + Tg)
\end{equation}
valid for all $g$ in respective domains.  By general results for elliptic problems of \cite[Thm. 7.6, 7.8]{McLean} and their refinement for the Laplacian (namely, \cite[Theorem 8.16]{McLean} for the single-layer boundary integral operator in $m=2$, \cite[Corollary 8.13]{McLean} for the single-layer boundary integral operator in $m=3$ and \cite[Theorem 8.21]{McLean} for the hypersingular integral operator), we have: 
\begin{lem}\label{lemma-RS}
The operator $S$ is a self-adjoint Fredholm operator of index $0$, and for a suitable choice of $r>0$ in \eqref{fundameltalsolution} one has
 $\ker S = \{0\}$, in particular, $S$ has a bounded inverse.
Furthermore, the operator $R$ is self-adjoint Fredholm of index $0$ with
$\ker R = \C \one_\Gamma$.
\end{lem}
From now we assume that $r$ in \eqref{fundameltalsolution} is chosen in such a way that the assertions of Lemma~\ref{lemma-RS} hold.

Remark that for any $g\in H^\half(\Gamma)$ and $x\in\R^m\setminus\Gamma$ one has
	\begin{align*}
	\DL g(x)&=\int_{\Gamma} \partial_{\nu_y} G(x-y) g(y)\,\dd s(y)=\dfrac{1}{\Upsilon_m}\int_\Gamma \dfrac{\langle \nu_y,x-y\rangle_{\R^{m}}}{|x-y|^m}g(y)\dd s(y).
\end{align*}		
Using a standard computation with an integration by parts 
we arrive at the following identities:
\begin{equation}
	\DL \one_\Gamma=0 \text{ in } \Omega_+, \quad
	\DL \one_\Gamma=-1 \text{ in } \Omega_-,
	\quad
	\gamma^+_0 \DL \one_\Gamma=0,
	\quad
	\gamma^-_0 \DL \one_\Gamma= -\one_\Gamma.
	\label{DLone}
\end{equation}

%

 The following two theorems are crucial for the subsequent considerations. First, \cite[Thm 7.15 + 8.9]{McLean} yield:
\begin{thm} \label{boundequ}
Let $f\in L^2_\comp(\Omega_+)$. 

(A) Let $g\in H^\half(\Gamma)$. If $u \in H^1_\loc(\Omega_+)$ solves the exterior Dirichlet problem
		\begin{equation}\label{dirpro}
			\left\{\begin{aligned}
				- \Delta u &= f  \text{ in }\Omega_+,\\
				\gamma^+_0 u &= g  \text{ on }\Gamma,\\
				u(x)&= \begin{cases}
					O(|x|^{2-m})  \text{ for } m\ge 3,\\
					b \log |x| + O(|x|^{-1}) \text{ with some $b\in\C$ for  } m=2,
				\end{cases} \text{ as } |x|\to \infty,
			\end{aligned}\right.
		\end{equation}
		then the function
		\[
		h:=\gamma^+_1 u \in H^{-\half}(\Gamma)
		\]
		satisfies the equation
		\begin{align}\label{shboundaryequ}
			Sh = \gamma^+_0 \cG f - \frac{1}{2}(g-Tg),
		\end{align}
		and $u$ can be represented as
		\begin{align}\label{solutionrepr}
			u= \cG f + \DL g - \SL h.
		\end{align}
		Conversely, if $h\in H^{-\frac{1}{2}}(\Gamma)$ is a solution of \eqref{shboundaryequ}, then the equation \eqref{solutionrepr} defines a solution of \eqref{dirpro}. 

(B) Let $h\in H^{-\frac{1}{2}}(\Gamma)$. If $u \in H^1_\loc(\Omega_+)$ is a solution of the exterior Neumann problem
	\begin{equation}\label{neupro}
	\left\{\begin{aligned}
		- \Delta u &= f  \text{ in }\Omega_+,\\
		\gamma^+_1 u &= h  \text{ on }\Gamma,\\
		u(x)&= \begin{cases}
			O(|x|^{2-m})  \text{ for } m\ge 3,\\
			b \log |x| + O(|x|^{-1}) \text{ with some $b\in\C$ for  } m=2,
		\end{cases} \text{ as } |x|\to \infty,
	\end{aligned}\right.
\end{equation}
then the function $g:= \gamma^+_0 u \in H^{\half}(\Gamma)$ is a solution of the equation 
		\begin{align}\label{rgboundaryequ}
			Rg = \gamma^+_1\cG f - \frac{1}{2}(h+T^*h)
		\end{align}
		and $u$ is of the form \eqref{solutionrepr}. 		Conversely, if $g\in H^{\half}(\Gamma)$ is a solution of \eqref{rgboundaryequ}, then the function $u$ defined by~\eqref{solutionrepr} defines a solution of \eqref{neupro}. 
\end{thm}

\begin{rmk}
	The required behavior of $u$ at infinity in Theorem \ref{boundequ} is termed as $\mathcal{M}u=0$
	in most assertions in the book \cite{McLean}, see \cite[Thm. 8.9]{McLean} for more detail.
\end{rmk}

	The above result states an equivalence between the solutions to the exterior boundary-value problems and the representation (Kirchhoff) formula \eqref{solutionrepr}, coupled with associated boundary integral equations for unknown traces. However, at this point it is unclear whether the exterior boundary-value problems are well-posed. This will be clarified in Theorem \ref{DTNdge2}.

The following theorem combines the claims of \cite[Thm. 8.10 and 8.18]{McLean}:
\begin{lem} \label{uniq} 
	Let $u\in H^1_\loc(\Omega_+)$ with $\Delta u = 0$ on $\Omega_+$ satisfy the radiation condition
	\[
	u(x)=O\big(|x|^{2-m}\big)  \text{ for }|x| \to \infty,
	\]
	then:
	\begin{enumerate}
		
		\item[(a)] $\gamma^+_0u = 0$ if and only if $u=0$ on $\Omega_+$,
		
		\item[(b)] $\gamma^+_1 u = 0$ if and only if
		\begin{itemize}
			\item[(i)] $u=0$ in $\Omega_+$ for $m\ge 3$,
			\item[(ii)]  $u$ is constant on $\Omega_+$ for $m=2$.
		\end{itemize}
		
	\end{enumerate}
\end{lem}

The following computation will also be used at several places:

\begin{lem} \label{uvskal} 
	Let $u,v\in H^1_\loc(\Omega_+)$ with $\Delta u = \Delta v=0$ on $\Omega_+$ satisfy the radiation condition
\[
u(x)=O\big(|x|^{2-m}\big), \quad v(x)=O\big(|x|^{2-m}\big)\quad \text{for }|x| \to \infty,
\]
then
\[
\big(\gamma^+_1 u,\overline{\gamma^+_0v}\,\big)_{H^{-\half}(\Gamma),H^\half(\Gamma)}=-\int_{\Omega_+} \langle \nabla u,\nabla v\rangle_{\C^m}\dd x.
\]
\end{lem}

\begin{proof}
Remark first that by \cite[Prop.~2.75]{Folland} the radial derivative $\partial_ r u$ of $u$ satisfies
	\[
	\partial_r u(x)=\begin{cases} O\big(|x|^{1-m}\big), & m\ge 3,\\
		O\big(|x|^{-2}\big), & m=2,
		\end{cases} \text{ for } |x|\to\infty
	\]
Let $R>0$ such that $\overline{\Omega_-}\subset B_R(0)$ and denote $\Omega^+_R = B_R(0) \cap \Omega_+$.
We have then
\begin{align*}	0&= \int_{\Omega^+_R} (\Delta u) \overline{v} \,\dd x=(\Delta u,v)_{H^{-1}_\comp(\Omega_+),H^1_\loc(\Omega_+)}\\
&= -\big(\gamma^+_1  u, \overline{\gamma^+_0 v} \big)_{H^{-\half}(\Gamma), H^{\half}(\Gamma)} +  \int_{|x|=R} (\partial_r u) \overline{v}\,\dd s-\int_{\Omega^+_R}\langle \nabla u,\nabla v\rangle_{\C^m}\dd x,
\end{align*}
with $\dd s$ being the hypersurface measure, which yields
\begin{equation}
	\label{lok22}
	\big(\gamma^+_1  u, \overline{\gamma^+_0 v} \big)_{H^{-\half}(\Gamma), H^{\half}(\Gamma)}=-\int_{\Omega^+_R}\langle \nabla u,\nabla v\rangle_{\C^m}\dd x+\int_{|x|=R} (\partial_r u) \overline{v}\,\dd s.	
\end{equation}
Using the above bounds for $\partial_r u$ and $v$, for large $R$ we have (with some fixed $C>0$)
\[
\Big|
\int_{|x|=R} (\partial_r u) \overline{v}\,\dd s\Big|
\le
\begin{cases}
 C \cdot R^{m-1} \cdot R^{1-m}\cdot R^{2-m}=O(R^{2-m})=o(1), & m\ge 3,\\
 C \cdot R \cdot R^{-2}\cdot  1=O(R^{-1})=o(1), & m=2,
\end{cases} 
\]
and sending $R$ to $\infty$ in \eqref{lok22} gives
\begin{equation}
	\label{lok23}
	\big(\gamma^+_1  u, \overline{\gamma^+_0 v} \big)_{H^{-\half}(\Gamma), H^{\half}(\Gamma)}=-\lim_{R\to\infty}\int_{\Omega^+_R}\langle \nabla u,\nabla v\rangle_{\C^m}\dd x.	
\end{equation}
Using the last identity for $u=v$ we conclude that $|\nabla u|^2$ and $|\nabla v|^2$ are integrable on $\Omega_+$,
so the limit on the right-hand side of \eqref{lok23} is exactly the Lebesgue integral over $\Omega_+$.
\end{proof}

\subsection{Dirichlet-to-Neumann map for $m\ge 3$}\label{sec-m3}

\begin{thm} \label{DTNdge2}
	Let $m\geq 3$ and $f\in L^2_\comp(\Omega_+)$. Then 
	\begin{enumerate}
		\item[(a)] For every $g\in H^\half (\Gamma)$ there is a unique solution $u\in H^1_\loc(\Omega_+)$
		to the Dirichlet problem with radiation condition
		\[
			\left\{\begin{aligned}
			- \Delta u &= f  \text{ in }\Omega_+,\\
			\gamma^+_0 u &= g  \text{ on }\Gamma,\\
			u(x)&= 	O(|x|^{2-m}) \text{ for } |x|\to \infty,
		\end{aligned}\right.
		\]
	and for any cut-off function $\varphi\in C^\infty_c(\R^m)$ one can find a constant $c>0$ such that for any $(f,g)$ as above and the respective solution $u$ one has
	\begin{equation}
		 \label{uc1}
	\|\varphi u\|_{H^1(\Omega_+)}\le c\big( \|f\|_{L^2(\Omega)}+\|g\|_{H^\half(\Gamma)}\big).
	\end{equation}
	\item[(b)] For every $h\in H^{-\half}(\Gamma)$ there is a unique $u\in H^1_\loc(\Omega_+)$ fulfilling the Neumann problem with radiation condition
	\[
				\left\{\begin{aligned}
		- \Delta u &= f  \text{ in }\Omega_+,\\
		\gamma^+_1 u &= h  \text{ on }\Gamma,\\
		u(x)&= 	O(|x|^{2-m}) \text{ for } |x|\to \infty.
	\end{aligned}\right.
	\]
	\item[(c)]
	For $g\in H^\half(\Gamma)$ let $u_g$ be the solution of (a) for $f=0$.	 Define the Dirichlet-to-Neumann operator 
	\[
		\dtn: \ H^{\half}(\Gamma)\ni g\mapsto \gamma^+_1 u_g\in H^{-\half}(\Gamma),
	\]
	then $\dtn$ is bounded, invertible and coercive, i.e. there is a constant $c>0$ such that for every $g\in H^{\half}(\Gamma)$ one has
	\[
	-\big(\dtn g, \overline{g}\big)_{H^{-\half}(\Gamma), H^\half(\Gamma)} \geq c \|g\|^2_{H^\half(\Gamma)}.
	\]  
	\end{enumerate}
\end{thm}
\begin{proof}
The uniqueness in (a) and (b) follows directly by Lemma~\eqref{uniq}(a,b-i). For the existence
	we note first that by Theorem \ref{boundequ} the solvability in (a) resp. (b)  is equivalent to the  solvability of the boundary equations \eqref{shboundaryequ} resp. \eqref{rgboundaryequ}. Since $S$ is a bijective isomorphism, the problem \ref{shboundaryequ} has a unique solution $h\in H^{-\half}(\Gamma)$ for any choice of $g$, and this solution continuously depends on $g$, which proves the existence in (a). The stability estimate \eqref{uc1} follows from the representation \eqref{solutionrepr} and the continuity properties of $\cG$, $\SL$ and $\DL$.
	
To complete the argument for (b) slightly more work is needed. 
First, by using \eqref{tracesl} we rewrite the condition \eqref{rgboundaryequ} as $Rg=\gamma_1^+\cG f - \gamma^-_1 \SL h$.
As $R$ is Fredholm and self-adjoint with $\ker R=\C \one_\Gamma$, this last equation is solvable with respect to $g$
if and only if the right-hand side satisfies
\begin{align} \label{solutioncondition}
	I:=( \one_{\Gamma}, \gamma^+_1 \cG f - \gamma^-_1 \SL h )_{H^\half (\Gamma), H^{-\half}(\Gamma)} = 0.
\end{align}
Using the definition of $\DL$ we obtain
\begin{align*}
( \one_{\Gamma}, \gamma^+_1 \cG f)_{H^\half (\Gamma), H^{-\half}(\Gamma)}&=\big(\DL \one_\Gamma, f\big)_{H^1_{\loc}(\Omega_+),H^{-1}_{\comp}(\Omega_+)}\stackrel{\eqref{DLone}}{=}
\big(0, f\big)_{H^1_{\loc}(\Omega_+),H^{-1}_{\comp}(\Omega_+)}=0,\\
( \one_{\Gamma}, \gamma^-_1 \SL h )_{H^\half (\Gamma), H^{-\half}(\Gamma)}&=
( \gamma^+_0 \DL \one_{\Gamma}, h )_{H^\half (\Gamma), H^{-\half}(\Gamma)}
\stackrel{\eqref{sl-dl}}{=}( 0, h )_{H^\half (\Gamma), H^{-\half}(\Gamma)}=0,
\end{align*}
therefore, $I=0$ and the solvability condition \eqref{solutioncondition} is fulfilled, which proves (b).

It remains to prove (c). Let $g\in H^\half(\Gamma)$ with  $\dtn g=0$, then $u_g$ is a solution of (b) for $f=0$ and $h=0$, i.e. $u_g=0$ in $\Omega_+$, and $g=\gamma^+_0 u_g=0$. This proves the injectivity of $\dtn$. Let $h\in H^{\half}(\Gamma)$ and let $u$ be a solution of (b) for $f=0$ and the chosen $h$, then $h=\dtn g$ for $g:=\gamma^+_0 u$, which shows the surjectivity of $\dtn$.

We further remark that for any $g\in H^\half(\Gamma)$, as shown by \eqref{shboundaryequ} with $f=0$, 
\begin{align*}
\dtn g= -\frac{1}{2}S^{-1} (g-Tg),
\end{align*}
which yields the boundedness of $\dtn$, and the open mapping theorem implies that $\dtn^{-1}$ is bounded as well.

It remains to show the coercivity. For $h\in H^{-\half}(\Gamma)$ let $v_h$ be the solution of (b) with $f=0$.
For any $g\in H^\half(\Gamma)$ and any $h\in H^{-\half}(\Gamma)$ we have then, due to Lemma \ref{uvskal}:
	\[
	-(h,\overline{g})_{H^{-\half}(\Gamma), H^{\half}(\Gamma)}=\int_{\Omega_+}\langle \nabla v_h,\nabla u_g\rangle_{\C^N}\dd x.
\]
In particular, noting that $v_h=u_g$ for $h=\dtn g$ one arrives at
\begin{equation}
	\label{dtn1}
	-(\dtn g,\overline{g})_{H^{-\half}(\Gamma), H^{\half}(\Gamma)}=\int_{\Omega_+} |\nabla u_g|^2\dd x
\end{equation}
and, similarly,
\[
\int_{\Omega_+}	|\nabla v_h|^2\dd x =-\big(h, \overline{\dtn^{-1} h} \big)_{H^{-\half}(\Gamma), H^{\half}(\Gamma)}\leq  \|\dtn^{-1}\|_{H^{-\half}(\Gamma) \to H^{\half}(\Gamma)} \|h\|^2_{H^{-\half}(\Gamma)}. 
\]
Then for any $g \in H^{\frac{1}{2}}(\Gamma)$ we have
	\begin{align*}
		\|g\|_{H^{\frac{1}{2}}(\Gamma)}^2 &= \sup\limits_{h\in H^{-\frac{1}{2}}(\Gamma),\, h\ne 0} \dfrac{\big|(h,\Bar g )_{H^{-\half }(\Gamma), H^{\half }(\Gamma)}\big|^2}{\|h\|_{H^{-\half}(\Gamma)}^2}\\
		& = \sup\limits_{h\in H^{-\half}(\Gamma),\, h\ne 0} \dfrac{\bigg|\displaystyle \int_{\Omega_+}\langle \nabla v_h,\nabla u_g\rangle_{\C^m}\dd x	\bigg|^2	
		}{\|h\|_{H^{-\half}(\Gamma)}^2}\\
		& \leq \sup\limits_{h\in H^{-\half}(\Gamma),\,h\ne 0} \dfrac{\displaystyle \int_{\Omega_+}|\nabla v_h|^2\dd x \int_{\Omega_+} |\nabla u_g|^2\dd x}{\|h\|_{H^{-\half }(\Gamma)}^{2}}\\
		&\leq \sup\limits_{h\in H^{-\half}(\Gamma),\,h\ne 0} \dfrac{\displaystyle \|\dtn^{-1}\|_{H^{-\half}(\Gamma) \to H^{\half}(\Gamma)} \|h\|^2_{H^{-\half}(\Gamma)} \int_{\Omega_+} |\nabla u_g|^2\dd x}{\|h\|_{H^{-\half }(\Gamma)}^{2}}\\
		&\le  \|\dtn^{-1}\|_{H^{-\half}(\Gamma) \to H^{\half}(\Gamma)}  \int_{\Omega_+} |\nabla u_g|^2\dd x,
	\end{align*}
	i.e.
	\[
	\int_{\Omega_+} |\nabla u_g|^2\dd x\ge \|\dtn^{-1}\|_{H^{-\half}(\Gamma) \to H^{\half}(\Gamma)}^{-1} \|g\|_{H^{\frac{1}{2}}(\Gamma)}^2,
	\]
	and the substitution into \eqref{dtn1} gives the required result.
\end{proof}

\subsection{Dirichlet-to-Neumann map for $m=2$}\label{sec-m2}
Compared to the case $m\geq 3$ considered in the statement of Theorem \ref{DTNdge2}, the two-dimensional case presents additional difficulties. Recall the lemma from \cite[Sec.~8.14]{McLean}:
\begin{lem} \label{Sh+c=G}
	Let $m=2$, then for any $(g,b)\in H^{\half}(\Gamma)\times \C$ there is a unique solution $(h,c)\in H^{-\half}(\Gamma) \times \C$ to the system
	\[
	\left\{\begin{aligned}
	Sh + c &= g,\\
	\langle h,\one_{\Gamma} \rangle_{H^{-\half}(\Gamma), H^{\half}(\Gamma)} &= b.
	\end{aligned}\right.
	\] 
	In addition, the solution $(h,c)$ is continuously dependent on the right-hand side $(g,b)$, i.e. there is $c_0>0$ such that
	\[
	\|h\|^2_{H^{-\half}(\Gamma)} + |c|^2 \le c_0\big(\|g\|^2_{H^\half(\Gamma)}+|b|^2\big)
	\]
	for all $(g,b)$.
\end{lem}

\begin{thm} \label{Dirchletd1}
	Let $m=2$ and $f\in L^2_\comp(\Omega_+)$, then for any $g\in H^{\half}(\Gamma)$ there is a unique $u\in H^1_\loc(\Omega_+)$ solving the Dirichlet problem with the radiation condition
	\[
	\left\{
	\begin{aligned}
		-\Delta u &= f & \text{ in }&\Omega_+,\\
		\gamma^+_0u &= g & \text{ on }& \Gamma,\\
		u(x) &= O(1) & \text{ for }|x| &\to \infty.
	\end{aligned}\right.
	\]
	This solution has the form
	\begin{equation}
	\label{u-ansatz}
	u = \cG f + \SL h + c, 
\end{equation}
with $h\in H^{-\half} (\Gamma)$ and $c\in\C$ satisfying
\begin{equation}
	\label{eq-shc}
Sh+c=g-\gamma^+_0\cG f,\qquad ( h,\one_\Gamma)_{H^{-\half}(\Gamma),H^\half(\Gamma)}=-\int_{\supp f} f\,\dd x.
\end{equation}
In addition, for any cut-off function $\varphi\in C^\infty_c(\R^2)$ there is a constant $c>0$ 
such that for all $(f,g)$ as above and the respective solution $u$ it  holds that
	\begin{equation}
		\label{unorm00}
	\|\varphi u\|_{H^1(\Omega_+)} \le c \big(\|f\|_{L^2(\Omega_+)} + \|g\|_{H^\half(\Gamma)}\big).
	\end{equation}
\end{thm}
Remark that the radiation condition at infinity in the above theorem differs from the condition at infinity stated in Theorem \ref{boundequ}. This explains appearance of an extra constant term in the representation formula \eqref{u-ansatz}.
\begin{proof}
The uniqueness of $u$ follows by Lemma~\ref{uniq}(a). For the existence we will use the ansatz \eqref{u-ansatz}
with $(h,c) \in H^{-\half}(\Gamma) \times \C$ to be determined. Remark that $-\Delta u = f$ for any choice of $(h,c)$. For any $x\in \Omega_+$ one has by construction
	\begin{align*}
		(\cG f)(x)&=\int_{\supp f} G(x-y)f(y)\dd y=-\frac{1}{2\pi}\int_{\supp f} \log\dfrac{|x-y|}{r} f(y)\dd y,\\[\medskipamount]
	\big(\SL h\big)(x)&=\big( h,G(x-\cdot)\big)_{H^{-\half}(\Gamma),H^\half(\Gamma)}
	=-\frac{1}{2\pi} \Big( h,\gamma^+_0 \log \dfrac{|x-\cdot|}{r}\Big)_{H^{-\half}(\Gamma),H^\half(\Gamma)}.
	\end{align*}
	Representing
	\[
	-\frac{1}{2\pi}\log\dfrac{|x-y|}{r}=-\frac{1}{2\pi}\log\dfrac{|x|}{r}+F_x(y)
	\]
	with 
	\[
	F_x(y):=-\dfrac{1}{2\pi}\log\Big|\dfrac{x}{|x|}-\dfrac{y}{|x|}\Big|\equiv -\dfrac{1}{4\pi}
	\log\Big( 1-2\dfrac{\langle x,y\rangle_{\R^2}}{|x|^2}+\dfrac{|y|^2}{|x|^2}\Big)
	\]
	we obtain
	\begin{align*}
	(\cG f)(x)&=-\dfrac{1}{2\pi}\log \dfrac{|x|}{r}\int_{\supp f} f\,\dd x+\int_{\supp f} F_x(y)f(y)\dd y,\\
	\big(\SL h\big)(x)&=-\dfrac{1}{2\pi}\log \dfrac{|x|}{r}( h,1)_{H^{-\half}(\Gamma),H^\half(\Gamma)}
	+( h,\gamma^+_0 F_x)_{H^{-\half}(\Gamma),H^\half(\Gamma)}.
	\end{align*}
Using Taylor expansion of $\log(1+\cdot)$ we immediately obtain
\[
\|F_x\|_{L^\infty(B)}=O\Big(\dfrac{1}{|x|}\Big) \text{ for }|x|\to\infty,
\]
for any bounded $B\subset\R^2$, and for $B:=\supp f$ this results in
\begin{equation}
	\label{g-asymp}
(\cG f)(x)=-\dfrac{1}{2\pi}\log \dfrac{|x|}{r}\int_{\supp f} f\,\dd x+O\Big(\dfrac{1}{|x|}\Big) \text{ for }|x|\to\infty.
\end{equation}
Now let $B\subset\R^2$ be a bounded open set containing $\Gamma$, then for each $j\in\{1,2\}$ we have additionally
\[
\|\partial_{y_j} F_x\|_{L^\infty(B)}=O\Big(\dfrac{1}{|x|}\Big) \text{ for }|x|\to\infty,
\]
showing
\[
\|F_x\|_{H^1(B)}=O\Big(\dfrac{1}{|x|}\Big) \text{ for }|x|\to\infty,
\]
and the boundedness of the trace map implies
\[
\|\gamma^+_0 F_x\|_{H^\half(\Gamma)}=O\Big(\dfrac{1}{|x|}\Big) \text{ for }|x|\to\infty,
\]
and then
\begin{equation}
	\label{slhx}
\big(\SL h\big)(x)=-\dfrac{1}{2\pi}\log \dfrac{|x|}{r}( h,1)_{H^{-\half}(\Gamma),H^\half(\Gamma)}
+O\Big(\dfrac{1}{|x|}\Big) \text{ for }|x|\to\infty.
\end{equation}
We conclude that the function $u$ given by \eqref{u-ansatz} behaves
as
\[
u(x)=-\dfrac{1}{2\pi}\Bigg[\int_{\supp f} f\,\dd x+ ( h,1)_{H^{-\half}(\Gamma),H^\half(\Gamma)}\Bigg]
\log \dfrac{|x|}{r}+c +O\Big(\dfrac{1}{|x|}\Big) \text{ for }|x|\to\infty,
\]
and it satisfies the radiation condition $u(x)=O(1)$ for large $|x|$ if and only if the coefficient in the square brackets vanishes. Furthermore,
\[
\gamma_0^+u=\gamma^+_0\cG f +S h+c.
\]
Using Lemma \ref{Sh+c=G} we can choose $(h,c)$ such that \eqref{eq-shc} holds,
then
\[
u(x)=c+O\Big(\dfrac{1}{|x|}\Big)=O(1) \text{ for }|x|\to\infty,
\qquad \gamma^+_0 u= g,
\]
i.e. all requirements are satisfied. The norm control \eqref{unorm00} follows from the continuity of the map $(g,f)\mapsto (h,c)$ and the mapping properties of $\cG$ and $\SL$.
\end{proof}
Again, unlike in the case $m\geq 3$, the exterior Neumann problem for the Laplacian is no longer well-posed for arbitrary initial data, but rather requires a certain compatibility relation that ensures the existence of the solution.
\begin{thm} \label{Neumannd1}
Let $m=2$, then for any $f\in L^2_\comp(\Omega_+)$ and $h \in H^{-\half}(\Gamma)$ with
\[
\int_{\Omega_+} f\dd x+\big( h, \one_{\Gamma}\big)_{H^{-\half}(\Gamma), H^{\half}(\Gamma)}=0
\]
 there is a unique solution $u \in H^1_{\loc}(\Omega_+)$ for the Neumann problem 	
	\[
	\left\{
	\begin{aligned}
	-\Delta u &= f  \text{ in }\Omega_+,\\
		\gamma^+_1 u &= h  \text{ on }\Gamma, \\
		u(x) &= o(1) \text{ for } |x| \to \infty.
	\end{aligned}\right.
	\]
\end{thm}
\begin{proof}
The uniqueness follows by Lemma \ref{uniq}(b-ii). Indeed, assume that $u$ solves the Neumann problem with $f=0$ and $h=0$. From the radiation condition $u(x)=o(1)$ for $|x|\rightarrow +\infty$ it follows that $u(x)=O(1)$ for $|x|\rightarrow +\infty$  and by Lemma \ref{uniq}, $u=\operatorname{const}$. However, since $u(x)=o(1)$ for $|x|\rightarrow +\infty$, necessarily, $u=0$. 
To show the existence we use the ansatz
	\[
	u= \cG f-\DL g + \SL h 
	\]
	with $g\in H^{\half}(\Gamma)$ to be determined. Note that $-\Delta u = f$ for any choice of $g$.
	
	Let us show that $u$ satisfies the required bound at infinity.	By using the above assumptions on $f$ and $h$
	in the previously obtained estimates \eqref{g-asymp} and \eqref{slhx} we conclude that
	\[
	\cG f(x)+\SL h(x)=o(1) \text{ for }|x|\to\infty,
	\]
	and it remains to study $\DL g(x)$ for large $x$. For $x\in \Omega_+$ one has
	\begin{align*}
		\big|\DL g(x)\big|&=\bigg|\int_{\Gamma} \partial_{\nu_y} G(x-y) g(y)\,\dd s(y)\bigg|=\frac{1}{2\pi}\bigg|\int_\Gamma \dfrac{\langle \nu_y,x-y\rangle_{\R^2}}{|x-y|^2}g(y)\dd s(y)\bigg|\\
		&\le \dfrac{1}{2\pi}\int_\Gamma \bigg|\dfrac{\langle \nu_y,x-y\rangle_{\R^2}}{|x-y|^2}\bigg|\cdot \big|g(y)\big|\dd s(y)\le \int_\Gamma \dfrac{\big|g(y)\big|}{|x-y|}\dd s(y).
	\end{align*}		
	Let $R>0$ with $\Gamma\subset B_R(0)$, then for $|x|>R$ and $y\in\Gamma$ one has $|x-R|\ge |x|-R$, which yields
	\begin{equation}
		\label{dlg}
	\big|\DL g(x)\big|\le \dfrac{1}{2\pi}\cdot\dfrac{1}{|x|-R}\|g\|_{L^1(\Gamma)}=o(1) \text{ for } |x|\to\infty.
	\end{equation}
	This shows $u(x)=o(1)$ for $|x|\to\infty$ for any choice of $g$.
	
	It remains to check that $g$ can be chosen such that $u$ satisfies the required boundary condition. Note that
	$\gamma^+_1 u=\gamma^+_1\cG f+Rg+\gamma^+_1 \SL h$, and the boundary condition $\gamma^+_1 u=h$ is satisfied if and only if
	\[
	h=\gamma^+_1\cG f+Rg+\gamma^+_1 \SL h \quad \text{ i.e. } \quad Rg=h-\gamma^+_1\cG f-\gamma^+_1 \SL h.
	\]
	As $R$ is Fredholm and self-adjoint with $\ker R=\C \one_\Gamma$ (see Lemma \ref{lemma-RS}), this last equation is solvable with respect to $g$ if any only if
	\begin{equation}
		\label{hsl}
	I:=(\one_\Gamma,h-\gamma^+_1\cG f -\gamma^+_1 \SL h)_{H^\half(\Gamma),H^{-\half}(\Gamma)}=0.
	\end{equation}
	We have $I=I_1-(I_2+I_3)$, where
	\begin{align*}
		I_1&:=(\one_\Gamma,\gamma^+_1\cG f)_{H^\half(\Gamma),H^{-\half}(\Gamma)},\\
		I_2&:=(\one_\Gamma,h)_{H^\half(\Gamma),H^{-\half}(\Gamma)},\\
		I_3&:=(\one_\Gamma,\gamma^+_1 \SL h)_{H^\half(\Gamma),H^{-\half}(\Gamma)},
	\end{align*} 
	For $I_1$ we have, due to the definition of $\DL$,
	\[
	I_1=(\DL \one, f)_{H^1_\loc(\Omega_+),H^{-1}_\comp(\Omega_+)}=0,
	\]
	as $\DL\one_\Gamma=0$ in $\Omega_+$, see \eqref{DLone}. With the help of \eqref{sl-dl} we obtain
\[
I_3=(\gamma^-_0 \DL \one_\Gamma,h)_{H^{\half}(\Gamma),H^{-\half}(\Gamma)}.
\]
By \eqref{DLone} one has $\gamma^-_0 \DL \one_\Gamma= -\one_\Gamma$, hence,
	\[
	I_3=-(\one_\Gamma,h)_{H^{-\half}(\Gamma),H^\half(\Gamma)}=-I_2.
	\]
This shows $I=0$, i.e. the solvability condition \eqref{hsl} is satisfied, which gives a harmonic function $u$ of the above form with $\gamma^+_1 u=h$. 
\end{proof}

To prove the result that follows, it will be convenient to introduce the following decomposition of the space $H^{\half}(\Gamma)$:
		\begin{align}
			\label{eq:decomp_direct_sum}
			H^{ \half}(\Gamma)&=\mathbb{C}\one_{\Gamma}\overset{.}{+}H^{ \half}_0(\Gamma),\\
			\nonumber
			H^\half_0(\Gamma)&:=\Big\{ g\in H^\half(\Gamma):\, (\one_\Gamma,g)_{H^{-\half}(\Gamma),H^\half(\Gamma)}=0\Big\}. 
		\end{align}
This decomposition can be obtained by introducing the $L^2-$orthogonal projector $\Pi_1$ onto the space $\C\one_{\Gamma}$: 
\begin{align*}
	\Pi_1: \, H^{\half}(\Gamma)\ni g\mapsto g:=\frac{(g, \one_{\Gamma})_{L^2(\Gamma)}}{\|\one_{\Gamma}\|_{L^2(\Gamma)}^2}\one_{\Gamma}=\frac{1}{|\Gamma|}\bigg(\int_{\Gamma}g\dd x\bigg)\one_{\Gamma}\in H^{\half}(\Gamma).
\end{align*} 
The projector on the subspace $H^{\half}_0(\Gamma)$ with respect to the direct sum \eqref{eq:decomp_direct_sum} is
given by $\Pi_{0}=\operatorname{Id}-\Pi_1$, and thus the corresponding direct sum becomes orthogonal with respect to the $L^2(\Gamma)$-scalar product. Remark that $\Pi_1$ and $\Pi_0$ are continuous in $H^{\half}(\Gamma)$ and consider their
adjoints (which are also projectors):
\[
			\Pi_1^*, \, \Pi_0^*: \, H^{-\half}(\Gamma)\rightarrow H^{-\half}(\Gamma).
\]
The usual computation for the adjoints, see e.g. \cite[Theorem 2.13]{McLean}, gives
\begin{equation}
				\label{eq:h0m12}
\begin{aligned}
			\ran \Pi^*_0&=	\Big\{ h\in H^{-\half}(\Gamma):\ 
			\big( h, \one_{\Gamma}\big)_{H^{-\half}(\Gamma), H^{\half}(\Gamma)}=0\Big\}=:H^{-\half}_0(\Gamma),\\
			\ran \Pi_1^*&=\C\one_{\Gamma},
\end{aligned} 
\end{equation}
which gives the direct sum decomposition
\[
	H^{- \half}(\Gamma)=\C\one_{\Gamma}\overset{.}{+}H^{-\half}_0(\Gamma)
\]
with $\Pi_1^*$, respectively, $\Pi_0^*$ being the projector on the first, respectively, the second component,
and one obtains the following assertion:
		\begin{lem} The map
			\label{lem:equiv_norm}
		\begin{align*}
			\varphi\mapsto |\varphi|_{H^{-\half}(\Gamma)}:= \|\Pi_0^*\varphi\|_{H^{-\half}(\Gamma)}+\|\Pi_1^*\varphi\|_{H^{-\half}(\Gamma)}
		\end{align*}
		defines an equivalent norm on $H^{-\half}(\Gamma)$. 
	\end{lem}
\begin{proof}
Both projectors $\Pi_0^*$ and $\Pi_1^*$ are continuous, thus $|\varphi|_{H^{-\half}(\Gamma)}\leq C\|\varphi\|_{H^{-\half}(\Gamma)}$. On the other hand,
		\begin{align*}
			\|\varphi\|_{H^{-\half}(\Gamma)}&=\sup_{g\in H^{\half}(\Gamma): \, g\neq 0}\frac{( \varphi, g)_{H^{-\half}(\Gamma), H^{\half}(\Gamma)}}{\|g\|_{H^{\half}(\Gamma)}}
			=\sup_{g\in H^{\half}(\Gamma): \, g\neq 0}\frac{( \Pi_1^*\varphi+\Pi_0^*\varphi, g)_{H^{-\half}(\Gamma), H^{\half}(\Gamma)}}{\|g\|_{H^{\half}(\Gamma)}}\\
			&\leq \sup_{g\in H^{\half}(\Gamma): \, g\neq 0}\frac{\left(\|\Pi_1^*\varphi\|_{H^{-\half}(\Gamma)}+\|\Pi_0^*\varphi\|_{H^{-\half}(\Gamma)}\right)\|g\|_{H^{\half}(\Gamma)}}{\|g\|_{H^{\half}(\Gamma)}}=|\varphi|_{H^{-\half}(\Gamma)}.\qedhere
		\end{align*}
\end{proof}

\begin{cor} \label{DTNd=1}
Let $m=2$. The Dirichlet-to-Neumann operator $\dtn:H^\half(\Gamma)\to H^{-\half}(\Gamma)$
defined by 
\[
\dtn:\  g \mapsto \gamma^+_1 u_g,
\]
where $u_g\in H^1_\loc(\Omega_+)$ is the unique solution of the Dirichlet problem
\[
\left\{
\begin{aligned}
	-\Delta u &= 0  \text{ in }\Omega_+,\\
	\gamma^+_0u &= g  \text{ on } \Gamma,\\
	u(x) &= O(1) \text{ for }|x| \to \infty,
\end{aligned}\right.
\]
is bounded, non-positive and Fredholm of index $0$, with $\ker\dtn =\C \one_\Gamma$ and $\ran \dtn = H^{-\half}_0(\Gamma)$. 
In addition, it is coercive on $H^{\half}_0(\Gamma)$, 
i.e. there is a constant $c>0$ such that for any  $g\in H^{\half}(\Gamma)$ 
it holds that
\[
 - (\dtn g, \overline{g})_{H^{-\half}(\Gamma), H^{\half}(\Gamma)} \ge c \|\Pi_0 g\|^2_{H^\half (\Gamma)}. 
\]	
\end{cor}

\begin{proof}
Step 1. The boundedness of $\dtn$ follows directly from Theorem \ref{Dirchletd1}. 

Step 2: Let us show $\ker \dtn=\C\one_{\Gamma}$.
For $g:=\one_\Gamma$ one clearly has $u_g=\one_{\Omega_+}$
and $\dtn g=\gamma^+_1 u_g=0$, i.e. $\one_\Gamma\subset\ker \dtn$. On the other hand, let $g \in \ker \dtn$, then  $u_g$ satisfies
\[
	\Delta u_g = 0, \quad \gamma^+_1 u_g=0,\quad u(x)=O(1) \text{ for } |x|\to\infty,
\]
so $u$ is constant in $\Omega_+$ by Lemma \ref{uniq}(c-ii), and then $g$ is a multiple of $\one_\Gamma$.

Step 3: We show that $\dtn$ is a self-adjoint operator.  Let $g, h\in H^{\half}(\Gamma)$, and let $u_g, u_h$ resp. be the solutions of the Dirichlet problem of Theorem \ref{Dirchletd1} with $f=0$. Then, by Lemma \ref{uvskal}, 
\begin{align*}
	(\mathcal{C}g, \overline{h})_{H^{-\half}(\Gamma), H^{\half}(\Gamma)}&=-\int_{\Omega_+}\langle \nabla u_g, \nabla u_h\rangle_{\mathbb{C}^m}\dd x=-\overline{\int_{\Omega_+}\langle \nabla u_h, \nabla u_g\rangle_{\mathbb{C}^m}}\dd x\\
	&=\overline{(\mathcal{C}h, \overline{g})}_{H^{-\half}(\Gamma), H^{\half}(\Gamma)},
\end{align*}
which gives the sought conclusion.

Step 4: We show that $\ran \dtn=H^{-\half}_0(\Gamma)$. Due to the self-adjointness of $\dtn$ the subspace $\ran \dtn$
is contained in the annihilator of $\ker \dtn$, see e.g. \cite[Lemma 2.10  and p.~23]{McLean}.
The annihilator of $\ker\dtn\equiv\C\one_\Gamma$ is exactly $H^{-\half}_0(\Gamma)$, which gives the inclusion
$\ran \dtn\subset H^{-\half}_0(\Gamma)$. 
On the other hand, by Theorem \ref{Neumannd1} for any $h\in H^{-\half}_0(\Gamma)$ there is a function $v\in H^1_\loc(\Omega_+)$ with
\[
\Delta v=0,\quad \gamma^+_1 v=h,\quad  v(x)=o(1)=O(1) \text{ for } |x|\to\infty,
\]
so $v=u_g$ for $g:=\gamma^+_0 v$, and $h=\dtn g$. This shows $H^{-\half}_0(\Gamma)\subset \ran\dtn$.
Finally we obtain the sought equality, which then implies $\operatorname{codim}\ran\dtn=1$.

Step 5: Fredholm index. The above discussion shows that $\dtn$ is Fredholm with  $\operatorname{codim}\ran\dtn=1=\dim\ker\dtn$, thus $\dtn$ has zero index.

Step 6: Coercivity on the space $H^{\half}_0(\Gamma)$. 
For any $g\in H^\half(\Gamma)$ one has, due to Lemma \ref{uvskal},
\begin{equation}
	\label{nonpos00}
	-(\dtn g,\overline{g}\,)_{H^{-\half}(\Gamma),H^\half(\Gamma)}=-(\gamma^+_1 u_g,\overline{\gamma^+_0 u_g}\,)_{H^{-\half}(\Gamma),H^\half(\Gamma)}=\int_{\Omega_+}|\nabla u_g|^2\dd x\ge 0,
\end{equation}
which shows the non-positivity of $\dtn$. 
Further remark that by the above consideration the restricted operator
\[
\dtn_0:\ H^\half_0(\Gamma)\ni g_0\mapsto \dtn g_0\in H^{-\half}_0(\Gamma)
\]
is bounded and bijective, therefore, it has a bounded inverse. In particular, for any $h_0\in H^{-\half}_0(\Gamma)$
we have
\begin{equation}
	\label{ch00}
\begin{aligned}	
\int_{\Omega_+}	|\nabla u_{\dtn_0^{-1}h_0}|^2\dd x &=-\big(h, \overline{\dtn^{-1}_0 h_0} \big)_{H^{-\half}(\Gamma), H^{\half}(\Gamma)}\\
&\leq  \|\dtn^{-1}_0\|_{H^{-\half}_0(\Gamma) \to H^{\half}_0(\Gamma)} \|h_0\|^2_{H^{-\half}(\Gamma)}. 
\end{aligned}
\end{equation}
By Lemma  \ref{lem:equiv_norm} we can find some $B>0$ such that $|h|_{H^{-\half}(\Gamma)}\le B\|h\|_{H^{-\half}(\Gamma)}$
for all $h\in H^{-\half}(\Gamma)$.

Let $g_0 \in H^{\frac{1}{2}}_0(\Gamma)$, then
\begin{align*}
\|g_0\|_{H^{\frac{1}{2}}(\Gamma)}&=\sup\limits_{h\in H^{-\frac{1}{2}}(\Gamma),\, h\ne 0} \dfrac{\big| (h,\Bar g_0 )_{H^{-\half }(\Gamma), H^{\half }(\Gamma)}\big|}{\|h\|_{H^{-\half}(\Gamma)}}\leq B\sup\limits_{h\in H^{-\frac{1}{2}}(\Gamma),\, h\ne 0} \dfrac{\big| (h,\Bar g )_{H^{-\half }(\Gamma), H^{\half }(\Gamma)}\big|}{|h|_{H^{-\half}(\Gamma)}}\\
&= B\sup\limits_{h\in H^{-\frac{1}{2}}(\Gamma),\, h\ne 0} \dfrac{\big| (\Pi_{0}^*h+\Pi_1^* h,\Bar g_0)_{H^{-\half }(\Gamma), H^{\half }(\Gamma)}\big|}{\|\Pi_0^*h\|_{H^{-\half}(\Gamma)}+\|\Pi_1^*h\|_{H^{-\half}(\Gamma)}}=I
\end{align*}
Remark that for any $h\in H^{-\half}(\Gamma)$, 
\begin{align*}
	(\Pi_1^*h, \Bar g_0)_{H^{-\half }(\Gamma), H^{\half }(\Gamma)}=(h, \overline{\Pi_1 g_0})_{H^{-\half }(\Gamma), H^{\half }(\Gamma)}=0,
\end{align*}
since $g_0\in H^{-\half}_0(\Gamma)=\operatorname{Ker}\Pi_1$, 
and the preceding computation can be continued as:
\begin{align*}
I&=B\sup\limits_{h\in H^{-\frac{1}{2}}(\Gamma),\, h\ne 0} \dfrac{\big| (\Pi_{0}^*h,\Bar g_0)_{H^{-\half }(\Gamma), H^{\half }(\Gamma)}\big|}{\|\Pi_0^*h\|_{H^{-\half}(\Gamma)}+\|\Pi_1^*h\|_{H^{-\half}(\Gamma)}}\\
&=B\sup\limits_{h\in H^{-\frac{1}{2}}(\Gamma),\, \Pi_0^*h\ne 0} \dfrac{\big| (\Pi_0^*h,\Bar g )_{H^{-\half }(\Gamma), H^{\half }(\Gamma)}\big|}{\|\Pi_0^*h\|_{H^{-\half}(\Gamma)}+\|\Pi_1^*h\|_{H^{-\half}(\Gamma)}}\\
&\leq B\sup\limits_{h\in H^{-\frac{1}{2}}(\Gamma),\, \Pi_0^*h\ne 0} \dfrac{\big| (\Pi_0^*h,\Bar g )_{H^{-\half }(\Gamma), H^{\half }(\Gamma)}\big|}{\|\Pi_0^*h\|_{H^{-\half}(\Gamma)}}.
\end{align*}
Since $\ran \Pi_0^*=H^{-\half}_0(\Gamma)$, the reparametrization $h_0:=\Pi_0^* h$ yields
\begin{align*}
	\|g_0\|_{H^{\frac{1}{2}}(\Gamma)}&\leq B\sup\limits_{h_0\in H^{-\half}_0(\Gamma),\, h_0\ne 0} \dfrac{\big| (h_0,\Bar g_0)_{H^{-\half }(\Gamma), H^{\half }(\Gamma)}\big|}{\|h_0\|_{H^{-\half}(\Gamma)}}.
\end{align*}
By applying Lemma \ref{uvskal} on the right-hand side we arrive at
\begin{align*}
	\|g_0\|_{H^{\frac{1}{2}}(\Gamma)}^2 &\leq B \sup\limits_{h_0\in H^{-\half}_0(\Gamma),\, h_0\ne 0} \dfrac{\bigg|\displaystyle \int_{\Omega_+}\langle \nabla u_{\dtn_0^{-1}h_0},\nabla u_{g_0}\rangle_{\C^m}\dd x	\bigg|^2	
	}{\|h\|_{H^{-\half}(\Gamma)}^2}\\
	& \leq B \sup\limits_{h_0\in H^{-\half}_0(\Gamma),\,h_0\ne 0} \dfrac{\displaystyle \int_{\Omega_+}|\nabla u_{\dtn_0^{-1}h_0}|^2\dd x \int_{\Omega_+} |\nabla u_{g_0}|^2\dd x}{\|h\|_{H^{-\half }(\Gamma)}^2}\\
\text{use \eqref{ch00}: }	&\leq B \sup\limits_{h_0\in H^{-\half}_0(\Gamma),\,h_0\ne 0} \dfrac{\displaystyle \|\dtn^{-1}_0\|_{H^{-\half}_0(\Gamma) \to H^{\half}_0(\Gamma)} \|h_0\|^2_{H^{-\half}(\Gamma)} \int_{\Omega_+} |\nabla u_{g_0}|^2\dd x}{\|h\|_{H^{-\half }(\Gamma)}^{2}}\\
	&\le  B\|\dtn^{-1}_0\|_{H^{-\half}_0(\Gamma) \to H^{\half}_0(\Gamma)}  \int_{\Omega_+} |\nabla u_{g_0}|^2\dd x,
\end{align*}
i.e.
\[
\int_{\Omega_+} |\nabla u_{g_0}|^2\dd x\ge c \|g_0\|_{H^{\half}(\Gamma)}^2,
\quad
c:=\dfrac{1}{B\|\dtn^{-1}_0\|_{H^{-\half}_0(\Gamma) \to H^{\half}_0(\Gamma)}}.
\]
The substitution into \eqref{nonpos00} gives
\begin{equation}
	\label{cg0}
-(\dtn g_0,\overline{g_0}\,)_{H^{-\half}(\Gamma),H^\half(\Gamma)}\ge c \|g_0\|^2_{H^\half(\Gamma)}
\text{ for all } g_0\in H^\half_0(\Gamma).
\end{equation}

Using the self-adjointness of $\dtn$ and the identity $\dtn \Pi_1=0$, for any $g\in H^\half(\Gamma)$
we obtain:
\begin{align*}
- (\dtn g, \overline{g})_{H^{-\half}(\Gamma), H^{\half}(\Gamma)}&= - (\dtn (\Pi_0g+\Pi_1 g), \overline{\Pi_0g+\Pi_1 g})_{H^{-\half}(\Gamma), H^{\half}(\Gamma)}\\
&=-(\dtn \Pi_0g, \overline{\Pi_0 g})\stackrel{\eqref{cg0}}{\ge}
c \|\Pi_0 g\|^2_{H^\half(\Gamma)}.\qedhere
\end{align*}

\end{proof}

 \section{Hybrid transmission problem}
\label{sec:4}
\subsection{Well-posedness}
\label{sec:41}
Now we give a rigorous formulation of the transmission problem described in the introduction.

\begin{prob}\label{prob41}
Let $\alpha_1\in \C\setminus\{0\}$ and $\alpha_0\in L^\infty(\Gamma)$ be given.	
For $f_\cT\in L^2(\cT)$,  $f_\Omega\in L^2_\comp(\Omega)$ and $c\in\C$, find 
the solutions $(u_\cT,u_\Omega)\in H^1(\cT)\times H^1_\loc(\Omega)$ of the transmission problem
	\begin{equation}
	\label{transm00}
	\left\{\begin{aligned}
		\Delta_\cT u_\cT&=f_\cT \text{ on } \cT,\\
		\Delta u_\Omega&=f_\Omega \text{ on } \Omega,\\
		\gamma^\Omega_0 u_\Omega&=\gamma^\cT_0 u_\cT \text{ on } \Gamma,\\
		\gamma^\Omega_1 u_\Omega-\alpha_1\gamma^\cT_1 u_\cT&=\alpha_0 \gamma^\Omega_0 u_\Omega \text{ on }\Gamma,\\
		u_\Omega(x)&=O(|x|^{2-m}) \text{ for } |x|\to\infty,\\
		u_\cT(o)&=c.	
	\end{aligned}
	\right.	
\end{equation}
\end{prob}

Let us start with preliminary remarks on the associated homogeneous problem.

\begin{lem}\label{lem42}
Let $\alpha_1\in \C\setminus\{0\}$ and $\alpha_0\in L^\infty(\Gamma)$. Define the operator
\begin{equation}
	\label{moper}
	M:\ H^{\half}(\Gamma)\ni g\ \mapsto\   -\dtn g+\alpha_1\D g +\alpha_0 g \in H^{-\half}(\Gamma),
\end{equation}
then $\dim\ker M$ coincides with the dimension of the solution space of the homogeneous problem
	\begin{equation}
		\label{transm00hh}
		\left\{\begin{aligned}
			\Delta_\cT w_\cT&=0 \text{ on } \cT,\\
			\Delta w_\Omega&=0 \text{ on } \Omega,\\
			\gamma^\Omega_0 w_\Omega&=\gamma^\cT_0 w_\cT \text{ on } \Gamma,\\
			\gamma^\Omega_1 w_\Omega-\alpha_1\gamma^\cT_1 w_\cT&=\alpha_0 \gamma^\Omega_0 w_\Omega \text{ on }\Gamma,\\
			w_\Omega(x)&=O(|x|^{2-m}) \text{ for } |x|\to\infty,\\
			w_\cT(o)&=0.
		\end{aligned}
		\right.	
	\end{equation}
\end{lem}	
\begin{proof}
	Let $(w_\T,w_\Omega)$ be a solution of \eqref{transm00hh} and $g:=\gamma^\cT_0 w_\T\equiv \gamma^\Omega_0 w_\Omega$,
	then $w_\cT$ and $w_\Omega$ are solutions of
	\begin{equation}
				\left\{\begin{aligned}
			\Delta_\cT w_\cT&=0 \text{ on } \cT,\\
			\gamma^\cT_0 w_\cT&=g \text{ on } \Gamma,\\
			w_\cT(o)&=0,
		\end{aligned}
		\right.	
		\qquad
		\left\{\begin{aligned}
			\Delta w_\Omega&=0 \text{ on } \Omega,\\
			\gamma^\Omega_0 w_\Omega&=g \text{ on } \Gamma,\\
			w_\Omega(x)&=O(|x|^{2-m}) \text{ for } |x|\to\infty,
		\end{aligned}
		\right.	
		\label{w-homog}
	\end{equation}
	therefore, $\gamma^\Omega_1 w_\Omega=\dtn g$ and $\gamma^\cT_1 w_\cT=\D g$.
	The fourth condition in \eqref{transm00hh} is satisfied if and only if
$\dtn g-\alpha_1\D g=\alpha_0 g$, i.e. $g\in\ker M$. To arrive at the desired conclusion, it remains to note that by Lemma~\ref{DPLEwell} and Theorem~\ref{thm21}
	the map $g\mapsto (w_\cT,w_\Omega)$ is one-to-one.  
\end{proof}
Our principal result is as follows:
\begin{thm}\label{thm31}
Let $\alpha_1\in \C\setminus\{0\}$ and $\alpha_0\in L^\infty(\Gamma)$ be such that the homogeneous problem \eqref{transm00hh}
has the unique solution	$(w_\T,w_\Omega)=(0,0)$.

Then the operator $M$ in \eqref{moper}
is an isomorphism. Moreover, the non-homogeneous problem \eqref{transm00}
has a unique solution $(u_\cT,u_\Omega)\in H^1(\cT)\times H^1_\loc(\Omega)$ 
for any choice of $f_\cT\in L^2(\cT)$,  $f_\Omega\in L^2_\comp(\Omega)$ and $c\in\C$,  and the solution depends continuously on $(f_\cT,f_\Omega,c)$ as follows: for any cut-off function $\varphi\in C^\infty_c(\R^m)$ there is a constant $B>0$ such that
\[
\|u_\cT\|_{H^1(\cT)}+\|\varphi u_\Omega\|_{H^1(\Omega)}\le B\big(\|f_\cT\|_{L^2(\cT)}+\|f_\Omega\|_{L^2(\Omega)} + |c|\big).
\]
\end{thm}	

\begin{proof}
Step 1: We are going to show that the operator $M$ defined by \eqref{moper} is surjective.
The embeddings $\iota_1:\ H^{\half}(\Gamma) \hookrightarrow H^{\sigma d}(\Gamma)$ and $\iota_2:\ H^{-\sigma d}(\Gamma) \hookrightarrow H^{-\half}(\Gamma)$,
are compact, which yields the compactness of
$\iota_2\D\iota_1: H^{\frac{1}{2}}(\Gamma) \to H^{-\half}(\Gamma)$. Similarly, the compactness of
the embeddings
$\iota_3:\ H^{\half}(\Gamma) \hookrightarrow L^2(\Gamma)$, $\iota_4:\ L^2(\Gamma) \hookrightarrow H^{-\half}(\Gamma)$,
and the boundedness of the multiplication operator
\[
T_{\alpha_0}:\,L^2(\Gamma)\ni g\mapsto\alpha_0 g\in L^2(\Gamma)
\]
imply that
$\iota_4 T_{\alpha_0} \iota_3:\ H^\half(\Gamma)\to H^{-\half}(\Gamma)$
is compact. For the above operator $M$ we have the representation
$M=-\dtn+\alpha_1\iota_2\D\iota_1+\iota_4 T_{\alpha_0} \iota_3$,
so $M$ is a compact perturbation of the zero-index Fredholm operator $-\dtn$ (see Theorem \ref{thm22}), and it follows that $M$ is a zero-index Fredholm operator too. The assumption (unique solvability of the homogeneous problem)
and Lemma \ref{lem42} show that  $M$ is injective, so it is also surjective.

Step 2: Uniqueness of solutions. Let $(u_\cT,u_\Omega)$ and $(\Tilde u_\cT,\Tilde u_\Omega)$ be two solutions of the non-homogeneous problem \eqref{transm00}, then the functions  $w_\cT:=\Tilde u_\cT-u_\cT$ and $w_\Omega:=\Tilde u_\Omega-u_\Omega$ solve the homogeneous problem \eqref{transm00hh}. By assumption $(w_\cT,w_\Omega)=(0,0)$, which yields $(u_\cT,u_\Omega)=(\Tilde u_\cT,\Tilde u_\Omega)$.

Step 3: Existence of solutions. Use the ansatz $u_\cT=u+c u_1+u_f$, $u_\Omega=v+v_f$,
where
\begin{itemize}
	\item $u_1\in H^1_0(\cT)$ is an arbitrary function with $u_1(o)=1$ and $\Delta_\cT u_1\in L^2(\cT)$,
	\item $u_f\in \Tilde H^1_0(\cT)$ satisfies $\Delta_\cT u_f=f_\cT-c\Delta_\cT u_1$, which exists by~Lemma \ref{lem24},
	\item $u\in \Tilde H^1(\cT)$ is to be determined,
	\item $v_f\in H^1_\loc(\Omega)$ is the solution of
	\[
	\left\{\begin{aligned}
		- \Delta v_f &=f_\Omega  \text{ in }\Omega_+,\\
		\gamma^\Omega_0 v_f &=0  \text{ on }\Gamma,\\
		v_f(x)&= 	O(|x|^{2-m}) \text{ for } |x|\to \infty,
	\end{aligned}\right.
	\]	 
	which exists by Theorem \ref{thm21},
	\item $v\in H^1_\loc(\Omega)$ is to be determined.
\end{itemize}
Then $(u_{\cT}, u_{\Omega})$ satisfy \eqref{transm00} if and only if the new unknown functions $(u,v)$ satisfy
	\begin{equation}
		\label{transm11}
		\left\{\begin{aligned}
			\Delta_\cT u&=0 \text{ on } \cT,\\
			\Delta v&=0 \text{ on } \Omega,\\
			\gamma^\Omega_0 v&=\gamma^\cT_0 u \text{ on } \Gamma,\\
			\gamma^\Omega_1 v-\alpha_1\gamma^\cT_1 u&=\alpha_0 \gamma^\Omega_0 v +h \text{ on }\Gamma,\\
			v(x)&=O(|x|^{2-m}) \text{ for } |x|\to\infty,\\
			u(o)&=0,		
		\end{aligned}
		\right.	
	\end{equation}
with $h:=- \gamma^\Omega_1 v_f+\alpha_1 \gamma^\cT_1 \big(cu_1 +u_f\big)\in H^{-\half}(\Gamma)$.
Note that for any $g\in H^\half(\Gamma)$ the solutions $(u,v)$ to 
\begin{equation}
	\left\{\begin{aligned}
		\Delta_\cT u &=0 \text{ on } \cT,\\
		\gamma^\cT_0 u&=g \text{ on } \Gamma,\\
		u(o)&=0,
	\end{aligned}
	\right.	
	\qquad
		\left\{\begin{aligned}
		\Delta v&=0 \text{ on } \Omega,\\
		\gamma^\Omega_0 v&=g \text{ on } \Gamma,\\
		v(x)&=O(|x|^{2-m}) \text{ for } |x|\to\infty,
	\end{aligned}
	\right.	
	\label{uv-homog}
\end{equation}
satisfy all conditions in \eqref{transm11} expect the fourth one for the normal derivatives.
In order to fulfill this remaining condition we note that for $(u,v)$ in \eqref{uv-homog}
we have
\[
\gamma^\Omega_1 v= \dtn g,\qquad
\gamma^\cT_1 u=\D g,
\]
and the fourth condition in \eqref{transm11} holds if and only if 
\begin{align}
\label{eq:alphag}
\dtn g-\alpha_1 \D g-\alpha_0 g=h,	
\end{align}
which can be rewritten as $M g=-h$. As $M$ is surjective (as shown in Step 1), this equation
has a solution $g\in H^\half(\Gamma)$. By solving \eqref{uv-homog} for this $g$ we obtain a required solution for \eqref{transm11}.

Step 4: Dependence on the initial data. The continuity of the solution on the initial data easily follows by noting that the transitions
\[
(f_\cT,c)\mapsto u_f,\quad
f_\Omega\mapsto v_f,
\quad
(f_\cT,f_\Omega,c)\mapsto h,
\quad h\mapsto g,
\quad g\mapsto v,
\quad g\mapsto u
\]
are continuous in respective norms (with an additional cut-off on $\Omega$).	
\end{proof}
Specific cases for which Theorem \ref{thm31} is applicable are easily obtained using the sign-definiteness and {coercivity} of $\dtn$ and $\D$:
\begin{cor} \label{cor44}
The assumptions of Theorem \ref{thm31} are satisfied in the following cases:

\begin{itemize}
	\item[(i)] $\Re \alpha_1, \, \Re \alpha_0\geq 0$ with $\Re \alpha_1+\Re \alpha_0>0$ a.e. 
	\item [(ii)] $\Im \alpha_1,\Im \alpha_0\geq 0$, with $\Im \alpha_1+\Im \alpha_0>0$ a.e. 
\end{itemize}
Hence, for each of these cases 	Problem \ref{prob41} is uniquely solvable.
\end{cor}
\begin{proof}
Let $\alpha_1\in \C\setminus\{0\}$ and $\alpha_0\in L^\infty(\Gamma)$.
Let $g\in H^\half(\Gamma)$ with $Mg=0$, then
\begin{align}
\nonumber
0&=(Mg,\Bar g)_{H^{-\half}(\Gamma),H^\half(\Gamma)}\\
\label{eq:dtnid}
&=-(\dtn g,\Bar g)_{H^{-\half}(\Gamma),H^\half(\Gamma)}
+\alpha_1 (\D g,\Bar g)_{H^{-\sigma d}(\Gamma),H^{\sigma d}(\Gamma)}+\int_\Gamma \alpha_0 |g|^2\,\dd s.
\end{align}
We start by arguing that $(i)$ is sufficient. Taking the real part and using the coercivity of $\D$ (Theorem \ref{DTNtree}) and the non-positivity of $\dtn$ (Corollary \ref{DTNd=1} for $m=2$, Theorem \ref{DTNdge2} for $m\geq 3$), we obtain, with some $c>0$,
\begin{align*}
0&=	-(\dtn g,\Bar g)_{H^{-\half}(\Gamma),H^\half(\Gamma)}
+(\Re \alpha_1) (\D g,\Bar g)_{H^{-\sigma d}(\Gamma),H^{\sigma d}(\Gamma)}+\int_\Gamma (\Re \alpha_0) |g|^2\,\dd s\\
&\ge (\Re \alpha_1) (\D g,\Bar g)_{H^{-\sigma d}(\Gamma),H^{\sigma d}(\Gamma)}+\int_\Gamma (\Re \alpha_0) |g|^2\,\dd s,\\
&\ge (\Re \alpha_1)c\|g\|^2_{H^{\sigma d}(\Gamma)}+\int_\Gamma (\Re \alpha_0) |g|^2\,\dd s\\
&\ge (\Re \alpha_1)c\|g\|^2_{L^2(\Gamma)}+\int_\Gamma (\Re \alpha_0) |g|^2\,\dd s\\
&\ge \min\{1,c\} \int_\Gamma (\Re \alpha_1+\Re \alpha_0) |g|^2\,\dd s,
\end{align*}
and the assumptions $(i)$ yields $g=0$ a.e.

To argue that $(ii)$ is sufficient, we take the imaginary part of \eqref{eq:dtnid} and use the coercivity of $\D$. This yields
	\begin{align*}
0&= (\Im \alpha_1) (\D g,\Bar g)_{H^{-\sigma d}(\Gamma),H^{\sigma d}(\Gamma)}+ \int_\Gamma (\Im \alpha_0) |g|^2\,\dd s,
\end{align*}
and we conclude by proceeding almost verbatim like in the previous case. 
\end{proof}

\subsection{Approximations by finite truncations}\label{sec42}

In this subsection we illustrate how the finite-dimensional approximations of $\D$ constructed in Subsection~\ref{sec-dtn} can be employed to approximate the solutions of the transmission problem \eqref{transm00}. For that we assume that the metric tree $\cT$ is geometric from some generation $N_1$ (see Definition \ref{regularity}) and denote
\[
\D_N:=\D P_N,\quad N\in\N.
\]
The result of Corollary~\ref{cor219} and the boundedness of the embeddings
\[
H^\half(\Gamma)\hookrightarrow H^{\sigma' d}(\Gamma),\quad H^{-\sigma d}(\Gamma)\hookrightarrow H^{-\half}(\Gamma)
\]
imply that for large $N$ one has
\begin{equation}
	\label{error01}
\|\D_N-\D\|_{H^{\half}(\Gamma)\to H^{-\half}(\Gamma)}=O(p^{-\rho N}) \text{ with any } \rho\in\Big(0,\frac{1-2\sigma}{2d}\Big).
\end{equation}
Further recall that Theorem~\ref{thm221} provides an efficient way of computing $\D_N$ for large $N$ using $N$-condensations of $\cT$.

In virtue of Theorem \ref{thm31} and Corollary \ref{cor44}, for suitably chosen $\alpha_1\in \C$ and $\alpha_0\in L^{\infty}(\Gamma)$
and any $f\in L^2_\comp(\Omega)$  there is a unique solution $(u_\cT,u_\Omega)\in H^1(\cT)\times H^1_\loc(\Omega)$  of the transmission problem
\begin{equation}
	\label{transm33}
	\left\{\begin{aligned}
		\Delta_\cT u_\cT&=0\text{ on } \cT,\\
		\Delta u_\Omega&=f_\Omega \text{ on } \Omega,\\
\gamma^\Omega_0 u_\Omega&=\gamma^\cT_0 u_\cT \text{ on } \Gamma,\\
		\gamma^\Omega_1 u_\Omega-\alpha_1\gamma^\cT_1 u_\cT&=\alpha_0 \gamma^\Omega_0 u_\Omega \text{ on }\Gamma,\\
		u_\Omega(x)&=O(|x|^{2-m}) \text{ for } |x|\to\infty,\\
		u_\cT(o)&=0.		
	\end{aligned}
	\right.	
\end{equation}
To construct this solution, we proceed like in the proof of Theorem~\ref{thm221} (Step 3). Repeating the corresponding argument and using the same notation we conclude that 
\begin{align*}
	u_{\Omega}=v+v_f, \quad u_{\mathcal{T}}=u,
\end{align*}
where $u, v$ solve \eqref{uv-homog} with $g$ being the unique solution to  \eqref{eq:alphag}, i.e. 
\begin{align*}
	\dtn g-\alpha_1\D g-\alpha_0 g=-\gamma_1^{\Omega}v_f.
\end{align*}
A natural approximation to the above problem consists in replacing $\D$ by $\D_N$ (all the related quantities will be marked by index $N$). 

By combining \eqref{error01} with Theorem \ref{thm31} we conclude that for large $N$ the operator
\begin{align*}
	q&\mapsto -\dtn q+\alpha_1\D_N q+\alpha_0 q\equiv \big(-\dtn q+\alpha_1\D q+\alpha_0  q) +\alpha_1(\D_N-\D)q
\end{align*}
is an isomorphism $H^\half(\Gamma)\to H^{-\half}(\Gamma)$, and applying the Neumann series for its inverse yields that 
there is a unique solution $g_N\in H^\half(\Gamma)$ of
\[
-\dtn g_N+\alpha_1\D_N g_N+\alpha_0 g_N=h,
\]
and one has $\|g_N-g\|_{H^\half(\Gamma)}=O(p^{-\rho N})$. 

By using the continuous dependence of the solutions of \eqref{uv-homog} on $g$ we conclude that
the solutions $v^N$ and $u^N$ of
\begin{equation}
	\left\{\begin{aligned}
		\Delta_\cT u^N &=0 \text{ on } \cT,\\
		\gamma^\cT_0 u^N&=g_N \text{ on } \Gamma,\\
		u^N_\cT(o)&=0,
	\end{aligned}
	\right.	
	\qquad 
		\left\{\begin{aligned}
		\Delta v^N&=0 \text{ on } \Omega,\\
		\gamma^\Omega_0 v^N&=g_N \text{ on } \Gamma,\\
		v^N(x)&=O(|x|^{2-m}) \text{ for } |x|\to\infty,
	\end{aligned}
	\right.	
\end{equation}
satisfy
\[
\|u - u^N\|_{H^1(\cT)}=O(p^{-\rho N}),
\quad 
\|\varphi(v-v^N)\|_{H^1(\Omega)}=O(p^{-\rho N}),
\]
for any cut-off function $\varphi\in C^\infty_c(\R^m)$. Hence, the functions $u^N_{\mathcal{T}}:=u^N$ and $u^N_\Omega:=v_f+v^N$ with large $N$ provide a good approximation to the solution $(u_\cT,u_\Omega)$ in the sense that
\[
\|u^N_\cT - u^N\|_{H^1(\cT)}=O(p^{-\rho N}),
\quad 
\|\varphi(u^N_\Omega-u_\Omega)\|_{H^1(\Omega)}=O(p^{-\rho N})
\]
for any cut-off function $\varphi\in C^\infty_c(\R^m)$.

\begin{rmk}
Consider in greater details the case $\alpha_0\equiv 0$. Then the unique solvability of Problem \ref{prob41} fails if and only if
the homogeneous problem
	\begin{equation}
	\label{transm00hh-mod}
	\left\{\begin{aligned}
		\Delta_\cT w_\cT&=0 \text{ on } \cT,\\
		\Delta w_\Omega&=0 \text{ on } \Omega,\\
		\gamma^\Omega_0 w_\Omega&=\gamma^\cT_0 w_\cT \text{ on } \Gamma,\\
		\gamma^\Omega_1 w_\Omega&=\alpha_1\gamma^\cT_1 w_\cT \text{ on }\Gamma,\\
		w_\Omega(x)&=O(|x|^{2-m}) \text{ for } |x|\to\infty,\\
		w_\cT(o)&=0.
	\end{aligned}
	\right.	
\end{equation}
has non-trivial solutions. Lemma \ref{lem42} implies that the values $\alpha_1$ for which it happens are exactly those
with
\begin{equation}
	\label{kcd0}
\ker(-\dtn+\alpha_1\D)\ne\{0\}.
\end{equation}
Using the non-positivity of $\dtn$ one concludes
that $\Id-\dtn:\,H^\half(\Gamma)\to H^{-\half}(\Gamma)$ is an isomorphism.
Then the factorization
\[
-\dtn+\alpha_1\D=\Big(\Id +(\alpha_1\D-\Id)(\Id-\dtn)^{-1}\Big)(\Id-\dtn)
\]
shows that~\eqref{kcd0} rewrites as
\[\ker \Big(\Id +(\alpha_1\D-\Id)(\Id-\dtn)^{-1}\Big)\ne\{0\}.
\]
Noting that $(\Id-\dtn)^{-1}$ and $\D(\Id-\dtn)^{-1}$ can be viewed as compact operators in $H^{-\half}(\Gamma)$
we conclude (using the analytic Fredholm theorem) that the values of $\alpha_1$ for which
\eqref{transm00hh-mod} has non-trivial solutions form a discrete subset of $\C$ (i.e. without finite accumulation points).
In fact, in view of Corollary \ref{cor44} and using the complex conjugation one easily sees that all these values belong to $(-\infty,0)$.

The problem \eqref{transm00hh-mod} and the respective critical values of $\alpha_1$ represent a natural mixed-dimensional counterpart
of the so-called plasmonic eigenvalue problem in $\R^m$, see e.g. the discussion in \cite{grieser, grieser2}, and it is known that
the plasmonic eigenvalues in the Euclidean case have a finite accumulation point, contrary to what was just observed
for our problem.
\end{rmk}

\end{document}